\documentclass[envcountsame,envcountsect]{svmult}

\usepackage{helvet}         
\usepackage{courier}        
\usepackage{type1cm}        
\usepackage{makeidx}         
\usepackage{graphicx}        
\usepackage{multicol}        
\usepackage[bottom]{footmisc}
\usepackage{bm}
\usepackage{mathtools}
\usepackage[top=4.cm, bottom=4cm, outer=4cm, inner=3.8cm]{geometry}

\makeindex             

\usepackage{latexsym}
\usepackage{dsfont}
\usepackage{bbm}
\usepackage{amssymb}
\usepackage{amsmath}
\usepackage{tikz}
\usetikzlibrary{automata,arrows,positioning,calc}

\newtheorem{Theorem}{Theorem}[section]
\newtheorem{Lemma}[Theorem]{Lemma}

\newtheorem{Prop}[Theorem]{Proposition}
\newtheorem{Rem}[Theorem]{Remark}
\newtheorem{Exa}[Theorem]{Example}

\def\cL{\mathcal{L}}

\def\cS{\mathcal{S}}

\def\cZ{\mathcal{Z}}

\def\bbD{\mathbb{D}} 
\def\Erw{\mathbb{E}}

\def\N{\mathbb{N}}
  
\def\Prob{\mathbb{P}} 

\def\R{\mathbb{R}}      

\def\T{\mathbb{T}}
\def\U{\mathbb{U}}
\def\V{\mathbb{V}}

\def\W{\mathbb{W}}

\def\Z{\mathbb{Z}}

\def\bfE{\textbf{E}}

\def\bM{\textit{\bfseries M}}

\def\bfP{\textbf{P}}

\def\bfE{\text{\rm\bfseries E}}
\def\bfP{\text{\rm\bfseries P}}

\def\eps{\varepsilon}
\def\vph{\varphi}
\def\vth{\vartheta}

\def\1{\vec{1}}
\def\3{{\ss}}

\def\eqdist{\stackrel{d}{=}}

\def\RA{\Rightarrow}

\def\LRA{\Leftrightarrow}
\def\wh{\widehat}

\def\ovl{\overline}

\def\sg{\sigma^{>}}
\def\osg{{\ovl{\sigma}}^{>}}
\def\sgn{\sigma_{n}^{>}}

\def\sln{\sigma_{n}^{<}}
\def\sle{\sigma^{\leqslant}}
\def\slen{\sigma_{n}^{\leqslant}}

\def\tg{\tau^{>}}

\def\Mgn{M_{n}^{>}}

\def\Mln{M_{n}^{<}}

\def\Sln{S_{n}^{<}}

\def\Ug{\U^{>}}

\allowdisplaybreaks[4] 

\begin{document}

\title*{Fluctuation theory for Markov random walks}
\titlerunning{Fluctuation theory for Markov random walks}
\author{Gerold Alsmeyer and Fabian Buckmann}
\institute{Inst.~Math.~Statistics, Department
of Mathematics and Computer Science, University of M\"unster,
Orl\'eans-Ring 10, D-48149 M\"unster, Germany.\at
\email{gerolda@math.uni-muenster.de, f\_buck01@uni-muenster.de}\at
Research partially supported by the Deutsche Forschungsgemeinschaft (SFB 878).}

\maketitle

\abstract{Two fundamental theorems by Spitzer-Erickson and Kesten-Maller on the fluctuation type (positive divergence, negative divergence or oscillation) of a real-valued random walk $(S_{n})_{n\ge 0}$ with iid increments $X_{1},X_{2},\ldots$ and the existence of moments of various related quantities like the first passage into $(x,\infty)$ and the last exit time from $(-\infty,x]$ for arbitrary $x\ge 0$ are studied in the Markov-modulated situation when the $X_{n}$ are governed by a positive recurrent Markov chain $M=(M_{n})_{n\ge 0}$ on a countable state space $\cS$, thus for a Markov random walk $(M_{n},S_{n})_{n\ge 0}$. Our approach is based on the natural strategy to draw on the results in the iid case for the embedded ordinary random walks $(S_{\tau_{n}(i)})_{n\ge 0}$, where $\tau_{1}(i),\tau_{2}(i),\ldots$ denote the successive return times of $M$ to state $i$, and an analysis of the excursions of the walk between these epochs. However, due to these excursions,  generalizations of the afore-mentioned theorems are surprisingly more complicated and require the introduction of various excursion measures so as to characterize the existence of moments of different quantities.}

\bigskip

{\noindent \textbf{AMS 2000 subject classifications:}
60K15 (60J10, 60K05, 60J15, 60G40)}

{\noindent \textbf{Keywords:} Markov random walk, fluctuation theory, discrete Markov chain, fluctuation-type trichotomy, Kesten trichotomy, null-homology, positive divergence, ladder epoch, Spitzer-Erickson theorem, Kesten-Maller theorem, first exit time, last exit time, ladder epoch, renewal theory}

\section{Introduction}

Classical fluctuation theory deals with the \emph{fine structure} (to steal an expression used by Chung in his textbook \cite[Section 8.4]{Chung:01}) of ordinary random walks on $\R$, i.e., partial sums $S_{n}=\sum_{k=1}^{n}X_{k}$ of iid real-valued random variables $X_{1},X_{2},\ldots$ It comprises, amongst others,
\begin{description}[(4)]\itemsep2pt
\item[(1)] the basic trichotomy regarding the almost sure behavior of $S_{n}$ as $n\to\infty$,
\item[(2)] results about the existence of moments for related quantities like $\min_{n\ge 0}S_{n}$ or the number of nonpositive sums $N(0)=\sum_{n\ge 0}\1_{\{S_{n}\le 0\}}$,
\end{description}
A short review of the results relevant for this article will be given in the Section \ref{sec:ORW}.

\vspace{.2cm}
The present work aims at providing results of type (1) and (2) for the more general situation when the increments $X_{1},X_{2},\ldots$ are \emph{modulated} or \emph{driven} by a positive recurrent Markov chain $\bM=(M_{n})_{n\ge 0}$ with countable state space $\cS$. More precisely, the $X_{n}$ are conditionally independent given $\bM$, and
$$ \Prob((X_{1},\ldots,X_{n})\in\cdot\,|M_{0}=i_{0},\ldots,M_{n}=i_{n})\ =\ K_{i_{0}i_{1}}\otimes\ldots\otimes K_{i_{n-1}i_{n}} $$
for all $n\ge 1$, $s_{0},\ldots,s_{n}\in\cS$ and some stochastic kernel $K$ from $\cS^{2}$ to $\R$. Then $(M_{n},S_{n})_{n\ge 0}$, and sometimes also its additive part $(S_{n})_{n\ge 0}$, is called a \emph{Markov random walk (MRW)} or \emph{Markov additive process} and $\bM$ its \emph{driving chain}. Let $\vec{P}=(p_{ij})_{i,j\in\cS}$ denote the transition matrix of $\bM$ and $\pi=(\pi_{i})_{i\in\cS}$ its unique stationary distribution. For any $i\in\cS$, put further $\Prob_{i}:=\Prob(\cdot|M_{0}=i)$ and let $(\tau_{n}(i))_{n\ge 1}$ denote the renewal sequence of successive return epochs to $i$.

\vspace{.1cm}
The well-known fact that, for each $i\in\cS$, $(S_{\tau_{n}(i)})_{n\ge 0}$ with $\tau_{0}(i):=0$ constitutes an ordinary zero-delayed random walk under $\Prob_{i}$ suggests that results of the above kind for $(S_{n})_{n\ge 0}$ should be obtainable by drawing on the known results for these embedded walks. On the other hand, it should also be clear that the excursions between the successive return epochs require additional analysis and may in fact be intriguing and result in surprising effects. For instance, it is possible that $S_{\tau_{n}(i)}\to\infty$ a.s. \emph{for all} $i\in\cS$ while $(S_{n})_{n\ge 0}$ is oscillating (see Example \ref{exa:infinite petal flower chain}). Our main results will actually show that the attempt to simply ``lift'' fluctuation-theoretic results for ordinary random walks to the class of MRW is not at all straightforward and often even fails without proper adjustments. In other words, despite the fact that $(S_{n})_{n\ge 0}$ can be viewed as the (generally infinite) union of the ordinary random walks $(S_{\tau_{n}(i)})_{n\ge 0}$, $i\in\cS$, the way those are intertwined may lead to nontrivial complications that must be taken care of in the analysis.

\vspace{.1cm}
There is an extensive literature on MRW with \emph{discrete} driving chain and finite stationary drift $\Erw_{\pi}X_{1}$, mostly within the framework of Markov renewal theory and focussing on the Markov renewal theorem and results derived therefrom. \c Cinlar \cite{Cinlar:69,Cinlar:75} provides good accounts of the early developments and references, while Asmussen's monography \cite[Ch. XI]{Asmussen:03} and \cite{Alsmeyer:14} may be consulted for more recent treatments of some aspects of the theory including the discrete Markov renewal theorem, the dual process, and Wiener-Hopf factorization, for the latter see also \cite{Asmussen:89,PrabhuTangZhu:91,FuhLai:98}. The ladder variables and the associated ladder chain of a MRW have been studied in \cite{Alsmeyer:00,Alsmeyer:16a}, see also Section \ref{sec:preliminaries} for further information, an arcsine law for the number of positive sums is derived in \cite{AlsBuck:17b}, and the topological recurrence of $(S_{n})_{n\ge 0}$ in the case when $\Erw_{\pi}X_{1}=0$ is shown in \cite{Alsmeyer:01}. For conditional Markov renewal theorems in the case when $\cS$ is countable, we mention an article by Lalley \cite{Lalley:84}. Finally, there are two papers, by Newbould \cite{Newbould:73} and Prabhu et al. \cite{PrabhuTangZhu:91}, of closer relation to the present work by providing the basic trichotomy for MRW, the first one in the case when the driving chain has finite state space, the second one without this restriction and including a discussion of degeneracy. We refer to Section \ref{sec:classification} for further details.

\vspace{.1cm}
Regarding the more general situation when the driving chain has continuous state space and is positive Harris recurrent, we believe that an extension of our results is possible, at least to some extent, but not without considerable additional work. To explain, we note that the natural substitute in our approach for the return times $\tau_{n}(i)$ to some state $i$, which are generally no longer a.s. finite, is a sequence of regeneration epochs $(\tau_{n})_{n\ge 1}$, marked by the successive epochs where the Harris chain returns to some recurrent small set as defined by Meyn and Tweedie \cite[Sect.~5.2]{MeynTweedie:09} and a regeneration occurs in the sense of Athreya  and Ney \cite{Athreya+Ney:78a,Athreya+Ney:78b} or Nummelin \cite{Nummelin:78a}. Unfortunately, the associated embedded random walk $(S_{\tau_{n}})_{n\ge 1}$ has iid increments only in special cases, namely when the bivariate chain $(M_{n},X_{n})_{n\ge 0}$ satisfies a certain Harris-type condition (see \cite{Athreya+et.al.:78a} and also \cite{Ney+Nummelin:86}). In general, however, the increments are 1-dependent and stationary. Therefore, in all places where we have drawn on rather deep fluctuation-theoretic results for ordinary RW due to Spitzer \cite{Spitzer:56}, Erickson \cite{Erickson:73}, and Kesten and Maller \cite{KesMal:96} (see Section \ref{sec:ORW} for further information) extensions to the case of stationary, 1-dependent increments are needed. Since this cannot be done shortly, we have restricted this work to the case when the driving chain has countable state space.

\vspace{.1cm}
Recently, results of similar type as in this article have been derived by the first author with Iksanov and Meiners \cite{AlsIksMei:15} for another generalization of ordinary RW, called perturbed random walks (PRW), which have interesting connections with perpetuities, the Bernoulli sieve, regenerative and shot-noise processes (see their Sections 1 and 3 for further information). We refer to Section \ref{sec:PRW} for a more detailed discussion of how the results here relate to those in \cite{AlsIksMei:15}.

\vspace{.1cm}
Further organization: Three examples where MRW play a prominent role are described in some detail in Section \ref{sec:examples}, in particular random difference equations in Markovian environment which have been a major motivation for this work. Section \ref{sec:ORW} provides a short survey of the relevant fluctuation-theoretic results for ordinary RW,  followed by a preliminary section giving some basic facts about MRW with discrete driving chain. A classification of MRW as to their fluctuation type, a short discussion of null-homologous MRW, which are the counterpart to ordinary RW with zero increments, and an extension of Kesten's trichotomy to MRW (Theorem \ref{thm:Kesten trichotomy MRW}) form the content of Section \ref{sec:classification}. All main results are presented in Section \ref{sec:main results}. For their proofs, provided in Section \ref{sec:proof main results}, various quantities, defined as functions of $i\in\cS$, must be considered and shown to share certain properties for all $i$. Such solidarity results will be collected in Section \ref{sec:solidarity}. Section \ref{sec:moment sg(x)} is devoted to a further discussion of the level $x$ first passage times $\inf\{n\ge 1:S_{n}>x\}$ the behavior of which is more difficult to describe than for ordinary RW. A short discussion of the strong law of large numbers can be found in Section \ref{sec:SLLN}, while Section \ref{sec:counterexamples} collects some counterexamples that, as a supplement to our main results, will show that various equivalences given in the theorems by Spitzer-Erickson and Kesten-Maller for ordinary RW do not carry over to MRW. Finally, Section \ref{sec:PRW} provides the above-mentioned discussion of PRW, followed by an Appendix containing some auxiliary lemmata of purely technical nature and a Glossary providing a comprehensive list of the most important notation used in this article.

\section{Examples}\label{sec:examples}

Markov modulation forms a common tool in Applied Probability, for instance in queuing, risk or reliability theory, to provide models of greater variability by allowing input parameters (like interarrival or service times, claim sizes, lifetime distributions or the type of agents in the model) to depend on the state of an underlying Markov chain. Examples of this kind may be found e.g. in the monographs by Asmussen \cite[Chs. VI and VIII]{Asmussen:00}, \cite[Ch.~11]{Asmussen:03}, Prabhu \cite[Part III]{Prabhu:98} or Limnios and Opri\c san \cite{LimOpr:01}. In the following, we give three examples where the occurrence of MRW may be less known and begin with one related to random difference equations (iterations of random affine linear maps) that has been an area of very active research over the last twenty years, see the recent monograph by Buraczewski et al. \cite{BurDamMik:16} and also \cite{AlsLoe:13}.

\subsection{Random difference equations (perpetuities) in Markovian environment}\label{subsec:perpetuities}

A main motivation for the present work originated from the question of convergence of iterations of affine linear maps $\Psi_{n}(x)=A_{n}x+B_{n}$, $n=1,2,\ldots$, in the situation when $(A_{n},B_{n})_{n\ge 1}$ of $\R^{2}$-valued random vectors is modulated by a positive recurrent Markov chain $(M_{n})_{n\ge 0}$ with countable state space $\cS$, see \cite{AlsBuck:17a}. This means that, conditioned upon $M_{0}=i_{0},M_{1}=i_{1},\ldots$ for arbitrary $i_{0},i_{1},\ldots\in\cS$,
\begin{itemize}
\item $(A_{1},B_{1}),\,(A_{2},B_{2}),...$ are conditionally independent,
\item the conditional law of $(A_{n},B_{n})$ depends only on $(i_{n-1},i_{n})$ and is temporally homogeneous, i.e. 
$$ \Prob((A_{n},B_{n})\in\cdot|M_{n-1}=i_{n-1},\,M_{n}=i_{n})\ =\ K_{i_{n-1}i_{n}} $$ 
for a stochastic kernel $K$ from $\cS^{2}$ to $\R^{2}$ and all $n\ge 1$.
\end{itemize}
The goal is to find necessary and sufficient conditions for the convergence in law of the iterated function system
\begin{equation}\label{ifsf}
R_{n}\ :=\ \Psi_{n}(R_{n-1})\ =\ \Psi_{n}\circ\ldots\circ\Psi_{1}(R_{0}),\quad n=1,2\ldots,
\end{equation} 
also called forward iterations, as well as such conditions for the convergence, almost surely or in law, of the corresponding backward iterations
\begin{equation}\label{backward sequence}
Z_{n}\:=\ \Psi_{1}\circ\ldots\circ\Psi_{n}(R_{0})\ =\ \Pi_{n}R_{0}\ +\ \sum_{k=1}^{n}\Pi_{k-1}B_k,\quad n=1,2,\ldots
\end{equation}
where
$$ \Pi_{0}\ :=\ 1\quad\text{and}\quad\Pi_{n}\ :=\ A_{1} A_{2} \cdot\ldots\cdot A_{n},\quad n=1,2,\ldots $$
and $R_{0}$ is a real number (this can be generalized but is ignored here for simplicity). Note that, if $R_{0}=0$ and the a.s. limit of $Z_{n}$ exists, then it equals $Z_{\infty}=\sum_{n\ge 1}\Pi_{n-1}B_{n}$, called perpetuity.

\vspace{.1cm}
In the case of iid $(A_{n},B_{n})$, the afore-mentioned stability questions were finally settled by Goldie and Maller \cite[Thms. 2.1 and 3.1]{GolMal:00}, based on earlier work by Vervaat \cite{Vervaat:79} and Grincevi\v{c}ius \cite{Grincev:81,Grincev:82}. One of their central results reads as follows: If
$$ \Prob(A_{1}=0)\,=\,0\quad\text{and}\quad\Prob(B_{1}=0)\,<\,1, $$
then $Z_{n}$ converges a.s. to $Z_{\infty}$ (regardless of the initial value $Z_{0}=R_{0}$) iff
\begin{equation}\label{eq:GM-condition}
\Pi_{n}\,\to\,0\quad\text{a.s.}\quad\text{and}\quad\Erw J(\log^{+}|B|)\,<\,\infty,
\end{equation}
where, for $x>0$,
\begin{align*}
J(0)\,:=\,1\quad\text{and}\quad J(x)\,:=\,
\begin{cases}
\displaystyle{\frac{x}{\Erw((\log^{-}|A|)\wedge x)}},&\text{if }\Prob(\log |A|<0)>0,\\[2mm]
\hfill x,&\text{otherwise}.
\end{cases}
\end{align*}
Since $Z_{n}$ and $R_{n}$ have obviously the same distribution for each $n$, we also infer the convergence in law of $R_{n}$ to $Z_{\infty}$. Let us note here in passing that this equality in law does no longer generally hold in Markovian environment. According to Theorem 2.1 in \cite{AlsIksMei:15}, Condition \eqref{eq:GM-condition} is equivalent to the negative divergence of the PRW
$$ W_{n}\ =\ \log|\Pi_{n-1}|+\log|B_{n}|,\quad n\ge 0, $$
which means that $W_{n}\to-\infty$ a.s. (see also Section \ref{sec:PRW}). This equivalence in turn is obtained by drawing on a fluctuation-theoretic result, stated as Theorem \ref{thm:Spitzer-Erickson} in the next section, due to Spitzer \cite{Spitzer:56} and Erickson \cite{Erickson:73}. 

\vspace{.1cm}
In view of this, it can be expected that an extension of the Goldie-Maller theorem to the Markov-modulated setup requires for an extension of the Spitzer-Erickson result to MRW, this being so because $(\log|\Pi_{n}|)_{n\ge 0}$ now forms a MRW with driving chain $(M_{n})_{n\ge 0}$. The latter extension is indeed obtained here as Theorem \ref{thm:Spitzer-Erickson MRW} and utilized for the proofs of some of the results in \cite{AlsBuck:17a} (stated there in Section 3). For all further details including relevant references, the interested reader is referred to that paper which actually gives a complete description of necessary and sufficient conditions for the convergence of forward and backward iterations. Unlike the iid case, this requires the distinction of various additional subregimes related to the lattice-type of $(M_{n},\log|\Pi_{n}|)_{n\ge 0}$.

\subsection{Branching random walk in random environment}\label{subsec:BRW}

Recently, Mallein and Mi\l os \cite{MalMil:15} have studied the maximal displacement in a supercritical branching random walk (BRW) in random environment, the latter given by a sequence $\cL=(\cL_{n})_{n\ge 0}$ of iid point process laws on $\R$. The basic assumptions about these laws are
\begin{align}
\Prob(\cL_{0}(\text{number of points}=0))\ &=\ 0,\label{eq:MaMi-1}\\
\Prob(\cL_{0}(\text{number of points}>1))\ &>\ 0.\label{eq:MaMi-2}
\end{align}
In other words, if $\cZ=\sum_{k=1}^{N}\delta_{Z_{k}}$ has law $\cL_{0}$, then $N\ge 1$ a.s. and $\Prob(N>1)>0$. Then a BRW in random environment $\cL$ originates from an initial particle $\varnothing$ sitting at the origin at time 0. At time 1, the particle dies while giving birth to a random number of children with random positions relative to their mother in accordance with the law $\cL_{0}$. At time 2, these offspring particles die while independently giving birth to a random number of children with random positions relative to their own position in accordance with the law $\cL_{1}$. Generally, at time $n$ all particles born at time $n-1$ die and independently give birth to a random number of children with random positions relative to their own position in accordance with the law $\cL_{n-1}$. Due to the assumptions on $\cL$, the genealogy of this process is described by an a.s. non-extinctive, thus supercritical Galton-Watson tree $\T$, say, in iid random environment. For $v\in\T$, let $S(v)$ denote the position of the particle $v$. It is obtained by summing the relative displacements of all particles along the unique path from the ancestor $\varnothing$ (root of $\T$) to $v$. Then the maximal displacement of the particles born at time $n$ is defined by
$$ \Lambda_{n}\ :=\ \max_{v\in\T,|v|=n}S(v), $$
where $|v|$ is the generation to which $v$ belongs.

\vspace{.1cm}
Next, let $\vph_{n}:\R_{\geqslant}\to\R\cup\{\infty\}$ for $n\in\N_{0}$ denote the log-Laplace transform of the random point-process law $\cL_{n}$, thus
$$ \vph_{n}(\theta)\ :=\ \log\Erw_{\cL}\sum_{x\in\cZ_{n}}e^{\theta x} $$
where $\cZ_{n}$ is a point process with law $\cL_{n}$ under the conditional measure $\Prob_{\cL}$. Note that $\vph_{n}(\theta)>-\infty$ is guaranteed by \eqref{eq:MaMi-1} and that $\vph_{0},\vph_{1},\ldots$ are iid random functions. As in \cite{MalMil:15}, we further assume that $\vph_{n}(\theta)^{-}$, the negative part of $\vph_{n}(\theta)$, has finite mean for all $\theta\ge 0$. Then $\ovl{\vph}:\R_{\geqslant}\to\R\cup\{\infty\}$, $\theta\mapsto\Erw\vph_{0}(\theta)$ is well-defined, and further a smooth and convex function on $\textit{int}\,(\bbD)$, the interior of $\bbD=\{\theta:\ovl{\vph}(\theta)<\infty\}$, if this interval is nonempty which will also be assumed hereafter along with the existence of $\theta^{*}\in\textit{int}\,(\bbD)$ such that
$$ \theta^{*}\ovl{\vph}'(\theta^{*})-\ovl{\vph}(\theta^{*})\ =\ 0. $$
It then follows that $\ovl{\vph}'(\theta^{*})=\Erw\vph_{0}'(\theta^{*})=\nu$, where
$$ \nu\ :=\ \inf_{\theta>0}\frac{\ovl{\vph}(\theta)}{\theta}. $$
Under these conditions and some further technical ones omitted here (see \cite[(1.6)--(1.8)]{MalMil:15}), the main result of Mallein and Mi\l os asserts that, for some $\beta^{*}>0$,
\begin{align*}
\lim_{n\to\infty}\Prob_{\cL}\left(\Lambda_{n}-\frac{1}{\theta^{*}}\sum_{k=0}^{n-1}\vph_{k}(\theta^{*})\ge -\beta\log n\right)\ =\ 
\begin{cases}
1,&\text{if }\beta>\beta^{*},\\ 0,&\text{if }\beta<\beta^{*}
\end{cases}\quad\text{in $\Prob$-probability},
\end{align*}
see \cite[Thm.~1.1]{MalMil:15}), also for the definition of the threshold $\beta^{*}$. This result extends earlier ones by Addario-Berry and Reed \cite{AddarioReed:09}, Hu and Shi \cite{HuShi:09} and A{\"{\i}}d{\'e}kon \cite{Aidekon:13} for the case of constant environment.

\vspace{.1cm}
An essential tool in the study of the extremal behavior of BRW is the so-called many-to-one lemma. In the present context of iid random environment, it renders the connection with a MRW. Certain fluctuation-theoretic properties of this walk, see \cite[Section 3]{MalMil:15}, are then used in the analysis of $M_{n}$. To see the connection, define the stochastic kernel $K$ by
$$ K(\ell,(-\infty,t])\ :=\ \Erw\left(\sum_{x\in\cZ_{0}}\1_{(-\infty,t]}(x)e^{\theta^{*}x-\vph_{n}(\theta^{*})}\Bigg|\cL_{0}=\ell\right) $$
where $\cZ_{0}$ is a point process with law $\ell$ given $\cL_{0}=\ell,\cL_{1},\cL_{2},\ldots$
Let $(X_{n})_{n\ge 1}$ be a sequence of real-valued random variables which, conditioned upon $\cL$, are conditionally independent and such that the conditional law of $X_{n}$ equals $K(\cL_{n},\cdot)$. Putting $S_{n}:=\sum_{k=1}^{n}X_{k}$ and $W_{n}:=\theta^{*}S_{n}-\sum_{k=1}^{n}\vph_{k}(\theta^{*})$ for $n\ge 1$ and $S_{0}=W_{0}:=0$, it follows that $(\cL_{n},S_{n})_{n\ge 0}$ as well as $(\cL_{n},W_{n})_{n\ge 0}$ constitute MRW with driving chain $\cL$. They fit into the framework of this article if further the $\cL_{n}$ take values in a countable set. The many-to-one lemma for the given model now states the following, see \cite[Lemma 2.1]{MalMil:15}: For any $n\in\N$ and any nonnegative measurable function $f$ on $\R^{n}$, we have
\begin{align*}
\Erw_{\cL}\left(\sum_{|v|=n}f(S(v^{1}),\ldots,S(v^{n}))\right)\ =\ \Erw_{\cL}\left(e^{-W_{n}}f(S_{1},\ldots,S_{n})\right)
\end{align*}
or, equivalently,
\begin{align*}
\Erw_{\cL}\left(\sum_{|v|=n}f(W(v^{1}),\ldots,W(v^{n}))\right)\ =\ \Erw_{\cL}\left(e^{-W_{n}}f(W_{1},\ldots,W_{n})\right),
\end{align*}
where $v^{0}=\varnothing\to v^{1}\to\ldots\to v^{n-1}\to v^{n}=v$ denotes the unique path from the root to $v$ and $W(u):=\theta^{*}S(u)-\sum_{k=1}^{|u|}\vph_{k}(\theta^{*})$ for $u\in\T$. This means that, in quenched regime (under $\Prob_{\cL}$), the average over all walks along the rays in $\T$ up to level $n$ is described by a MRW. The relevance of this walk for the asymptotic behavior of $\Lambda_{n}$ stems from the fact that, roughly speaking, it carries the main information of the extremal paths in the BRW.

\subsection{Superpositions of renewal processes}\label{subsec:superpositions}

Consider a single-server queue with $p\ge 2$ time-slotted input channels which are described by independent, integrable and discrete renewal sequences $(S_{k,n})_{n\ge 0}$, $k=1,\ldots,p$ taking values in $\N_{0}$. Thus,
$$ S_{k,n}\ =\ S_{k,0}+X_{k,1}+\ldots+X_{k,n} $$ 
for each $n\ge 1$ and $k=1,\ldots,p$, where
\begin{itemize}
\item $(X_{k,n})_{n\ge 1}$, $k=1,\ldots,p$, are independent sequences of iid positive integer-valued random variables with $\mu_{k}:=\Erw X_{k,1}<\infty$ for each $k$,
\item $S_{1,0},\ldots,S_{p,0}$ take values in $\N_{0}$ and are mutually independent as well as independent of all $X_{k,n}$.
\end{itemize}
Further defining the residual waiting time sequences $(R_{k,n})_{n\ge 0}$ by
$$ R_{k,n}\ :=\ \min\{S_{k,l}-n:S_{k,l}\ge n,\,l\ge 0\} $$
for $k=1,\ldots,p$, it is a well-known fact from renewal theory that these sequences constitute independent discrete Markov chains on $\N_{0}$ with stationary distributions $\lambda_{k,\bullet}:=(\lambda_{k,n})_{n\ge 0}$, where
$$ \lambda_{k,n}\ =\ \mu_{k}^{-1}\,\Prob(X_{k,1}>n), $$
see, for example, \cite[Thm.~6.2 on p.~62]{Gut:09}. As a consequence, the $p$-variate sequence
$$ R_{n}\ :=\ (R_{1,n},\ldots,R_{p,n}),\quad n\ge 0 $$
forms a positive discrete Markov chain on $\N_{0}^{p}$, its stationary distribution being the product of the $\lambda_{k,\bullet}$, i.e. $\lambda=\lambda_{1,\bullet}\otimes\ldots\otimes\lambda_{p,\bullet}$. Let $(M_{n})_{n\ge 0}$ be the associated hit chain of the set 
$$ \cS\ :=\ \big\{(n_{1},\ldots,n_{p})\in\N_{0}^{p}:n_{k}=0\text{ for some }1\le k\le p\big\}, $$
thus $M_{n}:=R_{\tau_{n}}$ for $n\ge 0$ with $\tau_{0}:=0$ and $\tau_{n}:=\inf\{k>\tau_{n-1}:R_{k}\in\cS\}$ for $n\ge 1$. Its stationary distribution equals $\pi:=\lambda(\cdot\cap\cS)/\lambda(\cS)$. Note that an arrival in one of the channels occurs iff the Markov chain $(R_{n})_{n\ge 0}$ hits the set $\cS$. Now the aggregated arrival process $(S_{n})_{n\ge 0}$, where simultaneous arrivals from different channels are viewed as one arrival epoch, is obtained as the superposition of the $(S_{k,n})_{n\ge 0}$ (that is, the order statistics of the sample $\{S_{k,n}:k=1,\ldots,p,\,n\ge 0\}$ with multiple points aggregated into one) and can be shown to constitute a Markov random walk with driving chain $(M_{n})_{n\ge 0}$. For more details see \cite{Alsmeyer:96} where this has been verified for the (more difficult) case when the renewal sequences are continuous and, as a consequence, the driving chain has continuous state space. We also refer to this article and \cite{LamLehoczky:91} for further results on $(S_{n})_{n\ge 0}$ by making use of Markov renewal theory.

\section{Ordinary random walks}\label{sec:ORW}

It is a well-known fact that any ordinary random walk $(S_{n})_{n\ge 0}$ in $\R$ whose iid increments $X_{1},X_{2},\ldots$ are not degenerate at 0 exhibits exactly one of the following behaviors:
\begin{description}[\sf (Osc)]\itemsep1pt
\item[\sf(PD)] \emph{Positive divergence:} $\lim_{n\to\infty}\,S_{n}=\infty$ a.s.
\item[\sf (ND)] \emph{Negative divergence:} $\lim_{n\to\infty}\,S_{n}=-\infty$ a.s.
\item[\sf(Osc)] \emph{Oscillation:} $\liminf_{n\to\infty}\,S_{n}=-\infty$ and $\limsup_{n\to\infty}\,S_{n}=\infty$ a.s.
\end{description}
Let $X$ denote a generic copy of the $X_{n}$ and suppose that $\Erw X$ exists, i.e. $\Erw X^{-}<\infty$ or $\Erw X^{+}<\infty$. Then
\begin{align*}
\textsf{(PD)}\ &\LRA\ \Erw X^{-}<\Erw X^{+}\le\infty,\\
\textsf{(ND)}\ &\LRA\ \Erw X^{+}<\Erw X^{-}\le\infty,\\
\textsf{(Osc)}\ &\LRA\ \Erw X^{-}=\Erw X^{+}<\infty,\text{ i.e. }\Erw X=0,
\end{align*}
see Chung \cite[Thm.~8.3.4]{Chung:01}. Moreover, if $\Erw|X|=\infty$, thus $\Erw X^{-}\vee\Erw X^{+}=\infty$, then Kesten \cite{Kesten:70} showed that the above trichotomy even holds with $n^{-1}S_{n}$ instead of $S_{n}$. We refer to this result as \emph{Kesten's trichotomy}.

\vspace{.1cm}
Suppose $S_{0}=0$ hereafter. Many authors have dealt with the problem of relating the three fundamental types with other quantities of interest related to $(S_{n})_{n\ge 0}$, notably the level $x$ first passage times
\begin{align*}
\sg(x)\ &:=\ \inf\{n\ge 1:S_{n}>x\},\\
\sle(-x)\ &:=\ \inf\{n\ge 1:S_{n}\le -x\}
\end{align*}
for $x\in\R_{\geqslant}$, the level $x$ last exit time
$$ \rho(x)\,:=\,\sup\{n\ge 0:S_{n}\le x\} $$
for $x\in\R$, the hitting of the minimum epoch
$$ \sigma_{\min}\ := \inf\left\{n\ge 1:S_{n}=\inf_{k\ge 1}S_{k}\right\}, $$
the renewal counting process
$$ N(x)\,:=\,\sum_{n\ge 1}\1_{\{S_{n}\le x\}} $$
for $x\in\R$, and the weighted renewal (or occupation) measures
$$ \sum_{n\ge 1}n^{\alpha-1}\Prob(S_{n}\le x)\ =\ \Erw\left(\sum_{n\ge 1}n^{\alpha-1}\1_{\{S_{n}\le x\}}\right) $$
for $x\in\R$ and $\alpha\ge 0$. The following theorem, treating the positive divergent case, is a blend of results due to Spitzer \cite[Thm.~4.1 on p.~331]{Spitzer:56} and Erickson \cite{Erickson:73} (see also \cite[Thm.~2.1 with $\alpha=0$]{KesMal:96}). Let $A(x):=\Erw(X^{+}\wedge x)-\Erw(X^{-}\wedge x)$ for $x\in\R_{\geqslant}$, and
\begin{align*}
J(0)\,:=\,1\quad\text{and}\quad J(x)\,:=\,
\begin{cases}
\displaystyle{\frac{x}{\Erw(X^{+}\wedge x)}},&\text{if }\Prob(X>0)>0\\[2mm]
\hfill x,&\text{otherwise}
\end{cases},\quad x>0.
\end{align*}

\begin{Theorem}\label{thm:Spitzer-Erickson}
The following assertions are equivalent:
\begin{description}[(d)]\itemsep3pt
\item[(a)] $\lim_{n\to\infty}S_{n}=\infty$ a.s.
\item[(b)] $A(x)>0$ for all sufficiently large $x$ and $\Erw J(X^{-})<\infty$.
\item[(c)] $\sum_{n\ge 1}n^{-1}\Prob(S_{n}\le x)<\infty$ for some/all $x\in\R_{\geqslant}$.
\item[(d)] $\Erw\sg(x)<\infty$ for some/all $x\in\R_{\geqslant}$.
\end{description}
\end{Theorem}

\begin{Rem}\label{rem:Erickson's extension}\rm
Erickson \cite[Corollary 1]{Erickson:73} actually showed that, if $\Erw|X|=\infty$, then the positive divergence of $(S_{n})_{n\ge 0}$ is already entailed by $\Erw J(X^{-})<\infty$ alone, i.e., one may dispense with $A(x)>0$ for all sufficiently large $x$. On the other hand, if $\Erw|X|<\infty$, then $J(x)=O(x)$ as $x\to\infty$ and therefore $\Erw J(X^{-})<\infty$. Consequently, (b) reduces to $0<\Erw X=\lim_{x\to\infty}A(x)$ in this case.
\end{Rem}

Due to the fluctuation-type trichotomy of nontrivial RW, each of $\Prob(\sigma_{\min}<\infty)=1$, $\Prob(\rho(x)<\infty)=1$ for some/all $x\in\R_{\geqslant}$, $\Prob(N(x)<\infty)=1$ for some/all $x\in\R_{\geqslant}$, and $\Prob(\sigma^\leqslant(-x)=\infty)>0$ for some/all $x\in\R_{\geqslant}$ is immediately seen to be also equivalent to the positive divergence of $(S_{n})_{n\ge 0}$.

\vspace{.1cm}
Plainly, the corresponding result for negative divergent $(S_{n})_{n\ge 0}$ can be read off directly from the previous one by replacing $(S_{n})_{n\ge 0}$ with $(-S_{n})_{n\ge 0}$, and the result for oscillating $(S_{n})_{n\ge 0}$ then follows by contraposition. Let us also note that the function $J(x)$ in (b) may be replaced with $xA(x)^{-1}$, see \cite[proof of Lemma 3.1]{KesMal:96}, for $A(x)$ and $\Erw(X^{+}\wedge x)$ are of the same order of magnitude as $x\to\infty$. The series occurring in (d) is the harmonic renewal function of $(S_{n})_{n\ge 0}$ and its evaluation at 0 is often called Spitzer series.

\vspace{.1cm}
Kesten and Maller \cite[Thm.~2.1 and p. 27]{KesMal:96} generalized the above result, stated as Theorem \ref{thm:Kesten-Maller} below, by establishing equivalent conditions for the finiteness of moments of $\sg(x),\rho(x),N(x),\min_{n\ge 0}S_{n}$, and $\sigma_{\min}$. In the case when $0<\Erw X\le\Erw|X|<\infty$, this had already been done by other authors, most notably Gut \cite{Gut:74} and Janson \cite{Janson:86}, the last one by providing a comprehensive result with a total of ten equivalences and economical proofs. For a good survey of the relevant literature, the reader is referred to Gut's monography \cite[Ch.~3]{Gut:09}.

\begin{Theorem}\label{thm:Kesten-Maller}
Given a positive divergent RW $(S_{n})_{n\ge 0}$, the following conditions are equivalent for $\alpha>0$:
\begin{description}[(d)]\itemsep3pt
\item[(a)] $\Erw J(X^{-})^{1+\alpha}<\infty$.
\item[(b)] $\Erw\rho(x)^{\alpha}<\infty$ for some/all $x\in\R_{\geqslant}$.
\item[(c)] $\Erw\sigma_{\min}^{\alpha}<\infty$.
\item[(d)] $\Erw\sle(-x)^{\alpha}\1_{\{\sle(-x)<\infty\}}<\infty$ for some/all $x\in\R_{\geqslant}$.
\item[(e)] $\sum_{n\ge 1}n^{\alpha-1}\Prob(S_{n}\le x)<\infty$ for some/all $x\in\R_{\geqslant}$.
\item[(f)] $\Erw N(x)^{\alpha}<\infty$ for some/all $x\in\R_{\geqslant}$.
\item[(g)] $\Erw\sg(x)^{1+\alpha}<\infty$ for some/all $x\in\R_{\geqslant}$.
\end{description}
\end{Theorem}

Putting $S_{n}^{*}:=\max_{0\le k\le n}S_{k}$ for $n\in\N_{0}$, Kesten and Maller \cite[Thm. 2.2]{KesMal:96} further showed that 
\begin{equation}\label{eq:harmonic series asymptotics}
\sum_{n\ge 1}\frac{1}{n}\,\Prob\left(S_{n}^{*}\le x\right)\,\asymp\,\sum_{n\ge 1}\frac{1}{n}\,\Prob(S_{n}\le x)\,\asymp\,\log J(x)
\end{equation}
holds under the conditions of Theorem \ref{thm:Spitzer-Erickson}, and
\begin{equation}\label{eq:harmonic series asymptotics, alpha>0}
\Erw \rho(x)^{\alpha}\,\asymp\,\Erw N(x)^{\alpha}\,\asymp\,\sum_{n\ge 1}n^{\alpha-1}\,\Prob\left(S_{n}^{*}\le x\right)\,\asymp\,\sum_{n\ge 1}n^{\alpha-1}\,\Prob(S_{n}\le x)\,\asymp\,J(x)^{\alpha}
\end{equation}
for $\alpha>0$ under the conditions of Theorem \ref{thm:Kesten-Maller}. Here $f(x)\asymp g(x)$ means that $f$ and $g$ are of the same oder of magnitude as $x\to\infty$, viz.
\begin{equation}\label{eq:def asymp}
0\ <\ \liminf_{x\to\infty}\frac{f(x)}{g(x)}\quad\text{and}\quad\limsup_{x\to\infty}\frac{f(x)}{g(x)}\ <\ \infty.
\end{equation}
We also write $f(x)\gtrsim g(x)$ and $f(x)\lesssim g(x)$ as shorthand for the left and the right relation in \eqref{eq:def asymp}, respectively. Finally, given two expressions $A,B$ (series or integrals), $A\asymp B$, $A\lesssim B$ and $A\gtrsim B$ will be used if, for some $c\in\R_{>}$, $c^{-1}B\le A\le cB$, $A\le cB$ and $A\ge cB$, respectively. Note that $\{\sg(x)>n\}=\{S_{n}^{*}\le x\}$ for all $x\in\R_{\geqslant}$ and $n\in\N_{0}$ implies
\begin{align*}
\sum_{n\ge 1}n^{\alpha-1}\,\Prob\left(S_{n}^{*}\le x\right)\,\asymp\,
\begin{cases}
\Erw\log\sg(x),&\text{if }\alpha=0,\\
\Erw\sg(x)^{\alpha},&\text{if }\alpha>0.
\end{cases}
\end{align*}


\vspace{.1cm}
Regarding the finiteness of $\Erw|\min_{n\ge 0}S_{n}|^{\alpha}$ for $\alpha>0$, an equivalent condition of similar type as Theorem \ref{thm:Kesten-Maller}(a) has also been given in \cite[Prop.~4.1]{KesMal:96}, but is actually stronger unless $\Erw X^{+}$ is finite and thus $\Erw X>0$.

\begin{Theorem}\label{thm:Kesten-Maller min}
Given a positive divergent RW $(S_{n})_{n\ge 0}$, the following conditions are equivalent for $\alpha>0$:
\begin{description}[(d)]\itemsep3pt
\item[(a)] $\Erw(X^{-})^{\alpha}J(X^{-})<\infty$.
\item[(b)] $\Erw|\min_{n\ge 0}S_{n}|^{\alpha}<\infty$.
\item[(c)] $\Erw|S_{\sle(-x)}|^{\alpha}\1_{\{\sle(-x)<\infty\}}<\infty$ for some/all $x\in\R_{\geqslant}$.
\end{description}
\end{Theorem}


\section{Preliminaries}\label{sec:preliminaries}

We return to the Markov-modulated setup and suppose for the rest of this article a MRW $(M_{n},S_{n})_{n\ge 0}$ be given which has positive recurrent discrete driving chain with stationary distribution $\pi=(\pi_{i})_{i\in\cS}$. For any distribution $\lambda=(\lambda_{i})_{i\in\cS}$ on $\cS$, put $\Prob_{\lambda}:=\sum_{i\in\cS}\lambda_{i}\Prob_{i}$. Since $\pi_{i}=(\Erw_{i}\tau(i))^{-1}>0$ for all $i\in\cS$, it follows that ``$\,\Prob_{\pi}$-a.s.'' means the same as ``$\,\Prob_{i}$-a.s. for all $i\in\cS\,$'', and this will henceforth be abbreviated by ``a.s.''

\vspace{.1cm}
Due to its particular Markovian structure, $(M_{n},X_{n})_{n\ge 1}$ forms a stationary sequence under $\Prob_{\pi}$, and for any measurable function $f:\cS\times\R\to\R$ satisfying
$$ \Erw_{\pi}f^{-}(M_{1},X_{1})<\infty\quad\text{or}\quad\Erw_{\pi}f^{+}(M_{1},X_{1})<\infty $$
we have the useful occupation measure formula
\begin{equation}\label{eq:occ measure formula}
\Erw_{\pi}f(M_{1},X_{1})\ =\ \frac{1}{\Erw_{i}\tau(i)}\Erw_{i}\left(\sum_{n=1}^{\tau(i)}f(M_{n},X_{n})\right),
\end{equation}
valid for any $i\in\cS$. As a particular consequence,
\begin{equation}\label{eq:stat mean drift}
\Erw_{\pi}X_{1}\ =\ \frac{1}{\Erw_{i}\tau(i)}\Erw_{i}S_{\tau(i)}\ =\ \pi_{i}\,\Erw_{i}S_{\tau(i)}
\end{equation}
whenever $\Erw_{\pi}X_{1}$ exists, i.e. $(\Erw_{\pi}X_{1}^{-})\wedge(\Erw_{\pi}X_{1}^{+})<\infty$. Note, however, that the right-hand side in \eqref{eq:stat mean drift} may be finite for all $i\in\cS$ even if $\Erw_{\pi}X_{1}$ does not exist. In other words, $\Prob_{\pi}$-integrability of the $S_{\tau(i)}$ does not generally imply the very same for $X_{1}$ (see Example \ref{exa:SLLN}).

\vspace{.2cm}
Ladder variables are well known to form an important tool in the analysis of random walks. We define
\begin{align*}
\sg\ =\ \sg_{1}\,:=\,\inf\{n\ge 1:S_{n}>0\},\quad\sle\ =\sle_{1}\,:=\,\inf\{n\ge 1:S_{n}\le 0\},
\end{align*}
thus $\sg=\sg(0)$ and $\sle=\sle(0)$, and then recursively for $n\ge 2$
\begin{align*}
\sgn\,:=\,\inf\{k>\sg_{n-1}:S_{k}>S_{\sg_{n-1}}\},\quad\slen\,:=\,\inf\{k>\sle_{n-1}:S_{k}\le S_{\sg_{n-1}}\}.
\end{align*}
We further put $\Mgn:=M_{\sgn}\1_{\{\sgn<\infty\}}+M_{\sigma_{*}^{>}}\1_{\{\sgn=\infty\}}$, where $\sigma_{*}^{>}:=\sup\{\sgn:\sgn<\infty\}$.

\vspace{.1cm}
The \emph{dual of} $(M_{n},S_{n})_{n\ge 0}$, denoted $({}^{\#}M_{n},{}^{\#}S_{n})_{n\ge 0}$ herafter, is again a MRW and its driving chain $({}^{\#}M_{n})_{n\ge 0}$ the time reversal of $(M_{n})_{n\ge 0}$ under $\Prob_{\pi}$ with transition matrix 
$$ {}^{\#}\vec{P}\ =\  \left(\frac{\pi_{j}p_{ji}}{\pi_{i}}\right)_{i,j\in\cS}. $$
Moreover,
$$ \Prob({}^{\#}X_{1}\in\cdot\,|{}^{\#}M_{0}=i,{}^{\#}M_{1}=j)\ =\ K_{ji} $$
for all $i,j\in\cS$. More generally, the duality relation
\begin{align}
\begin{split}\label{eq:duality relation}
\pi_{i}\,\Prob_{i}&((M_{k},X_{k})_{0\le k\le n}\in\cdot\,,M_{n}=j)\\
&=\ \pi_{j}\,\Prob_{j}(({}^{\#}M_{n-k},{}^{\#}X_{n-k})_{0\le k\le n}\in\cdot\,,{}^{\#}M_{n}=i)
\end{split}
\end{align}
holds true for all $i,j\in\cS$ and $n\in\N_{0}$. Considering a doubly infinite extension $(M_{n},X_{n})_{n\in\Z}$ of the stationary unilateral chain $(M_{n},X_{n})_{n\ge 1}$ under $\Prob_{\pi}$ and putting $S_{0}:=0$ and $S_{n}:=S_{n-1}+X_{n}$ for $n\ne 0$, thus
\begin{equation*}
S_{n}\,:=\,
\begin{cases}
\hfill X_{1}+\ldots+X_{n},&\text{if }n\ge 1,\\
\hfill 0,&\text{if }n=0,\\
-X_{0}-\ldots-X_{n+1},&\text{if }n\le -1,
\end{cases}
\end{equation*}
one can easily verify that the laws of the dual and of $(M_{-n},-S_{-n})_{n\ge 0}$ are the same under $\Prob_{\pi}$, thus $({}^{\#}M_{n},{}^{\#}X_{n})_{n\ge 1}$ has the law of $(M_{-n},X_{-n+1})_{n\ge 1}$ under $\Prob_{\pi}$. Let ${}^{\#}\sgn,{}^{\#}\slen$ denote the counterparts of $\sgn,\slen$ for $({}^{\#}M_{n},{}^{\#}S_{n})_{n\ge 0}$. For later use, we quote from \cite{Alsmeyer:16a} the following result about the positive recurrence of the \emph{ladder chain} $(\Mgn)_{n\ge 0}$.

\begin{Prop}\label{prop:ladder chain}
Given a MRW $(M_{n},S_{n})_{n\ge 0}$ with positive recurrent driving chain, suppose that the dual sequence $({}^{\#}S_{n})_{n\ge 0}$ is positive divergent, that is ${}^{\#}S_{n}\to\infty$ a.s. Then the ladder chain $(\Mgn)_{n\ge 0}$ has the unique stationary law
\begin{equation*}
\pi_{i}^{>}\ =\ \frac{\pi_{i}\Prob_{i}({}^{\#}\sle=\infty)}{\Prob_{\pi}({}^{\#}\sle=\infty)},\quad i\in\cS,
\end{equation*}
and is positive recurrent on $\cS^{>}=\{i:\pi_{i}^{>}>0\}=\{i:\Prob_{\pi}({}^{\#}\sle=\infty)>0\}$.
\end{Prop}

Later, we will also need the strictly ascending ladder epochs of $(S_{\tau_{n}(i)})_{n\ge 0}$, denoted $(\tg_{n}(i))_{n\ge 0}$ and defined by $\tg_{0}(i):=0$, $\tg_{n}(i):=\tau_{\zeta_{n}}(i)$ for $n\ge 1$, where 
\begin{align}
&\hspace{1cm}\zeta_{1}\ =\ \zeta_{1}(i)\ :=\ \inf\{k\ge 1:S_{\tau_{k}(i)}>0\}\label{eq:LI S_tau(i)}
\shortintertext{and}
&\zeta_{n}\ =\ \zeta_{n}(i)\ :=\ \inf\{k>\zeta_{n-1}(i):S_{\tau_{k}(i)}>S_{\tg_{n-1}(i)}\}\nonumber
\end{align}
for $n\ge 2$. Put $\tg(i):=\tg_{1}(i)$. Finally, let
\begin{align}
\begin{split}\label{eq:def of nu(x)}
\nu(x)\ &=\ \nu(i,x)\ :=\ \inf\{n\ge 1:S_{\tau_{n}(i)}>x\},\\
\nu^{>}(x)\ &=\ \nu^{>}(i,x)\ :=\ \inf\{n\ge 1:S_{\tg_{n}(i)}>x\}
\end{split}
\end{align}
for $x\in\R_{\geqslant}$ and notice that $\nu(0)=\zeta_{1}$, $\nu(x)=\zeta_{\nu^{>}(x)}$, $\sg(x)\le\tau_{\nu(x)}=\tg_{\nu^{>}(x)}$ a.s.

\section{Null-homology and classification of Markov random walks}\label{sec:classification}

A natural starting point for our analysis is to provide the basic classification of MRW as to their divergence type. Unlike ordinary RW, which can be bounded only if their increments are a.s. 0 (trivial case), there is an infinite subclass of MRW exhibiting this kind of behavior. After a short discussion in this section, those will therefore be ruled out from the subsequent analysis. For the remaining ones, the same trichotomy as for ordinary RW holds true (Prop.~\ref{prop:classify nontrivial}). This was shown by Newbould \cite[Thm.~1]{Newbould:73} for finite $\cS$ and by Prabhu et al. \cite[Thm.~7]{PrabhuTangZhu:91} in the general situation.

\vspace{.1cm}
Following Lalley \cite{Lalley:86}, we call a MRW $(M_{n},S_{n})_{n\ge 0}$ \emph{null-homologous} hereafter if there exists a function $g:\cS\to\R$ such that
\begin{equation}\label{NH1}
X_{n}\ =\ g(M_{n})-g(M_{n-1})\quad\Prob_{\pi}\text{-a.s.}
\end{equation}
or, equivalently,
\begin{equation}\label{NH2}
S_{n}\ =\ g(M_{n})-g(M_{0})\quad\Prob_{\pi}\text{-a.s.}
\end{equation}
for all $n\in\N$. Otherwise, the MRW is called \emph{nontrivial}. Obviously, $(S_{n})_{n\ge 0}$ fluctuates within a compact subset of $\R$ if $g$ is a bounded function. The following two lemmata show that all embedded RW $(S_{\tau_{n}(i)})_{n\ge 0}$, $i\in\cS$, of a null-homologous MRW are trivial and that the finiteness of one of $\liminf_{n\to\infty}S_{n}$ or $\limsup_{n\to\infty}S_{n}$ entails null-homology. A complete classification of null-homologous MRW as to their asymptotic behavior is then very easy and stated without proof as Prop.~\ref{prop:classify null}. The result is preceded by the following two lemmata the first of which was again obtained by Prabhu et al. \cite[Thm.~2]{PrabhuTangZhu:91} regarding the equivalence of (b) and (c).

\begin{Lemma}\label{lem:zero increments}
Given a MRW $(M_{n},S_{n})_{n\ge 0}$, the following assertions are equivalent:
\begin{description}[xxx]
\item[(a)] $(M_{n},S_{n})_{n\ge 0}$ is null-homologous.
\item[(b)] $(S_{\tau_{n}(i)})_{n\ge 0}$ has zero increments for some $i\in\cS$.
\item[(c)] $(S_{\tau_{n}(i)})_{n\ge 0}$ has zero increments for all $i\in\cS$.
\end{description}
\end{Lemma}

\begin{proof}
Only ``(c)$\RA$(a)'' remains to be proved, for ``(a)$\RA$(c)'' is trival. Let $\psi_{ij}$ be the characteristic function of $S_{\tau(j)}$ under $\Prob_{i}$ for $i,j\in\cS$. Since $(S_{\tau_{n}(i)})_{n\ge 0}$ has zero increments, we easily find that
$$ \psi_{ii_{1}}(t)\psi_{i_{1}i_{2}}(t)\cdot\ldots\cdot\psi_{i_{n-1}i_{n}}(t)\psi_{i_{n}i}(t)\ =\ 1 $$
for all $t\in\R$, $n\ge 2$ and $i_{1},\ldots,i_{n}\in\cS$, in particular
$|\psi_{ij}|=|\psi_{ij}\psi_{ji}|\equiv 1$ for all $i,j\in\cS$. Therefore, with $\vec{i}:=\sqrt{-1}$, we have $\psi_{ij}(t)=e^{\vec{i}h(i,j)t}$ for some function $h:\cS^{2}\to\R$ and 
$$ e^{\vec{i}h(j,i)t}\ =\ \psi_{ji}(t)\ =\ \overline{\psi_{ij}(t)}\ =\ e^{-\vec{i}h(i,j)t}, $$
thus $h(j,i)=-h(i,j)$ and particularly $h(i,i)=0$ for all $i,j\in\cS$. Finally, fix any $i\in\cS$, define $g(j):=h(i,j)$ for $j\in\cS$ and use $\psi_{ij}\psi_{jk}\psi_{ki}\equiv 1$ to infer
$$ 0\ =\ h(i,j)+h(j,k)+h(k,i)\ =\ h(i,j)+h(j,k)-h(i,k)\ =\ g(j)+h(j,k)-g(k), $$
i.e. $h(j,k)=g(k)-g(j)$ for all $j,k\in\cS$. But the latter means that $S_{\tau(k)}=g(k)-g(j)$ $\Prob_{j}$-a.s. and therefore
\begin{align*}
\Prob_{j}(X_{1}=g(M_{1})-g(M_{0}))\ &=\ \sum_{k\in\cS}\Prob_{j}(X_{1}=g(k)-g(j),\,M_{1}=k)\\
&=\ \sum_{k\in\cS}\Prob_{j}(S_{\tau(k)}=g(k)-g(j),\tau(k)=1)\\
&=\ \sum_{k\in\cS}\Prob_{j}(\tau(k)=1)\ =\ 1
\end{align*}
for all $j,k\in\cS$ which shows that $(M_{n},S_{n})_{n\ge 0}$ is indeed null-homologous.\qed
\end{proof}

\begin{Lemma}\label{lem:classify null}
If  $\liminf_{n\to\infty}\,S_{n}$ or $\limsup_{n\to\infty}\,S_{n}$ is $\Prob_{i}$-a.s. finite for some $i\in\cS$, then $(M_{n},S_{n})_{n\ge 0}$ is null-homologous.
\end{Lemma}

\begin{proof}
First observe that, if $Y$ is a copy of $\liminf_{n\to\infty}S_{n}$ or $\limsup_{n\to\infty}S_{n}$ and is independent of $S_{\tau(i)}$ under $\Prob_{i}$ for each $i$, then the stochastic fixed point equation
\begin{equation}\label{SFPE limsup}
Y\ \eqdist\ S_{\tau(i)}+Y\quad\text{under }\Prob_{i}
\end{equation}
holds for all $i\in\cS$, where $\eqdist$ means equality in law. This follows from
$$ \liminf_{n\to\infty}\,S_{n}\ =\ S_{\tau(i)}\ +\ \liminf_{n\to\infty}\,(S_{n}-S_{\tau(i)})\quad\Prob_{i}\text{-a.s.} $$
and a similar equation for $\limsup_{n\to\infty}S_{n}$.

Now, if $\liminf_{n\to\infty}\,S_{n}$ or $\limsup_{n\to\infty}\,S_{n}$ is a.s.\ real-valued with respect to some $\Prob_{i}$, then \eqref{SFPE limsup} has a real-valued solution which in turn entails $S_{\tau(i)}=0$ $\Prob_{i}$-a.s. as one can easily check by using characteristic functions. By Lemma \ref{lem:zero increments}, we infer that $(M_{n},S_{n})_{n\ge 0}$ is null-homologous.
\qed
\end{proof}

Here is the classification result for null-homologous MRW, the straightforward proof of which can be omitted.

\begin{Prop}\label{prop:classify null}
If $(M_{n},S_{n})_{n\ge 0}$ is null-homologous with $g$ as in \eqref{NH1}, then exactly one of the following five alternatives holds true:
\begin{description}[\sf (NH-5)]
\item[\sf (NH-1)] $g\equiv 0$ and $S_{n}=0$ a.s. for all $n\ge 1$.\vspace{.1cm}
\item[\sf (NH-2)] $0\ne\sup_{i\in\cS}|g(i)|<\infty$ and 
$$ -\infty\ <\ \liminf_{n\to\infty}\,S_{n}\ \le\ \limsup_{n\to\infty}\,S_{n}\ <\ \infty\quad\Prob_{\pi}\text{-a.s.} $$
\item[\sf (NH-3)] $-\infty=\inf_{i\in\cS}g(i)<\sup_{i\in\cS}g(i)<\infty$ and
$$ \liminf_{n\to\infty}\,S_{n}\ =\ -\infty\quad\text{and}\quad\limsup_{n\to\infty}\,S_{n}\ \in\ \R\quad\Prob_{\pi}\text{-a.s.} $$
\item[\sf (NH-4)] $-\infty<\inf_{i\in\cS}g(i)<\sup_{i\in\cS}g(i)=\infty$ and
$$ \liminf_{n\to\infty}\,S_{n}\ \in\ \R\quad\text{and}\quad\limsup_{n\to\infty}\,S_{n}\ =\ \infty\quad\Prob_{\pi}\text{-a.s.} $$
\item[\sf (NH-5)] $\inf_{i\in\cS}g(i)=-\infty$, $\sup_{i\in\cS}g(i)=\infty$ and
$$ \liminf_{n\to\infty}\,S_{n}\ =\ -\infty\quad\text{and}\quad\limsup_{n\to\infty}\,S_{n}\ =\ \infty\quad\Prob_{\pi}\text{-a.s.} $$
\end{description}
Moreover, whenever $\sigma_{n}^{\dagger}<\infty$ a.s. for some $\dagger\in\{>,<\}$ and all $n\ge 0$, then the associated ladder chain $(M_{n}^{\dagger})_{n\ge 0}$ is transient, that is
$$ \Prob_{i}(M_{n}^{\dagger}=i\text{ \rm infinitely often})\ <\ 1 $$
for all $i\in\cS$.
\end{Prop}

Notice that alternatives {\sf (NH-3)}--{\sf (NH-5)} are only possible if $\cS$ is infinite. We also point out that the null-homology of $(M_{n},S_{n})_{n\ge 0}$ (with $g$ as in \eqref{NH1}) and its dual (with $-g$) are equivalent as following immediately from the fact that $({}^{\#}M_{n},{}^{\#}X_{n})_{n\ge 1}$ and $(M_{-n},X_{-n+1})_{n\ge 1}$ are equal in law under $\Prob_{\pi}$ (see Section \ref{sec:preliminaries}).

\vspace{.1cm}
With the help of Lemma \ref{lem:classify null}, it is easy to prove the following zero-one law which in turn entails the announced trichotomy for nontrivial MRW, Prop.~\ref{prop:classify nontrivial} below. Recall that $S_{n}^{*}=\max_{0\le k\le n}S_{k}$ for $n\in\N$.

\begin{Prop}\label{prop:classify nontrivial}
For the additive part of a nontrivial MRW $(M_{n},S_{n})_{n\ge 0}$, exactly one of the three alternatives {\sf (PD)}, {\sf (ND)}, or {\sf (Osc)} occurs a.s.
\end{Prop}

\begin{proof}
We first show that $Y\in\{\limsup_{n\to\infty}S_{n},\liminf_{n\to\infty}S_{n}\}$ satisfies
$$ \Prob_{i}(Y=\infty)\ =\ 1-\Prob_{i}(Y=-\infty)\ \in\ \{0,1\} $$
for all $i\in\cS$. Fixing any $i$, it suffices to consider $Y=\limsup_{n\to\infty}S_{n}$. Put $\kappa:=\Prob_{i}(Y=\infty)$ and $\tau_{n}^{*}(i):=\inf\{\tau_{k}(i):\tau_{k}(i)\ge n\}$ for $n\ge 0$. Then
\begin{align*}
\kappa\ &=\ \lim_{m\to\infty}\,\Prob_{i}(Y=\infty,\,S_{m}^{*}>t)\\
&=\ \lim_{m\to\infty}\,\Prob_{i}\left(\limsup_{n\to\infty}\,(S_{n}-S_{\tau_{m}^{*}(i)})=\infty,\,S_{m}^{*}>t\right)\\
&=\ \Prob_{i}(Y=\infty)\,\lim_{m\to\infty}\,\Prob_{i}(S_{m}^{*}>t)\\
&=\ \kappa\,\Prob_{i}\left(\sup_{n\ge 0}\,S_{n}>t\right)
\end{align*}
for all $t\in\R_{>}$. Hence, either $\kappa=0$, or $\sup_{n\ge 0}S_{n}=\infty$ $\Prob_{i}$-a.s. and thus $\kappa=1$.

\vspace{.1cm}
Now, if $(M_{n},S_{n})_{n\ge 0}$ is nontrivial, then Lemma \ref{lem:classify null} implies that any of $\liminf_{n\to\infty}S_{n}$ and $\limsup_{n\to\infty}S_{n}$ equals one of the values $-\infty$ or $+\infty$ with probability one under $\Prob_{\pi}$. Hence, if {\sf (ND)} and {\sf (Osc)} fail to hold, {\sf (PD)} remains as the only alternative.\qed 
\end{proof}

An analog of Kesten's trichotomy in the case when $\Erw_{i}|S_{\tau(i)}|=\infty$ for some $i\in\cS$ can also be given and will be proved in Subsection \ref{subsec:Kesten trichotomy}. On the other hand, Example \ref{exa:infinite petal flower chain} will show that the more natural, but weaker condition $\Erw_{\pi}|X_{1}|=\infty$ is not sufficient for this result.

\begin{Theorem}\label{thm:Kesten trichotomy MRW}
If $\Erw_{i}|S_{\tau(i)}|=\infty$ for some/all $i\in\cS$, then exactly one of the following three alternatives holds true:
\begin{description}[\sf (Osc+)]\itemsep3pt
\item[\sf (PD+)] $\lim_{n\to\infty}\,n^{-1}S_{n}=\infty$ a.s.
\item[\sf (ND+)] $\lim_{n\to\infty}\,n^{-1}S_{n}=-\infty$ a.s.
\item[\sf (Osc+)] $\liminf_{n\to\infty}\,n^{-1}S_{n}=-\infty$ and $\limsup_{n\to\infty}\,n^{-1}S_{n}=\infty$ a.s.
\end{description}
\end{Theorem}

\section{Main results}\label{sec:main results}

Our main results are generalizations of the Theorems \ref{thm:Spitzer-Erickson}, \ref{thm:Kesten-Maller} and \ref{thm:Kesten-Maller min} to nontrivial MRW $(M_{n},S_{n})_{n\ge 0}$. In order to state them, a number of quantities must be introduced. For $i\in\cS$, $n\in\N$ and $x\in\R_{\geqslant}$, put
\begin{align*}
&A_{i}(x)\,:=\,\Erw_{i}(S_{\tau(i)}^{+}\wedge x)-\Erw_{i}(S_{\tau(i)}^{-}\wedge x),\\
&J_{i}(0)\,:=\,1\quad\text{and}\quad J_{i}(x)\,:=\,
\begin{cases}
\displaystyle{\frac{x}{\Erw_{i}(S_{\tau(i)}^{+}\wedge x)}},&\text{if }{\Prob_{i}(S_{\tau(i)}> 0)>0}\\[2mm]
\hfill x,&\text{otherwise}
\end{cases},\quad x>0,\\
&D_{n}^{i}\,:=\,\max_{\tau_{n-1}(i)<k\le\tau_{n}(i)}(S_{k}-S_{\tau_{n-1}(i)})^{-}
\end{align*}
and let $D^{i}$ be a copy of the iid $D_{n}^{i}$ under $\Prob_{i}$ which is independent of all other occurring random variables. Note that the $D_{n}^{i}$ describe the maximal downward excursions within the cycles of the driving chain determined by successive visits of state $i$. Their common law, thus $\Prob_{i}(D^{i}\in\cdot)$, is one of the relevant excursion measures that appear in our results, but we also need the measures $\V_{i}^{\alpha}$, defined on $\R_{>}$ by
\begin{align}\label{eq:def V_alpha^s}
\V_{i}^{\alpha}((x,\infty))\,:=\,\Erw_{i}\left(\sum_{n=1}^{\tau(i)}\1_{\{S_{n}^{-}>x\}}\right)^{\alpha},\quad x\in\R_{\geqslant},
\end{align}
for $\alpha>0$ and $i\in\cS$. They obviously satisfy
\begin{equation}\label{eq:tail D^s and V_alpha}
\Prob_{i}(D^{i}>x)\ \le\ \V_{i}^{\alpha}((x,\infty))\ \le\ \Erw_{i}\big(\tau(i)^{\alpha}\1_{\{D^{i}>x\}}\big)
\end{equation}
for all $x\in\R_{\geqslant}$ and have therefore heavier tails than $D^{i}$ under $\Prob_{i}$. In the case $\alpha=1$, we simply write $\V_{i}$ for $\V_{i}^{1}$.

\begin{Theorem}\label{thm:Spitzer-Erickson MRW}
For a nontrivial MRW $(M_{n},S_{n})_{n\ge 0}$, consider the following assertions:
\begin{description}[(b)]\itemsep3pt
\item[(a)] $(M_{n},S_{n})_{n\ge 0}$ is positive divergent, that is $\lim_{n\to\infty}S_{n}=\infty$ a.s.
\item[(b)] $A_{i}(x)>0$ for all sufficiently large $x$ and $\Erw_{i}J_{i}(D^{i})<\infty$ for some/all $i\in\cS$.
\item[(c)] $\sum_{n\ge 1}n^{-1}\,\Prob_{i}(S_{n}\le x)<\infty$ for some/all $(i,x)\in\cS\times\R_{\geqslant}$.
\item[(d)] $\Erw_{i}\sg(x)<\infty$ for all $(i,x)\in\cS\times\R_{\geqslant}$.
\end{description}
Then (a) $\LRA$ (b) $\Longrightarrow$ (c) $\Longrightarrow$ (d). Moreover, if $\Erw_{i}\tau(i)\log\tau(i)<\infty$ for some/all $i\in\cS$ is assumed, then part (c) holds true iff
\begin{equation}\label{eq:Spitzer series criterion}
\int\log J_{i}(x)\ \V_{i}(dx)\ <\ \infty
\end{equation}
for some/all $i\in\cS$.
\end{Theorem}

\begin{Rem}\label{rem:Erickson's extension MRW}\rm
Using Erickson's result mentioned in Remark \ref{rem:Erickson's extension}, our proof will show that, if $\Erw_{i}|S_{\tau(i)}|=\infty$ and thus $\Erw_{i}S_{\tau(i)}^{+}+\Erw_{i}D^{i}=\infty$, then (a) follows from (b) without the assumption of ultimate positivity of $A_{i}(x)$. On the other hand, if $\Erw_{i}|S_{\tau(i)}|<\infty$ and thus $J_{i}(x)=O(x)$ as $x\to\infty$, then (b) does not reduce to the ultimate positivity of $A_{i}(x)$ as in the case of ordinary random walks because $\Erw_{i}D^{i}=\infty$ is still possible.
\end{Rem}

Further conditions equivalent to (a) are $\Prob_{i}(\sigma_{\min}<\infty)=1$, $\Prob_{i}(\rho(x)<\infty)=1$ and $\Prob_{i}(N(x)<\infty)=1$ for some/all $(i,x)\in\cS\times\R_{\geqslant}$ and also $\Prob_{i}(\sle(-x)=\infty)>0$ for some $(i,x)\in\cS\times\R_{\geqslant}$. This follows, as for ordinary RW (see after Remark \ref{rem:Erickson's extension}), from the basic fluctuation-type trichotomy for nontrivial MRW. On the other hand, the last condition does not need to be true for all $(i,x)$. For instance, it is easy to provide an example of a nontrivial  positive divergent MRW $(M_{n},S_{n})_{n\ge 0}$ with $\Prob_{i}(X_{1}\le -x)=1$ for some $(i,x)\in\cS\times\R_{\geqslant}$, thus $\Prob_{i}(\sle(-x)=1)=1$.

\vspace{.1cm}
Every ordinary random walk $(S_{n})_{n\ge 0}$ can be viewed as a MRW with single-state, say 0, driving chain $(M_{n})_{n\ge 0}$, giving $S_{\tau(0)}^{+}=X_{1}^{+}$ and $D_{1}^{0}= X_{1}^{-}$. Therefore, assertion (b) of Theorem \ref{thm:Spitzer-Erickson MRW} and of Theorem \ref{thm:Spitzer-Erickson} are identical. On the other hand, parts (c) and (d) of Theorem \ref{thm:Spitzer-Erickson MRW} are no longer sufficient for the positive divergence of $(S_{n})_{n\ge 0}$ as will be shown in Example \ref{exa:infinite petal flower chain}.

\vspace{.1cm}
It would be desirable and is plausible to believe that fluctuation-theoretic properties of $(S_{n})_{n\ge 0}$ are encoded in the stationary increment distribution $\Prob_{\pi}(X_{1}\in\cdot)$. As for Theorem \ref{thm:Spitzer-Erickson MRW}, this could mean to replace $\Erw_{i}J_{i}(D^{i})<\infty$ in part (b) with
$$ \int_{\R_{>}}\frac{x}{\Erw_{\pi}(X_{1}^{+}\wedge x)}\ \Prob_{\pi}(X_{1}^{-}\in dx)\ <\ \infty $$
However, Example \ref{exa:Sisyphus} will show that this fails to work in general. The tail behavior of $X_{1}^{\pm}$ under $\Prob_{\pi}$ can actually be very different from the tail behavior of $S_{\tau(i)}^{\pm}$ under $\Prob_{i}$ in the strong sense that
$$ \lim_{x\to\infty}\frac{\Prob_{\pi}(X_{1}^{\pm}>x)}{\Prob_{i}(S_{\tau(i)}^{\pm}>x)}\ =\ \infty. $$
The occurrence of $D^{i}$ in Theorem \ref{thm:Spitzer-Erickson MRW} and also in some of the subsequent results actually indicates that within a cycle between two successive visits to a state $i$ of the driving chain, it is the extremal behavior of the random walk rather than the average one that determines some of its characteristic features.

\vspace{.2cm}
The next four results combined provide the counterpart of Theorem \ref{thm:Kesten-Maller} and once again a more complex picture than in the case of ordinary random walks. For better presentation, we call a MRW $(M_{n},S_{n})_{n\ge 0}$ with positive recurrent driving chain $(M_{n})_{n\ge 0}$ to be of \emph{type $\alpha$} for $\alpha>0$ if $\Erw_{i}\tau(i)^{1+\alpha}<\infty$ for some/all $i\in\cS$.

\begin{Theorem}\label{thm:Kesten-Maller MRW}
Let $\alpha>0$ and $(M_{n},S_{n})_{n\ge 0}$ be a nontrivial, positive divergent MRW of type $\alpha$. Then the following assertions are equivalent:
\begin{description}[(b)]\itemsep3pt
\item[(a)] $\Erw_{i}J_{i}(D^{i})^{1+\alpha}<\infty$ for some/all $i\in\cS$.
\item[(b)] $\Erw_{i}\rho(x)^{\alpha}<\infty$ for some/all $(i,x)\in\cS\times\R_{\geqslant}$.
\item[(c)] $\Erw_{i}\sigma_{\min}^{\alpha}<\infty$ for some/all $i\in\cS$.
\item[(d)] For some/all $i\in\cS$, there exists $x\in\R_{\geqslant}$ such that $\Prob_{i}(\sle(-x)=\infty)>0$ and $\Erw_{i}\sle(-x)^{\alpha}\1_{\{\sle(-x)<\infty\}}<\infty$. In this case, the last expectation is also finite for any other $x\in\R_{\geqslant}$.
\end{description}
\end{Theorem}

\begin{Theorem}\label{thm:K/M MRW Spitzer series}
Given a nontrivial MRW $(M_{n},S_{n})_{n\ge 0}$ of type $\alpha>0$, the following assertions are equivalent:
\begin{description}[(b)]\itemsep3pt
\item[(a)] $A_{i}(x)>0$ for all sufficiently large $x$, $\Erw_{i}J_{i}(S_{\tau(i)}^{-})^{1+\alpha}<\infty$ and\begin{equation}\label{eq:MRW Spitzer series}
\int J_{i}(x)^{\alpha}\ \V_{i}(dx)\ <\ \infty
\end{equation}
for some/all $i\in\cS$.
\item[(b)] $\sum_{n\ge 1}n^{\alpha-1}\Prob_{i}(S_{n}\le x)<\infty$ for some/all $(i,x)\in\cS\times\R_{\geqslant}$.
\end{description}
\end{Theorem}

As stated at the end of Theorem \ref{thm:Spitzer-Erickson MRW}, the equivalence remains valid in the case $\alpha=0$ when replacing $J_{i}(x)^{\alpha}$ with $\log J_{i}(x)$ in \eqref{eq:MRW Spitzer series} and the assumption $\Erw_{i}\tau(i)^{1+\alpha}<\infty$ with $\Erw_{i}\tau(i)\log\tau(i)<\infty$.

\begin{Theorem}\label{thm:K/M MRW N(x)}
Given a nontrivial, positive divergent MRW $(M_{n},S_{n})_{n\ge 0}$ of type $\alpha$, consider the following assertions:
\begin{description}[(b)]\itemsep2pt
\item[(a)] $\Erw_{i}J_{i}(S_{\tau(i)}^{-})^{1+\alpha}<\infty$ and
\begin{equation}\label{eq:MRW renewal counting}
\int J_{i}(x)\ \V_{i}^{\alpha}(dx)\ <\ \infty
\end{equation}
for some/all $i\in\cS$.
\item[(b)] $\Erw_{i}N(x)^{\alpha}<\infty$ for some/all $(i,x)\in\cS\times\R_{\geqslant}$.
\end{description}
Then (a) $\LRA$ (b) if $\alpha\ge 1$, and (a) $\RA$ (b) if $0<\alpha<1$.
\end{Theorem}

\begin{Theorem}\label{thm:connections}
Given a nontrivial, positive divergent MRW $(M_{n},S_{n})_{n\ge 0}$ of type $\alpha$, the following implications hold true:
\begin{align*}
&\alpha\ge 1:\quad\text{Thm. \ref{thm:Kesten-Maller MRW}(a)-(d)}\ \Longrightarrow\ \text{Thm. \ref{thm:K/M MRW Spitzer series}(a),(b)}\ \Longrightarrow\ \text{Thm. \ref{thm:K/M MRW N(x)}(a),(b)}.\\
&\alpha\le 1:\quad\text{Thm. \ref{thm:Kesten-Maller MRW}(a)-(d)}\ \Longrightarrow\ \text{Thm. \ref{thm:K/M MRW N(x)}(a),(b)}\ \Longrightarrow\ \text{Thm. \ref{thm:K/M MRW Spitzer series}(a),(b)}.
\end{align*}
Furthermore, any of the conditions in the afore-mentioned theorems implies
\begin{equation}\label{eq: sigma alpha cond}
\Erw_{i}\sg(x)^{1+\alpha}<\infty \quad\text{ for all }(i,x)\in\cS\times\R_{\geqslant}
\end{equation}
\end{Theorem}

Again, the previous theorems are weaker than their counterpart in the iid case. For $\alpha>0$, $\Erw_{i}\sg(x)^{1+\alpha}<\infty$ and the conditions provided in Theorems \ref{thm:K/M MRW Spitzer series} and \ref{thm:K/M MRW N(x)} are only necessary for those in Theorem \ref{thm:Kesten-Maller MRW}. Moreover, Theorem \ref{thm:connections} holds the surprise that $\Erw_{i}N(x)^{\alpha}<\infty$ for all $(i,x)\in\cS\times\R_{\geqslant}$ implies $\sum_{n\ge 1}n^{\alpha-1}\Prob_{i}(S_{n}\le x)<\infty$ for all $(i,x)\in\cS\times\R_{\geqslant}$ if $\alpha\le 1$, whereas the converse is true if $\alpha\ge 1$ (in the case $\alpha=1$ both assertions are obviously identical). Example \ref{exa:integral criteria} will show that equivalence of all stated conditions generally fails to hold. In fact, the assertions of Theorem \ref{thm:K/M MRW Spitzer series} for $\alpha\in(0,1)$ may be valid although $(S_{n})_{n\ge 0}$ is not even positive divergent, thus $\Prob_{i}(\rho(x)<\infty)=0$ for all $i\in\cS$ and $x\in\R$. We further point out that one cannot generally dispense with
\begin{equation}\label{eq:moment tau(s)}
\Erw_{i}\tau(i)^{1+\alpha}<\infty
\end{equation}
for some/all $i\in\cS$. To see this, consider a MRW $(M_{n},S_{n})_{n\ge 0}$ such that, for some distribution $F$ on $\R_{>}$, the conditional law of $X_{n}$ given $(M_{n-1},M_{n})$ equals $F$ if $M_{n}=i$ and $\delta_{0}$ otherwise. Then $\rho(0)+1=N(0)+1=\sigma^>(0)=\tau(i)$ $\Prob_{i}$-a.s. and thus $\Erw_{i}\rho(0)^{1+\alpha}<\infty$ is indeed equivalent to \eqref{eq:moment tau(s)}.

\vspace{.1cm}
The counterpart of Theorem \ref{thm:Kesten-Maller min} is stated as the next theorem.

\begin{Theorem}\label{thm:Kesten-Maller MRW min}
For $\alpha>0$ and a positive divergent MRW $(M_{n},S_{n})_{n\ge 0}$, consider the following assertions:
\begin{description}[(b)]\itemsep3pt
\item[(a)] $\Erw_{i}(D^{i})^{\alpha}J_{i}(D^{i})<\infty$ for some/all $i\in\cS$.
\item[(b)] $\Erw_{i}|\min_{n\ge 0}S_{n}|^{\alpha}<\infty$ for some/all $i\in\cS$.
\item[(c)] $\Erw_{i}|S_{\sle(-x)}|^{\alpha}\1_{\{\sle(-x)<\infty\}}$ for some/all $(i,x)\in\cS\times\R_{\geqslant}$.
\end{description}
Then (a) $\LRA$ (b) $\Longrightarrow$ (c).
\end{Theorem}

Example \ref{exa:general IPFC example} will show that (c) does not generally imply (b).

\vspace{.1cm}
Despite the previous disclaiming remarks regarding the equivalence of all conditions in the above theorems, it is natural to ask whether this is at least true in the case when the driving chain has finite state space. The positive answer is provided by the three subsequent theorems.

\begin{Theorem}\label{thm:Spitzer-Erickson MRW finite S}
Given the situation of Theorem \ref{thm:Spitzer-Erickson MRW} with finite state space $\cS$, all its assertions (a)--(d) as well as
\begin{equation}\label{eq:J_pi condition S finite}
A_{\pi}(x)>0\text{ for all sufficiently large $x$\quad and}\quad\Erw_{\pi}J_{\pi}(X_{1}^{-})<\infty
\end{equation}
are equivalent, where
$A_{\pi}(x):=\Erw_{\pi}(X_{1}\wedge x)^{+}-\Erw_{\pi}(X_{1}\wedge x)^{-}$ and
\begin{equation*}
J_{\pi}(0)\,:=\,1\quad\text{and}\quad J_{\pi}(x)\,:=\,
\begin{cases}
\displaystyle{\frac{x}{\Erw_{\pi}(X_{1}\wedge x)^{+}}},&\text{if }\Prob_{\pi}(X_{1}> 0)>0\\
\hfill x,&\text{otherwise}
\end{cases},\quad x>0.
\end{equation*}
Moreover, $D^{i}$ in \ref{thm:Spitzer-Erickson MRW}(b) may be replaced with $S_{\tau(i)}^{-}$.
\end{Theorem}

\begin{Theorem}\label{thm:Kesten-Maller MRW finite S}
Let $(M_{n},S_{n})_{n\ge 0}$ be a MRW of type $\alpha$ and $\cS$ be finite. Then all assertions of Theorems \ref{thm:Kesten-Maller MRW}, \ref{thm:K/M MRW Spitzer series}, \ref{thm:K/M MRW N(x)} and also \eqref{eq: sigma alpha cond} are equivalent to
\begin{equation}\label{eq:J_pi condition S finite K/M}
\Erw_{\pi}J_{\pi}(X^{-})^{1+\alpha}<\infty.
\end{equation}
Moreover, $D^{i}$ in \ref{thm:Kesten-Maller MRW}(b) may be replaced with $S_{\tau(i)}^{-}$.
\end{Theorem}

\begin{Theorem}\label{thm:Kesten-Maller MRW min finite S}
Given the situation of Theorem \ref{thm:Kesten-Maller MRW min} with finite state space $\cS$, all its assertions (a)--(c) as well as \eqref{eq: sigma alpha cond} and
\begin{equation}\label{eq:J_pi condition S finite min}
\Erw_{\pi}(X_{1}^{-})^{\alpha}J_{\pi}(X^{-})<\infty.
\end{equation}
are equivalent. Moreover, $D^{i}$ in \ref{thm:Kesten-Maller MRW min}(b) may be replaced with $S_{\tau(i)}^{-}$. 
\end{Theorem}

Without the conditions \eqref{eq:J_pi condition S finite}, \eqref{eq:J_pi condition S finite K/M} and \eqref{eq:J_pi condition S finite min} involving $J_{\pi}$, the last three results would merely be corollaries to the previous ones. However, the inclusion of these conditions will cause some extra work provided by Lemmata \ref{lem:asymp finite space} and \ref{lem:pi int}.

\section{Solidarity results}\label{sec:solidarity}

Observe that the additive part of a nontrivial MRW $(M_{n},S_{n})_{n\ge 0}$ is actually a countable union of ordinary RW, namely
$$ \{S_{n}:n\ge 1\}\ =\ \bigcup_{i\in\cS}\left\{S_{\tau_{n}(i)}:n\ge 1\right\}, $$
which are however randomly intertwined. The following simple solidarity lemma shows that these embedded sequences share the fluctuation type, but the subsequent counterexample will disprove the natural conjecture that this type is also shared by $(S_{n})_{n\ge 0}$ itself. It further illustrates that the fluctuation types of $(S_{n})_{n\ge 0}$ and its dual $({}^{\#}S_{n})_{n\ge 0}$ may be different as well.

\begin{Lemma}\label{lem:solidarity}
If $(M_{n},S_{n})_{n\ge 0}$ is nontrivial, then all $(S_{\tau_{n}(i)})_{n\ge 0}$, $i\in\cS$, are nontrivial and of the same fluctuation type {\sf (PD)}, {\sf (ND)} or {\sf (Osc)}.
\end{Lemma}

\begin{proof}
All $(S_{\tau_{n}(i)})_{n\ge 0}$, $i\in\cS$, are nontrivial by Lemma \ref{lem:zero increments}. Fix any distinct $i,j\in\cS$. Since $\limsup_{n\to\infty}S_{\tau_{n}(i)}=x$ a.s. for $x\in\{\pm\infty\}$, we can choose a subsequence $(\tau_{n}'(i))_{n\ge 1}$ of $(\tau_{n}(i))_{n\ge 1}$ such that each $\tau_{n}'(i)$ is a stopping time for $(M_{n},S_{n})_{n\ge 0}$ and $\lim_{n\to\infty}S_{\tau_{n}'(i)}=x$ a.s. Now pick $m\in\N$ and $t>0$ such that $\Prob_{i}(M_{m}=j,|S_{m}|\le t)>0$. We then infer by a geometric trials argument that
$$ \Prob_{\pi}\left(M_{\tau_{n}'(i)+m}=j,\,|S_{\tau_{n}'(i)+m}-S_{\tau_{n}'(i)}|\le t\text{ infinitely often}\right)\ =\ 1 $$
and thus $\limsup_{n\to\infty}S_{\tau_{n}(j)}\ge\limsup_{n\to\infty}S_{\tau_{n}(i)}$ a.s. The reverse inequality and thus equality follows by interchanging the roles of $i$ and $j$.
Finally, by switching to the reflected MRW $(M_{n},-S_{n})_{n\ge 0}$, we find $\liminf_{n\to\infty}S_{\tau_{n}(j)}=\liminf_{n\to\infty}S_{\tau_{n}(i)}$ a.s. as well.\qed
\end{proof}

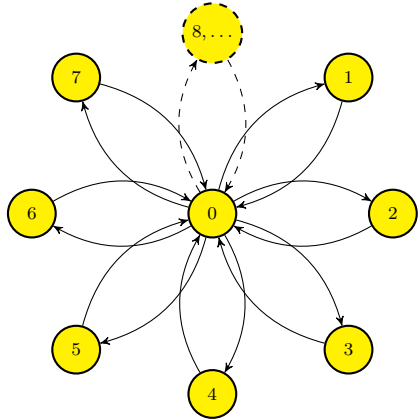
\begin{figure}[b]
\begin{center}
\begin{tikzpicture}[->, >=stealth', auto, thin, node distance=3cm]
\tikzstyle{every state}=[fill=yellow,draw=black,thick,text=black,scale=.8]
\node[state]    (0)                     {$0$};
\node[state]    (1)[above right of=0, node distance=3.2cm]   {$1$};
\node[state]    (2)[right of=0]   {$2$};
\node[state]    (3)[below right of=0, node distance=3.2cm]   {$3$};
\node[state]    (4)[below of=0]   {$4$};
\node[state]    (5)[below left of=0, node distance=3.2cm]   {$5$};
\node[state]    (6)[left of=0]   {$6$};
\node[state]    (7)[above left of=0, node distance=3.2cm]   {$7$};
\node[state,dashed]    (8)[above of=0]   {$8,\ldots$};
\path
(0) edge[bend left]     node{}     (1)
     edge[bend left,below]      node{}      (3)
     edge[bend left]    node{}      (2)
     edge[bend left]    node{}      (4)
     edge[bend left]    node{}      (5)
     edge[bend left]    node{}      (6)
     edge[bend left]    node{}      (7)
     edge[bend left,dashed]    node{}      (8)
(1) edge[bend left]     node{}     (0)
(2) edge[bend left]                node{}           (0)
(3) edge[bend left,above]     node{}         (0)
(4) edge[bend left,above]     node{}         (0)
(5) edge[bend left,above]     node{}         (0)
(6) edge[bend left,above]     node{}         (0)
(7) edge[bend left,above]     node{}         (0)
(8) edge[bend left,above,dashed]     node{}         (0);
\end{tikzpicture}
\end{center}
\caption{Transition graph of an infinite petal flower chain.}
\label{fig:IPFC}
\end{figure}

\begin{Exa}[MRW driven by an infinite petal flower chain]\label{exa:infinite petal flower chain}\rm Cons\-ider a Markov chain $(M_{n})_{n\ge 0}$ on the set $\N_{0}$ of nonnegative integers which, when in state 0, picks an arbitrary $i\in\N$ with positive probability $p_{0i}$ and jumps back to 0, otherwise, thus $p_{i0}=1$. If we figure the $i\in\N$ being placed on a circle around 0, the transition diagram of this chain looks like a \emph{flower with infinitely many petals}, each of the petals representing a transition from 0 to some $i$ and back. With all $p_{0i}$ being positive, the chain is clearly irreducible and positive recurrent with stationary probabilities $\pi_{0}=\frac{1}{2}$ and
$$ \pi_{i}\ =\ \frac{1}{2}\,\Erw_{0}\left(\sum_{n=0}^{\tau(0)-1}\1_{\{M_{n}=i\}}\right)\ =\ \frac{1}{2}\,\Prob_{0}(M_{1}=i)\ =\ \frac{p_{0i}}{2}. $$
In fact, under $\Prob_{0}$, the chain consists of independent random variables which are 0 for even $n$ and iid with common distribution $(p_{0i})_{i\ge 1}$ for odd $n$.

Turning to the additive component, we define the $X_{n}$ by
$$ X_{n}\,:=\,
\begin{cases}
\hfill -p_{0i}^{-1},&\text{if }M_{n-1}=0,\,M_{n}=i,\\
2+p_{0i}^{-1},&\text{if }M_{n-1}=i,\,M_{n}=0
\end{cases}
$$
for $i,n\in\N$, i.e.\ $K_{0i}=\delta_{-p_{0i}^{-1}}$ and $K_{i0}=\delta_{2+p_{0i}^{-1}}$. Then $\Erw_{\pi}|X_{1}|=\infty$ and
$$ S_{n}\,:=\,
\begin{cases}
n-1-p_{0M_{n}}^{-1},&\text{if }n\text{ is odd},\\
\hfill n,&\text{if }n\text{ is even}
\end{cases}
\quad\Prob_{0}\text{-a.s.}
$$
It follows that $(M_{n},S_{n})_{n\ge 0}$ is oscillating, for
\begin{align}
&\hspace{1cm}\lim_{n\to\infty}\frac{S_{2n}}{2n}\ =\ 1\label{eq:SLLN petal flower}
\shortintertext{and}
\liminf_{n\to\infty}S_{2n+1}\ &=\ \liminf_{n\to\infty}\left(n-1-\frac{1}{p_{0M_{2n+1}}}\right)\ =\ -\infty\quad\Prob_{0}\text{-a.s.}\nonumber
\end{align}
The last assertion follows from the fact that, for any $a>0$,
$$ \sum_{n\ge 0}\Prob_{0}\left(\frac{1}{p_{0M_{2n+1}}}>an\right)\ =\ \sum_{n\ge 0}\Prob_{0}(X_{1}^->an)\ =\ \frac{\Erw_{0}X_{1}^-}{a}\ =\ \infty $$
and an appeal to the Borel-Cantelli lemma, giving
\begin{equation}\label{eq:1/p_{0i}>n i.o.}
\Prob_{0}\left(\frac{1}{p_{0M_{2n+1}}}>an\text{ i.o.}\right)\ =\ 1.
\end{equation}
On the other hand, $(S_{\tau_{n}(0)})_{n\ge 0}$, which equals $(S_{2n})_{n\ge 0}$ under $\Prob_{0}$, is positive divergent, whence the same holds true for all other $(S_{\tau_{n}(i)})_{n\ge 0}$ by Lemma \ref{lem:solidarity}. Regarding Theorem \ref{thm:Kesten trichotomy MRW}, this example shows that $\Erw_{\pi}|X_{1}|=\infty$ is not sufficient for Kesten's trichotomy to hold (see \eqref{eq:SLLN petal flower}).

\vspace{.2cm}
Let us further point out that the dual $({}^{\#}M_{n},{}^{\#}S_{n})_{n\ge 0}$ has increments
$$ {}^{\#}X_{n}\,:=\,
\begin{cases}
2+p_{0i}^{-1},&\text{if }{}^{\#}M_{n-1}=0,\,{}^{\#}M_{n}=i,\\
\hfill -p_{0i}^{-1},&\text{if }{}^{\#}M_{n-1}=i,\,{}^{\#}M_{n}=0
\end{cases}
$$
for $n\ge 1$ and is therefore positive divergent, for
$$ {}^{\#}S_{n}\,:=\,
\begin{cases}
n+1+p_{0{}^{\#}\!M_{n}}^{-1},&\text{if }n\text{ is odd},\\
\hfill n,&\text{if }n\text{ is even}
\end{cases}
\quad\Prob_{0}\text{-a.s.}
$$

Regarding the strictly descending ladder epochs $\sln$ and the associated ladder walk $(\Mln,\Sln)_{n\ge 0}$, the most notable consequence of the properties assessed before, especially \eqref{eq:1/p_{0i}>n i.o.}, is that all $\sln$ are a.s.\ finite, but the ladder chain $\Mln$ must be transient, for otherwise $\liminf_{n\to\infty}S_{\tau_{n}(i)}=-\infty$ $\Prob_{i}$-a.s. would hold for at least one $i\in\cS$.

\vspace{.2cm}
If we alter the definition of the $X_{n}$ by setting
$$ X_{n}\,:=\,
\begin{cases}
\hfill -p_{0i}^{-1},&\text{if }M_{n-1}=0,\,M_{n}=i\text{ and $i$ is even},\\
\hfill -p_{0i}^{-1},&\text{if }M_{n-1}=i,\,M_{n}=0\text{ and $i$ is odd},\\
2+p_{0i}^{-1},&\text{if }M_{n-1}=i,\,M_{n}=0\text{ and $i$ is even},\\
2+p_{0i}^{-1},&\text{if }M_{n-1}=0,\,M_{n}=i\text{ and $i$ is odd}
\end{cases}
$$
for $n\ge 1$, then even both, $(S_{n})_{n\ge 0}$ and its dual $({}^{\#}S_{n})_{n\ge 0}$, are easily seen to be oscillating despite still having positive divergent embedded random walks $(S_{\tau_{n}(i)})_{n\ge 0}$. As a consequence, putting $\nu(x)=\inf\{n\ge 1:S_{\tau_{n}(i)}>x\}$ for $x\in\R_{\geqslant}$, we have that $\Erw_{i}\nu(x)<\infty$ and then, by using Wald's identity, also
$$ \Erw_{i}\tau_{\nu(x)}<\infty. $$ 
Since $\sg(x)\le\tau_{\nu(x)}$, we obtain $\Erw_{i}\sg(x)<\infty$ for all $x\in\R_{\geqslant}$, and this also ensures $\Erw_{j}\sg(x)<\infty$ for all $x\in\R_{\geqslant}$ and any other $j\in\cS$, because
\begin{align*}
\infty\ >\ \Erw_{i}\sg(x)\ &\ge\ \int_{(-\infty,x/2]}\Erw_{j}\sg(x-y)\ \Prob_{i}(S_{\tau(j)}\in dy,\sg(x)>\tau(j))\\
&\ge\ \Prob_{j}(S_{\tau(i)}\le x/2,\sg(x)>\tau(j))\,\Erw_{j}\sg(x/2)
\end{align*}
for all $x\in\R_{\geqslant}$. We have thus shown that Theorem \ref{thm:Spitzer-Erickson MRW}(d) may hold although $(S_{n})_{n\ge 0}$ is oscillating. In fact, since $\Erw_{0}\tau(0)\log\tau(0)<\infty$ trivially holds and \eqref{eq:Spitzer series criterion} for $i=0$ is readily verified to take the form
$$ \Erw_{0}\log^{+}S_{\tau(0)-1}^{-}\ =\ \sum_{i\ge 1}p_{0i}|\log p_{0i}|\ <\ \infty, $$
we further see that Theorem \ref{thm:Spitzer-Erickson MRW}(c) may or may not be valid, depending on whether or not $ \sum_{i\ge 1}p_{0i}|\log p_{0i}|$ is finite. In other words, the two assertions (c) and (d) of that theorem are only sufficient but not necessary for the positive divergence of $(S_{n})_{n\ge 0}$.
\end{Exa}

In view of the previous example, a nontrivial MRW $(M_{n},S_{n})_{n\ge 0}$ is called \emph{regular} if it shares its fluctuation type with its dual $({}^{\#}M_{n},{}^{\#}S_{n})_{n\ge 0}$ as well as all its embedded random walks $(S_{\tau_{n}(i)})_{n\ge 0}$. Nonregularity can only occur if $\Erw_{\pi}X_{1}^{+}=\Erw_{\pi}X_{1}^{-}=\infty$ as the next result shows.

\begin{Prop}\label{prop:regular <-> mean exists}
A nontrivial MRW $(M_{n},S_{n})_{n\ge 0}$ is regular if its stationary drift $\mu=\Erw_{\pi}X_{1}$ exists, i.e. $\Erw_{\pi}X_{1}^{+}<\infty$ or $\Erw_{\pi}X_{1}^{-}<\infty$.
\end{Prop}

\begin{proof}
If $\mu=\Erw_{\pi}X_{1}$ exists, then $\Erw_{i}S_{\tau(i)}=\pi_{i}^{-1}\mu$ (see \eqref{eq:stat mean drift}) exists as well for any $i\in\cS$. Now use Birkhoff's ergodic theorem (see e.g.\ \cite[Thm. 6.21]{Breiman:68}), applied to the ergodic stationary sequences $(X_{n})_{n\ge 1}$ and $({}^{\#}X_{n})_{n\ge 1}$ under $\Prob_{\pi}$, to infer $n^{-1}S_{n}\to\mu$ and $n^{-1}{}^{\#}S_{n}\to\mu$ a.s. Moreover, $n^{-1}S_{\tau_{n}(i)}\to\pi_{i}^{-1}\mu$ a.s. for all $i\in\cS$ by the strong law of large numbers. But this shows that $(S_{n})_{n\ge 0}$, its dual and all embedded sequences $(S_{\tau_{n}(i)})_{n\ge 0}$ do indeed share the fluctuation type.\qed
\end{proof}

For the proof of our main results, it is a crucial ingredient that the embedded RW $(S_{\tau_{n}(i)})_{n\ge 0}$, $i\in\cS$, not only share their fluctuation type (Lemma \ref{lem:solidarity}) but also finite power moments of the fluctuation-theoretic quantities 
we have introduced. In other words, we need solidarity of these RW as to validity of Theorems \ref{thm:Kesten-Maller} and \ref{thm:Kesten-Maller min} which in turn follows if the respective criteria
\begin{align}\label{eq:J-criteria}
\Erw_{i}J_{i}(S_{\tau(i)}^{-})^{1+\alpha}\ <\ \infty\quad\text{and}\quad\Erw_{i}(S_{\tau(i)}^{-})^{\alpha}J_{i}(S_{\tau(i)}^{-})\ <\ \infty
\end{align}
are valid for all $i\in\cS$ if valid for some $i$. The simple observation that these criteria may be rewritten as integrals involving the tail functions of $S_{\tau(i)}^{-}$ and $S_{\tau(i)}^{+}$ suggests to arrive at this conclusion by showing $\Prob_{i}(S_{\tau(i)}^{\pm}>y)\asymp\Prob_{j}(S_{\tau(j)}^{\pm}>y)$ as $y\to\infty$ for all $i,j\in\cS$ (tail solidarity). Unfortunately, by another look at the previous example, where $\Prob_0(S_{\tau(0)}>y)=0$ for $y>2$ but $\Prob_{i}(S_{\tau(i)}>y)>0$ for all $y\in\R_{\geqslant}$ and any $i\in\N$, we see that this cannot generally hold. We will therefore resort to a weaker but still sufficient kind of tail solidarity. The result will then be stated as Lemma \ref{lem:solidarity stronger}.

\vspace{.2cm}
For distinct $i,j\in\cS$, let
$$\upsilon\ :=\ \upsilon(i,j)\ :=\ \inf\{n\ge 1: \tau_n(i)>\tau(j)\} $$
denote the first return time to $i$ after the first visit to state $j$. Notice that $\Erw_{i}\upsilon^{1+\alpha}<\infty$ for any $\alpha\ge 0$ because $\Prob_{i}(\upsilon>n)=\Prob_{i}(\tau(i)<\tau(j))^n$ for all $n\ge 1$.

\begin{Lemma}\label{Lemma: tailasymp}
Let $i,j\in\cS$ be distinct states. 
\begin{description}
\item[(a)] There exists $x\in\R_{\geqslant}$ such that, as $y\to\infty$, 
$$ \Prob_{j}(S_{\tau(j)}>y)\ \lesssim\  \Prob_{i}(S_{\tau_{\upsilon}(i)}>y-x). $$
\item[(b)]$\Prob_{j}(S_{\tau(j)}>0)>0$ implies $\Prob_{i}(S_{\tau_{\upsilon}(i)}>0)>0$.
\end{description}
\end{Lemma}

\begin{proof}
(a) First note that
$$ \Prob_{j}(S_{\tau(j)}>y)\ =\ \Prob_{j}(S_{\tau(j)}>y,\tau(i)<\tau(j))\ +\ \Prob_{j}(S_{\tau(j)}>y,\, \tau(i)>\tau(j)).$$
Regarding the first term on the right-hand side and using
$$ \Prob_{i}(S_{\tau_{\upsilon}(i)}-S_{\tau(j)}\in\cdot)\ =\ \Prob_{j}(S_{\tau(i)}\in\cdot), $$ 
we infer for any $y\in\R$
\begin{align}
\begin{split}\label{eq:Stau(j)>y}
\Prob_{j}(S_{\tau(j)}>y,\,\tau(i)<\tau(j))\ &=\ \Prob_{j}(S_{\tau(i)} +(S_{\tau(j)}-S_{\tau(i)})>y,\, \tau(i)<\tau(j))\\
&\le\ \int\Prob_{i}(S_{\tau(j)}>y-x)\ \Prob_{j}(S_{\tau(i)}\in dx)\\
&=\ \int\Prob_{i}(S_{\tau(j)}>y-x)\ \Prob_{i}(S_{\tau_{\upsilon}(i)}-S_{\tau(j)}\in dx)\\
&=\ \Prob_{i}(S_{\tau_{\upsilon}(i)}>y),
\end{split}
\end{align}
and this proves the assertion if $\Prob_{j}(\tau(i)>\tau(j))=0$. Otherwise, pick $x\in\R_{\geqslant}$ such that 
$$ p_{1}\ :=\ \Prob_{i}(S_{\tau(j)}\ge -x/4,\, \tau(i)>\tau(j))\ >\ 0\quad\text{and}\quad p_{2}\ :=\ \Prob_{j}(S_{\tau(i)}\ge -x/4)\ >\ 0 $$
and note that $q:=\Prob_{i}(\tau(i)>\tau(j))>0$. Then we obtain with $p:=p_{1}p_{2}q>0$
\begin{align*}
&\Prob_{j}\big(S_{\tau(j)}>y,\, \tau(i)>\tau(j)\big)\\
&=\ q^{-1}\,\Prob_{i}\big(S_{\tau_{2}(j)}-S_{\tau(j)}>y,\,\tau(i)>\tau_{2}(j)\big)\\
&\le\ q^{-1}\,\Prob_{i}\big(S_{\tau_{2}(j)}-S_{\tau(j)}>y,\,M_{k}\ne i\text{ for }\tau(j)<k<\tau_{2}(j)\big)\\
&\le\ p^{-1}\,\Prob_{i}\big(S_{\tau(i)}>y-x/2,\,S_{\tau(j)}\ge -x/4,\, S_{\tau(i)}-S_{\tau_{2}(j)}\ge -x/4,\,\tau(i)>\tau_{2}(j)\big)\\
&\le\ p^{-1}\,\Prob_{i}\big(S_{\tau(i)}>y-x/2\big).
\end{align*}
Moreover,
\begin{align*}
&\Prob_{i}\big(S_{\tau_{\upsilon}(i)}>y-x\big)\\
&=\ \Prob_{i}\big(S_{\tau(i)}>y-x,\, \upsilon=1\big)\ +\ \Prob_{i}\big(S_{\tau(i)}+(S_{\tau_{\upsilon}(i)}-S_{\tau(i)})>y-x,\, \upsilon>1\big)\\
&\ge\ \Prob_{i}\big(S_{\tau(i)}>y-x/2,\,\upsilon=1\big)\\
&\qquad+\ \Prob_{i}\big(S_{\tau(i)}+(S_{\tau_{\upsilon}(i)}-S_{\tau(i)})>y-x,\,
S_{\tau_{\upsilon}(i)}-S_{\tau(i)}\ge -x/2,\,\upsilon>1\big)\\
&\ge\ \Prob_{i}\big(S_{\tau(i)}>y-x/2,\, \upsilon=1\big)\ +\ p_1\, p_{2}\,\Prob_{i}\big(S_{\tau(i)}>y-x/2,\,\upsilon>1\big)\\
&\ge\ p_{1}\,p_{2}\,\Prob_{i}\big(S_{\tau(i)}>y-x/2\big)
\end{align*}
which in combination with the previous estimate provides the assertion (with $x$ as chosen above) if $\Prob_{j}(\tau(i)>\tau(j))>0$.

\vspace{.2cm}	
(b) If $\Prob_{j}(S_{\tau(j)}>0,\tau(i)<\tau(j))>0$, the assertion follows immediately from \eqref{eq:Stau(j)>y}. Otherwise, 
$$\Prob_{j}(S_{\tau(j)}>\varepsilon,\, \tau(i)>\tau(j))\ >\ 0$$
for some $\varepsilon>0$. Then, upon choosing $x\in\R_{\geqslant}$ and $p_{1}, p_{2}>0$ as in (a), the assertion follows from 
\begin{align*}
\Prob_{i}(S_{\tau_{\upsilon}(i)}>0)\ &\ge\ p_{1}\,p_{2}\,\Prob_{j}(S_{\tau_{n}(j)}>x/2,\,\tau(i)>\tau_{n}(j))\\
&\ge\ p_{1}\,p_{2}\,\Prob_{j}(S_{\tau(j)}>\varepsilon,\, \tau(i)>\tau(j))^{n}\ >\ 0,
\end{align*}
where $n=\lceil x/2\eps\rceil$.\qed
\end{proof}

Defining, for $i\in\cS$, $\gamma\in[0,1]$ and $x\in\R_{\geqslant}$,
\begin{align*}
J_{i,\gamma}(x)\, &:=\,
\begin{cases}
\hfill\displaystyle{\frac{x}{[\Erw_{i}(S_{\tau(i)}^{+}\wedge x)]^{\gamma}}},&\text{if }\Prob_{i}(S_{\tau(i)}>0)>0\\
\hfill x,&\text{otherwise}
\end{cases},
\end{align*}
where $0/[\Erw_{i}(S_{\tau(i)}^{+}\wedge 0)]^\gamma:=1$ if $\gamma>0$, the expectations in \eqref{eq:J-criteria} can be treated hereafter in a unified manner because $J_{i}(x)=J_{i,1}(x)$ and $y^{\alpha}J_{i}(x)=J_{i,1/(1+\alpha)}(x)$. The next lemma collects some relevant properties of $J_{i,\gamma}$ and $A_{i}$.

\begin{Lemma}\label{lem:trunc mean}
The following assertions are true for any $\gamma\in[0,1]$:
\begin{description}\itemsep3pt
\item[(a)] $J_{i,\gamma}$ is subadditive and nondecreasing for all $i\in\cS$.
\item[(b)] $J_{i,\gamma}(y)\asymp J_{i,\gamma}(x+y)$ as $y\to\infty$ for all $(i,x)\in\cS\times\R$.
\item[(c)] $J_{i,\gamma}(y)\asymp J_{j,\gamma}(y)$ as $y\to\infty$ for all $i,j\in\cS$.
\end{description}
If the $(S_{\tau_{n}(i)})_{n\ge 0}$ are positive divergent, then furthermore
\begin{description}\itemsep3pt
\item[(d)] $A_{i}(y)>0$ for all sufficiently large $y$,
\item[(e)] $A_{i}(y) \asymp\Erw_{i}(S_{\tau(i)}^{+}\wedge y)\asymp\Erw_{i}(S_{\tg(i)}\wedge y)$ as $y\to\infty$
\end{description}
for all $i\in\cS$.
\end{Lemma}

\begin{proof}
(a) Subadditivity follows from the monotonicity of the denominator in the definition of $J_{i,\gamma}$ and monotonicity from the identity
$$J_{i,\gamma}(y)\ = \ \frac{y^{1-\gamma} }{ [y^{-1}\int_{0}^{y} \Prob_{i}(S_{\tau(i)}^{+}>x)\ dx]^\gamma}\ = \ \frac{y^{1-\gamma}}{[\int_{0}^{1} \Prob_{i}(S_{\tau(i)}>xy)\ dx]^\gamma},\quad y\in\R_{\geqslant}$$
(cf. the proof of \cite[Lemma 5.4(a)]{AlsIksMei:15}).

\vspace{.1cm}
(b) follows immediately from the properties asserted in (a).

\vspace{.1cm}
(c) Fix any distinct $i,j\in\cS$ and then $x\in\R_{\geqslant}$ such that Lemma \ref{Lemma: tailasymp}(a) holds. The assertion is obvious if we verify that
$$ \Erw_{i}(S_{\tau(i)}^{+}\wedge y)\ \asymp\ \Erw_{j}(S_{\tau(j)}^{+}\wedge y) $$
as $y\to\infty$. We obtain
\begin{align*}
\Erw_{j}(S_{\tau(j)}^{+}\wedge y)\ &=\ \int_{0}^{y}\Prob_{j}(S_{\tau(j)}>z)\ dz\ \lesssim\ \int_{0}^{y} \Prob_{i}(S_{\tau_{\upsilon}(i)}>z-x)\ dz\\
&\lesssim\ \int_{0}^{y-x} \Prob_{i}(S_{\tau_{\upsilon}(i)}>z)\ dz\ =\ \Erw_{i}\big[S_{\tau_{\upsilon}(i)}^{+}\wedge (y-x)\big]\\
&=\ \sum_{n\ge 1 }\Prob_{i}(\upsilon=n)\, \Erw_{i}\big[S_{\tau_n(i)}^{+}\wedge (y-x)\big|\upsilon=n\big]\\
&\le\ \sum_{n\ge 1 }\Prob_{i}(\upsilon=n)\, \sum_{k=1}^n \Erw_{i}\big[(S_{\tau_{k}(i)}-S_{\tau_{k-1}(i)})^{+}\wedge (y-x)\big|\upsilon=n\big],
\end{align*}
where Lemma \ref{Lemma: tailasymp}(b) has been utilized in the third step. Notice that assertion (b) particularly yields
$$\Erw_{i}(S_{\tau(i)}^{+}\wedge (y-x))\ \asymp\ \Erw_{i}(S_{\tau(i)}^{+}\wedge y) $$
as $y\to\infty$. Next, we have
\begin{align*}
\Erw_{i}\big[(S_{\tau_n(i)}-S_{\tau_{n-1}(i)})^{+}\wedge (y-x)\big|\upsilon=n\big]\ &=\ \Erw_{i}[S_{\tau(i)}^{+}\wedge (y-x)|\tau(i)>\tau(j)]\\
&\lesssim\ \Erw_{i}[S_{\tau(i)}^{+}\wedge(y-x)]\ \asymp \ \Erw_{i}(S_{\tau(i)}^{+}\wedge y)
\end{align*}
as $y\to\infty$ and, given $\Prob_{i}(\tau(i)<\tau(j))>0$, 
\begin{align*}
\Erw_{i}\big[(S_{\tau_{k}(i)}-S_{\tau_{k-1}(i)})^{+}\wedge (y-x)\big|\upsilon=n\big] \ &=\ \Erw_{i}[S_{\tau(i)}^{+}\wedge (y-x)|\tau(i)<\tau(j)]\\
&\lesssim\ \Erw_{i}(S_{\tau(i)}^{+}\wedge y)\qquad \text{as }y\to\infty
\end{align*}
for $1\le k<n$. Consequently,
$$ \Erw_{j}(S_{\tau(j)}^{+}\wedge y)\ \lesssim\ \sum_{n\ge 1} \Prob_{i}(\upsilon=n)\, n\cdot\Erw_{i}(S_{\tau(i)}^{+}\wedge y) \ = \ \Erw_{i}\upsilon \cdot \Erw_{i}(S_{\tau(i)}^{+}\wedge y), $$ 
thus $\Erw_{j}(S_{\tau(j)}^{+}\wedge y)\lesssim\Erw_{i}(S_{\tau(i)}^{+}\wedge y)$. The reverse relation follows by symmetry of the argument in $i$ and $j$.

\vspace{.1cm}	
(d) and (e) can be extracted from \cite[Lemma 3.2, the proof of Lemma 3.1 and (4.5)]{KesMal:96}.\qed
\end{proof}

\begin{Lemma}\label{lem:solidarity stronger}
The following assertions hold true either for all $i\in\cS$ or none:
\begin{description}\itemsep2pt
\item[(a)] $\Erw_{i}J_{i}(S_{\tau(i)}^{-})^{1+\alpha}<\infty$.
\item[(b)] $\Erw_{i}(S_{\tau(i)}^{-})^{\alpha}\,J_{i}(S_{\tau(i)}^{-})<\infty$.
\item[(c)] $\Erw_{i}S_{\tau(i)}^{-}<\infty$.
\end{description}
\end{Lemma}

By applying the lemma to the MRW $(M_{n},-S_{n})_{n\ge 0}$, we see it remains valid with $S_{\tau(i)}^{+}$ instead of $S_{\tau(i)}^{-}$, and a combination of both results shows that $\Erw_{i}|S_{\tau(i)}|<\infty$ is either true for all $i\in\cS$ or none.

\begin{proof}
All parts follow if we can prove that, for any $\gamma\in [0,1]$, $\Erw_{i} J_{i,\gamma}(S_{\tau(i)}^{-})^{1+\alpha}<\infty$ is true either for all $i\in\cS$ or none.

\vspace{.2cm}	
Suppose $\Erw_{i} J_{i,\gamma}(S_{\tau(i)}^{-})^{1+\alpha}<\infty$ and pick an arbitrary $j\in\cS\setminus\{i\}$. An application of Lemma \ref{Lemma: tailasymp}(a) to $(M_{n},-S_{n})_{n\ge 0}$ ensures the existence of $x\in\R_{\geqslant}$ with 
$$ \Prob_{j}(S_{\tau(j)}^{-}>y)\ \lesssim\ \Prob_{i}(S_{\tau_{\upsilon}(i)}^{-}>y-x) $$
as $y\to\infty$. Hence, by an appeal to Lemma \ref{lem:trunc mean}(b), it suffices to prove 
$$ \Erw_{i} J_{i,\gamma}(S_{\tau_{\upsilon}(i)}^{-})^{1+\alpha}\ <\ \infty. $$
Setting $Y_{k}:=S_{\tau_{k}(i)}-S_{\tau_{k-1}(i)}$ for $k\ge 1$, subadditivity of $J_{i,\gamma}$ yields
\begin{align*}
\Erw_{i} J_{i,\gamma}(S_{\tau_{\upsilon}(i)}^{-})^{1+\alpha}\ 
&\le\ \Erw_{i}\left[ \sum_{k=1}^{\upsilon} J_{i,\gamma}\big(Y_{k}^{-}\big)\right]^{1+\alpha}\ \le\ \Erw_{i}\left[\upsilon^{\alpha}\sum_{k=1}^{\upsilon}J_{i,\gamma}\big(Y_{k}^{-}\big)^{1+\alpha}\right]\\
&=\ \sum_{n\ge 1} \Prob_{i}(\upsilon=n)\,n^{\alpha}\,\sum_{k=1}^{n}\Erw_{i}\Big[J_{i,\gamma}\big(Y_{k}^{-}\big)^{1+\alpha}\Big|\upsilon=n\Big].\end{align*}
Now use
$$ \Erw_{i}\left[J_{i,\gamma}\big(Y_{n}^{-}\big)^{1+\alpha}\Big |\upsilon=n\right]\,=\,\Erw_{i} \big(J_{i,\gamma}(S_{\tau(i)}^{-})^{1+\alpha}\big|\tau(i)>\tau(j)\big)\,=:\,c_{1}\,<\,\infty $$
and, given $\Prob_{i}(\tau(i)<\tau(j))>0$,
$$ \Erw_{i} \left[J_{i,\gamma}\big(Y_{k}^{-}\big)^{1+\alpha}\Big|\upsilon=n\right]\,=\,\Erw_{i} \big(J_{i,\gamma}(S_{\tau(i)}^{-})^{1+\alpha}\big|\tau(i)<\tau(j)\big)\,=:\,c_{2}\,<\,\infty $$
for $1\le k<n$ to arrive at the desired conclusion
$$ \Erw_{i} J_{i,\gamma}(S_{\tau_{\upsilon}(i)}^{-})^{1+\alpha}\ \le\ (c_{1}\vee c_{2})\, \Erw_{i}\upsilon^{1+\alpha}\ <\ \infty.\eqno\qed $$
\end{proof}

The next solidarity lemma is of similar kind, but for $D^{i}$ instead of $S_{\tau(i)}^{-}$.

\begin{Lemma}\label{lem:solidarity D^s}
The following assertions hold either for all $i\in\cS$ or none:
\begin{description}\itemsep2pt
\item[(a)] $\Erw_{i}J_{i}(D^{i})^{1+\alpha}<\infty$.
\item[(b)] $\Erw_{i}(D^{i})^{\alpha}J_{i}(D^{i})<\infty$.
\item[(c)] $\Erw_{i}(D^{i})^{1+\alpha}<\infty$.
\end{description}
\end{Lemma}

\begin{proof}
Again, it suffices to prove that, for any $\gamma\in [0,1]$, $\Erw_{i} J_{i,\gamma}(D^{i})^{1+\alpha}<\infty$ holds either for all $i\in\cS$ or none. 

\vspace{.1cm}
Suppose $\Erw_{i} J_{i,\gamma}(D^{i})^{1+\alpha}<\infty$ for some $i\in\cS$ and $\gamma\in[0,1]$. Pick an arbitrary $j\in\cS\backslash\{i\}$, define 
$$ D_{\upsilon}\ := \ \max_{1\le k\le \tau_{\upsilon}(i)} S_{k}^{-} $$
and use $D_{\upsilon}\le \sum_{k=1}^\upsilon D_{k}^{i}$ to infer $\Erw_{i} J_{i,\gamma}(D_{\upsilon})^{1+\alpha^{}}<\infty$ in the same manner as the finiteness of $\Erw_{i} J_{i,\gamma}(S_{\tau_{\upsilon}(i)}^{-})^{1+\alpha}$ in the proof of 
Lemma \ref{lem:solidarity stronger}. Then put
$$ \upsilon_{2}\ :=\ \inf\{n\ge 1: \tau_n(i)>\tau_{2}(j)\}\quad\text{and}\quad D_{\upsilon_{2}}\ := \ \max_{1\le k\le \tau_{\upsilon_{2}}(i)}S_{k}^{-} $$
and notice that
$$ D_{\upsilon_{2}}\ \le\ D_{\upsilon}\ +\ \max_{\tau_{\upsilon}(i)< k\le  \tau_{\upsilon_{2}}(i)} (S_{k}-S_{\tau_{\upsilon}(i)})^{-}. $$
The second summand on the right-hand side is 0 if $\tau_{\upsilon}(i)=\tau_{\upsilon_{2}}(i)$ and an independent copy of $D_{\upsilon}$ otherwise. Hence, in both cases one easily obtains $\Erw_{i} J_{i,\gamma}(D_{\upsilon_{2}})^{1+\alpha}<\infty$. Picking $x\in\R_{\geqslant}$ with $\Prob_{i}(S_{\tau(j)}\le x)>0$, we finally infer with the help of Lemma \ref{lem:trunc mean} 
\begin{align*}
\infty\ &>\ \Erw_{i} J_{i,\gamma}(D_{\upsilon_{2}})^{1+\alpha}\ \gtrsim\ \int_{(0,\infty)} J_{i,\gamma}(y)^{1+\alpha}\ \Prob_{i}(D_{\upsilon_{2}}\in dy\,|S_{\tau(j)}\le x)\\
&\ge \ \int_{(0,\infty)} J_{i,\gamma}(y)^{1+\alpha}\ \Prob_{i}\Big(\max_{\tau_{1}(j)< k\le \tau_{2}(j)}S_{k}^{-}\in dy\,\Big|S_{\tau_{1}(j)}\le x\Big)\\
&\ge\ \int_{(0,\infty)} J_{i,\gamma}(y)^{1+\alpha}\ \Prob_{i}\Big( \max_{\tau_{1}(j)< k\le \tau_{2}(j)}(S_{k}-S_{\tau_{1}(j)})^{-}-x\in dy\,\Big|S_{\tau_{1}(j)}\le x\Big)\\
&=\ \int_{(0,\infty)} J_{i,\gamma}(y)^{1+\alpha}\ \Prob_{j}(D_{1}^{j}-x\in dy)\\
&\asymp\ \Erw_{j} J_{j,\gamma}(D^j)^{1+\alpha}.\hspace{8cm}\qed
\end{align*}
\end{proof}

Our final solidarity lemma provides sufficient conditions for the existence of power moments of $\sigma^{>}$, see Section \ref{sec:moment sg(x)}. Note that a geometric number of cycles marked by successive visits to a state $j\ne i$ contains a visit to $i\in\cS$. Hence $\Erw_{i}\tau(i)^{1+\alpha}<\infty$ is satisfied either for all $i\in\cS$ or none.

\begin{Lemma}\label{lem:solidarity LI}
Let $\alpha\ge 0$ and $\Erw_{i}\tau(i)^{1+\alpha}<\infty$ for some/all $i\in\cS$. The following conditions are equivalent:
\begin{description}[(b)]\itemsep2pt
\item[(a)] $\Erw_{i} \tau_{\nu(x)}(i)^{1+\alpha}<\infty$ for some/all $(i,x)\in\cS\times\R_{\geqslant}$.
\item[(b)] $A_{i}(x)>0$ for all sufficiently large $x$ and $\Erw_{i}J_{i}(S_{\tau(i)}^{-})^{1+\alpha}<\infty$ for some/all $i\in\cS$.
\end{description}	
In particular, these conditions imply \eqref{eq: sigma alpha cond}.
\end{Lemma}

\begin{proof}
By Lemmata \ref{lem:solidarity}, \ref{lem:solidarity stronger} and Theorem \ref{thm:Spitzer-Erickson}, (b) holds true either for all $i\in\cS$ or none. Moreover, \eqref{eq: sigma alpha cond} follows from $\tau_{\nu(x)}(i)\ge\sigma^{>}(x)$.

\vspace{.1cm}
Suppose $\Erw_{i}\tau_{\nu(x)}(i)^{1+\alpha}<\infty$ for some $(i,x)\in\cS\times\R_{\geqslant}$. Since $\tau_{\nu(x)}(i)\ge \nu(x)$, we obtain $\Erw_{i}\nu(x)^{1+\alpha}<\infty$ which is equivalent to (b) by Theorems \ref{thm:Spitzer-Erickson} and \ref{thm:Kesten-Maller}.

\vspace{.1cm}	
Since $\Erw_{i}\tau(i)^{1+\alpha}<\infty$, the reverse implication follows directly from $\Erw_{i}\nu(x)^{1+\alpha}<\infty$ and \cite[Theorem 1.5.4]{Gut:09}.\qed
\end{proof}

\section{Proofs of the main results}\label{sec:proof main results}

For ease of notation, we write $\tau,\tau_{n},{}^{\#}\tau,\tg$, etc. for $\tau(i),\tau_{n}(i),{}^{\#}\tau(i),\tg(i)$, etc. in all subsequent proofs when a fixed $i\in\cS$ is considered. We further put
$$ \Sigma_{\alpha}(i,x)\ :=\ \sum_{n\ge 1}n^{\alpha-1}\,\Prob_{i}(S_{n}\le x) $$
for $\alpha\in\R_{\geqslant}$ and $(i,x)\in\cS\times\R_{\geqslant}$. In analogy to $J_{i}$, let $J_{i}^{>}$ be defined by
\begin{align*}
J_{i}^{>}(0)\,:=\,1\quad\text{and}\quad J_{i}^{>}(x)\,:=\,
\frac{x}{\Erw_{i}(S_{\tg(i)}\wedge x)}\quad\text{for }x>0,
\end{align*}
and put
$$ D_{n}^{i,>}\,:=\,\max_{\tg_{n-1}(i)<k\le\tg_{n}(i)}(S_{k}-S_{\tg_{n-1}(i)})^{-} $$
for $n\in\N$. Let $D^{i,>}$ denote a generic copy of these iid random variables under $\Prob_{i}$ which is independent of all other occurring random variables. Note that
$$ D_{1}^{i}\ \le\ D_{1}^{i,>}\ \le\ \sum_{k=1}^{\zeta_{1}}D_{k}^{i}, $$
where $\tg(i)=\tau_{\zeta_{1}}(i)$ with $\zeta_{1}=\inf\{n\ge 1:S_{\tau_{n}(i)}>0\}$ should be recalled from the end of Section \ref{sec:preliminaries}. Then use Lemma \ref{lem:trunc mean}(e), giving $J_{i}(x)\asymp J_{i}^{>}(x)$, in combination with the monotonicity and subadditivity of $J_{i}$ and Thm. 1.5.4 in \cite{Gut:09} to infer that, if $(S_{\tau_{n}(i)})_{n\ge 0}$ is positive divergent, then
\begin{equation}\label{eq:equivalence J_s and J_{s}^{>} moments}
\Erw_{i}J_{i}^{>}(D^{i,>})^{1+\beta}\,<\,\infty\quad\Longleftrightarrow\quad\Erw_{i}J_{i}(D^{i})^{1+\beta}\,<\,\infty
\end{equation}
for any $\beta\in\R_{\geqslant}$.

\subsection{Proof of Theorem \ref{thm:Spitzer-Erickson MRW}}\label{subsec:proof Spitzer-Erickson}

The proof of Theorem \ref{thm:Spitzer-Erickson MRW} can be found at the end of this subsection after some auxiliary lemmata. We start by quoting the following straightforward generalization of a result by Erickson \cite[Lemma 4]{Erickson:73}.

\begin{Lemma}\label{lem:Erickson}
Let $(X_{n},Y_{n})_{n\ge 1}$ be a sequence of iid pairs of nonnegative random variables with generic copy $(X,Y)$ and $\Erw X+\Erw Y=\infty$. Then
\begin{align*}
\limsup_{n\to\infty}\frac{Y_{n+1}}{\sum_{k=1}^{n}X_{k}}\ =\ 0\quad\text{or}\quad=\ \infty\quad\text{a.s. \ }
\shortintertext{according as}
\int_{\R_{>}}\frac{y}{\Erw(X\wedge y)}\ \Prob(Y\in dy)\ <\ \infty\quad\text{or}\quad=\ \infty.
\end{align*}
\end{Lemma}

\begin{proof}[of Theorem \ref{thm:Spitzer-Erickson MRW}]
``(a)$\RA$(b)''. Suppose that $(S_{n})_{n\ge 0}$ is positive divergent. Noting that, for any $i\in\cS$, $(S_{\tau_{n}}-D_{n+1}^{i})_{n\ge 0}$ forms a subsequence of $(S_{n})_{n\ge 0}$ in the sense that the former sequence equals $(S_{\xi_{n}})_{n\ge 0}$ for an increasing random sequence $(\xi_{n})_{n\ge 0}$, we see that $S_{n}\to\infty$ a.s. ensures $S_{\tau_{n}}-D_{n+1}^{i}\to\infty$ a.s., hence $S_{\tau_{n}}\to\infty$ a.s. and
$$ \limsup_{n\to\infty}\frac{D_{n+1}^{i}}{\sum_{k=1}^{n}(S_{\tau_{k}}-S_{\tau_{k-1}})^{+}}\ \le\ \limsup_{n\to\infty}\frac{D_{n+1}^{i}}{S_{\tau_{n}}}\ <\ \infty\quad\text{a.s.} $$
Consequently, $\Erw_{i}J_{i}(D^{i})<\infty$ follows by Erickson's Lemma \ref{lem:Erickson}. Finally, $A_{i}(x)>0$ for all sufficiently large $x$ can now be inferred from Lemma 3.2 by Kesten and Maller \cite{KesMal:96}.

\vspace{.2cm}
``(b)$\RA$(a)''. Fix an arbitrary $i\in\cS$ and consider two cases.

\vspace{.1cm}
\textsc{Case 1}. If $\Erw_{i}S_{\tau}^{+}+\Erw_{i}D^{i}<\infty$, then $0<\lim_{y\to\infty}A_{i}(y)=\Erw_{i}S_{\tau}^{+}<\infty$ ensures the positive divergence of $(S_{\tau_{n}})_{n\ge 0}$ and $\lim_{n\to\infty}n^{-1}S_{\tau_{n}}=\Erw_{i}S_{\tau}$ a.s. Moreover,
$$ \sum_{n\ge 1}\Prob_{i}(D_{n}^{i}>n)\ \le\ \Erw_{i}D^{i}\ <\ \infty $$
entails $\lim_{n\to\infty}n^{-1}D_{n}^{i}=0$ a.s. As a consequence,
\begin{align*}
\lim_{n\to\infty}S_{n}\ &\ge\ \lim_{n\to\infty}(S_{\tau_{\Lambda(n)}}-D_{\Lambda(n)+1}^{i})\\
&\ge\ \lim_{n\to\infty}\Lambda(n)\left(\Erw_{i}S_{\tau(i)}-\frac{D_{\Lambda(n)+1}^{i}}{\Lambda(n)}\right)\ =\ \infty\quad\text{a.s.,}
\end{align*}
where $\Lambda(x):=\sup\{n\ge 0:\tau_{n}\le x\}$.

\vspace{.1cm}
\textsc{Case 2}. If $\Erw_{i}S_{\tau}^{+}+\Erw_{i}D^{i}=\infty$, then, again by Erickson's Lemma \ref{lem:Erickson}, $\Erw_{i}J_{i}(D^{i})<\infty$ is equivalent to
\begin{equation}\label{eq:Pruitt1}
Q_{n}\,:=\,\frac{D_{n+1}^{i}}{\sum_{k=1}^{n}(S_{\tau_{k}}-S_{\tau_{k-1}})^{+}}\ \stackrel{n\to\infty}{\longrightarrow}\ 0\quad\text{a.s.}
\end{equation}
Therefore, by invoking a result of Pruitt \cite[Lemma 8.1]{Pruitt:81}, we also have the equivalence of $\Erw_{i}J_{i}(D^{i})<\infty$ with
\begin{equation}\label{eq:Pruitt2}
R_{n}\,:=\,\frac{\sum_{k=1}^{n}(S_{\tau_{k}}-S_{\tau_{k-1}})^{-}}{\sum_{k=1}^{n}(S_{\tau_{k}}-S_{\tau_{k-1}})^{+}}\ \stackrel{n\to\infty}{\longrightarrow}\ 0\quad\text{a.s.}
\end{equation}
Now we arrive at the desired conclusion via
\begin{align*}
S_{n}\ &\ge\ \sum_{k=1}^{\Lambda(n)}(S_{\tau_{k}}-S_{\tau_{k-1}})^{+}-\sum_{k=1}^{\Lambda(n)}(S_{\tau_{k}}-S_{\tau_{k-1}})^{-}-D_{\Lambda(n)+1}^{i}\\
&=\ \sum_{k=1}^{\Lambda(n)}(S_{\tau_{k}}-S_{\tau_{k-1}})^{+}\left(1-R_{\Lambda(n)}-Q_{\Lambda(n)}\right)\ \stackrel{n\to\infty}{\longrightarrow}\ \infty\quad\text{a.s.}
\end{align*}
Note that $A_{i}(x)>0$ for all sufficiently large $x$ has not been used here.

\vspace{.2cm}
Lemma \ref{lem:Spitzer series} below establishes ``(b)$\RA$(c)''. For ``(c)$\LRA$\eqref{eq:Spitzer series criterion}'' under the additional condition on $\tau(i)$, we refer to the proof of Theorem \ref{thm:K/M MRW Spitzer series} for the case $\alpha=0$. Left with ``(c)$\RA$(d)'', Lemma \ref{lem:solidarity increments} in the Appendix provides us with
\begin{align*}
\sum_{n\ge 1}\frac{1}{n}\,\Prob_{i}(S_{\tau_{n}}\le x)\ &\asymp\ \Erw_{i}\left(\frac{1}{\tau_{n}}\,\1_{\{S_{\tau_{n}}\le x\}}\right)\ \le\ \Sigma_{0}(i,x)\ <\ \infty,
\end{align*}
for all $x\in\R_{\geqslant}$, which entails $S_{\tau_{n}}\to\infty$ a.s. and $\Erw_{i}\nu(x)<\infty$ by Theorem \ref{thm:Spitzer-Erickson}. Finally, use Wald's equation to conclude
\begin{equation}\label{eq:mean of sg(x)}
\Erw_{i}\sg(x)\ \le\ \Erw_{i}\tau_{\nu(x)}\ =\ \Erw_{i}\tau\,\Erw_{i}\nu(x)\ <\ \infty
\end{equation}
for any $x\in\R_{\geqslant}$.\qed
\end{proof}

\begin{Lemma}\label{lem:Spitzer series}
Given a nontrivial MRW $(M_{n},S_{n})_{n\ge 0}$, positive divergence implies $\Sigma_{0}(i,x)<\infty$ for some/all $(i,x)\in\cS\times\R_{\geqslant}$, and the converse is also true provided that $\Erw_{\pi}X_{1}$ exists.
\end{Lemma}

\begin{proof}
Suppose that $S_{n}\to\infty$ a.s. Fixing any $(i,x)\in\cS\times\R_{\geqslant}$ and recalling that $\Erw_{i}\tg<\infty$, we obtain
\begin{align*}
\Sigma_{0}(i,x)\ &=\ \sum_{n\ge 1}\Erw_{i}\left(\sum_{k=\tg_{n-1}+1}^{\tg_{n}}\frac{1}{k}\,\1_{\{S_{\tg_{n-1}}+(S_{k}-S_{\tg_{n-1}})\le x\}}\right)\\
&\le\ \sum_{n\ge 1}\Erw_{i}\left(\frac{1}{\tg_{n-1}+1}\sum_{k=\tg_{n-1}+1}^{\min(2\tg_{n-1},\tg_{n})}\1_{\{S_{\tg_{n-1}}-D_{n}^{i,>}\le x\}}\right)\\
&\hspace{3.5cm}+\ \sum_{n\ge 1}\Erw_{i}\left(\sum_{k=2\tg_{n-1}+1}^{\tg_{n}}\frac{1}{k}\,\1_{\{\tg_{n}-\tg_{n-1}>\tg_{n-1}\}}\right)\\
&\le\ \sum_{n\ge 0}\Prob_{i}(S_{\tg_{n}}-D_{n+1}^{i,>}\le x)\ +\ \sum_{n\ge 1}\Erw_{i}\left(\sum_{k=2\tg_{n-1}+1}^{\tg_{n}}\frac{1}{k}\,\1_{\{\tg_{n}-\tg_{n-1}>\tg_{n-1}\}}\right).
\end{align*}
Finiteness of the first series will be established in the subsequent Lemma \ref{lem:harmonic sum * D^s}. As for the second one, notice that $S_{\tau_{n}}\to\infty$ $\Prob_{i}$-a.s. ensures $\Erw_{i}\tg<\infty$. Therefore,
\begin{align*}
\sum_{n\ge 1}\Erw_{i}&\left(\sum_{k=2\tg_{n-1}+1}^{\tg_{n}}\frac{1}{k}\,\1_{\{\tg_{n}-\tg_{n-1}>\tg_{n-1}\}}\right)\ \le\  \sum_{n\ge 1} \sum_{k\ge 1} \Erw_{i}\left[\frac{1}{2\tg_{n-1}+k}\,\1_{\{\tg_{n}-\tg_{n-1}\ge n-1+k}\right]\\
&\hspace{1cm}\le\ \sum_{n\ge 1} \sum_{k\ge n}\frac{1}{k}\,\Prob_{i}(\tg\ge k)\ =\ \sum_{n\ge 1} \Prob_{i}(\tg\ge n)\  =\ \Erw_{i} \tg\ <\ \infty
\end{align*}
as claimed.

\vspace{.1cm}
Now suppose that $\Erw_{\pi}X_{1}$ exists and, furthermore, that $(S_{n})_{n\ge 0}$ is not positive divergent which, by Proposition \ref{prop:regular <-> mean exists}, entails that the same holds true for any $(S_{\tau_{n}(i)})_{n\ge 0}$, $i\in\cS$, hence $\sum_{n\ge 1}n^{-1}\Prob_{i}(S_{\tau_{n}(i)}\le x)=\infty$ for all $x\in\R_{\geqslant}$. Now use Lemma \ref{lem:solidarity increments} to conclude
$$ \infty\ =\ \Erw_{i}\left(\sum_{n\ge 1}\frac{1}{\tau_{n}(i)}\,\1_{\{S_{\tau_{n}(i)}\le x\}}\right)\ \le\ \Sigma_{0}(i,x) $$
for all $x\in\R_{\geqslant}$.\qed
\end{proof}

\begin{Lemma}\label{lem:harmonic sum * D^s}
Given a nontrivial MRW $(M_{n},S_{n})_{n\ge 0}$, positive divergence implies
\begin{align*}
\sum_{n\ge 1}\Prob_{i}(S_{\tg_{n}(i)}-D_{n+1}^{i,>}\le x)\ <\ \infty
\end{align*}
for all $(i,x)\in\cS\times\R_{\geqslant}$. 
\end{Lemma}

\begin{proof}
Fix any $(i,x)\in\cS\times\R_{\geqslant}$. We have already proved that $S_{n}\to\infty$ a.s. implies $\Erw_{i}J_{i}(D^{i})<\infty$ and thus also
$\Erw_{i}J_{i}^{>}(D^{i,>})<\infty$ by \eqref{eq:equivalence J_s and J_{s}^{>} moments}. By combining this fact with
$$ \sum_{n\ge 1}\Prob_{i}(S_{\tg_{n}}\le x+y)\ \asymp\ J_{i}^{>}(x+y) $$
as $y\to\infty$ (see \eqref{eq:harmonic series asymptotics, alpha>0}), we obtain
\begin{align*}
\sum_{n\ge 1}\Prob_{i}(S_{\tg_{n}}-D_{n+1}^{i,>}\le x)\ =\ \Erw_{i}J_{i}^{>}(x+D^{i,>})\ \asymp\ \Erw_{i}J_{i}^>(D^{i,>})\ <\ \infty
\end{align*}
which proves the assertion.\qed
\end{proof}

\subsection{Proof of Theorem \ref{thm:Kesten-Maller MRW min}}\label{subsec:proof Kesten-Maller min}

Since (c) is a direct consequence of (b) when noting that 
$$ |S_{\sle(-x)}|\1_{\{\sle(-x)<\infty\}}\ \le\ |\min_{n\ge 0}S_{n}| $$ 
for all $x\in\R_{\geqslant}$, we must only show the equivalence of (a) and (b) which is accomplished by the subsequent lemma.

\begin{Lemma}\label{lem:K/M (b) and (d)}
Let $(M_{n},S_{n})_{n\ge 0}$ be positive divergent and $\alpha>0$. Then
\begin{align}
&\Erw_{i}(D^{i})^{\alpha}J_{i}(D^{i})\ <\ \infty\label{eq:integral criterion min}
\shortintertext{and}
&\hspace{.25cm}\Erw_{i}\Big|\min_{n\ge 0}S_{n}\Big|^{\alpha}\ <\ \infty\label{eq:alpha moment minimum}
\end{align}
are equivalent conditions and, if valid for one $i\in\cS$, actually hold for all $i$.
\end{Lemma}

\begin{proof}
We first point out that \eqref{eq:integral criterion min}, if true for one $i$, is actually true for all $i$ by Lemma \ref{lem:solidarity D^s}. A similar solidarity holds for \eqref{eq:alpha moment minimum} because
\begin{align*}
\Erw_{i}\left|\min_{n\ge 0}S_{n}\right|^{\alpha}\ &=\ \Erw_{i}\left|\min_{n\ge 1}S_{n}\wedge 0 \right|^{\alpha}\\
&\ge\ \Erw_{i}\left|\left(\min_{n\ge\tau(i)+1} (S_{n}-S_{\tau(i)})+S_{\tau(i)}\right)\wedge 0\right|^{\alpha}\,\1_{\{S_{\tau(i)}\le x\}}\\  
&\ge\ p\,\Erw_{i}\left|\left(\min_{n\ge 1} S_{n}+x\right)\wedge 0\right|^{\alpha}\ \asymp\ \Erw_{i}\left|\min_{n\ge 0}S_{n}\right|^{\alpha},
\end{align*}
where $x\in\R_{\geqslant}$ is chosen so large that $p:= \Prob_{i}(S_{\tau(i)}\le x)>0$.

\vspace{.1cm}
By the positive divergence of $(M_{n},S_{n})_{n\ge 0}$, we can fix $i$ such that $p:=\Prob_{i}(\sle=\infty)>0$. Again, we write $\tau,\tau_{n}$ as shorthand for $\tau(i),\tau_{n}(i)$, respectively. Define
\begin{align*}
\eta\ =\ \eta_{1}\ &:=\ \inf\{k\ge 1:S_{\tau_{k-1}}-D_{k}^{i}< 0\}
\shortintertext{and, recursively,}
\eta_{n}\ &:=\ \inf\{k>\eta_{n-1}:S_{\tau_{k-1}}-S_{\tau_{\eta_{n-1}}}-D_{k}^{i}< 0\}
\end{align*}
with the usual convention $\inf\emptyset:=\infty$. Then
$$ \nu\,:=\,\inf\{n\ge 1:\eta_{n}=\infty\} $$
has a geometric distribution with parameter $p$. Moreover,
\begin{align*}
\Big|\min_{n\ge 0}S_{n}\Big|\ &\le\ \sum_{n=1}^{\nu-1}\left|S_{\tau_{\eta_{n}-1}}-S_{\tau_{\eta_{n-1}}}-D_{\eta_{n}}^{i}\right|
\end{align*}
with $\eta_{0}:=0$ and the fact that the $S_{\tau_{\eta_{k}-1}}-S_{\tau_{\eta_{k-1}}}-D_{\eta_{k}}^{i}$ for $k=1,\ldots,n$ are conditionally iid given $\nu>n$ can be used as in \cite[p.~871,(v)$_{0}\RA$(vii)]{Janson:86} to show that \eqref{eq:alpha moment minimum} holds iff
\begin{equation}\label{eq:auxiliary moment}
\Erw_{i}\left|S_{\tau_{\eta-1}}-D_{\eta}^{i}\right|^{\alpha}\1_{\{\eta<\infty\}}\ <\ \infty.
\end{equation}
Therefore, it suffices to show the equivalence of \eqref{eq:auxiliary moment} and \eqref{eq:integral criterion min}.

\vspace{.1cm}
Use Lemma \ref{Appendix: asymp} to obtain
\begin{align*}
F_{i}(x)\ :\!&=\ \Prob_{i}(-S_{\tau_{\eta-1}(i)}+D_{\eta}^{i}\ge x,\, \eta<\infty)\\
&=\ \sum_{n\ge 1} \Prob_{i}(-S_{\tau_{n-1}(i)}+D_{n}^{i}\ge x,\, \eta=n)\\
&\le\  \sum_{n\ge 1} \Prob_{i}(S_{\tau_{n-1}(i)}\ge 0,\, -S_{\tau_{n-1}(i)}+D_{n}^{i}\ge x)\\
&=\ \int_{[x,\infty)} \sum_{n\ge 1} \Prob_{i}(0\leq S_{\tau_{n-1}(i)}\leq y-x)\ \Prob_{i}(D^{i}\in dy)\\
&\asymp\ \int_{[x,\infty)} J_{i}(y-x)\ \Prob_{i}(D^{i}\in dy)
\end{align*}
for $x\in\R_{\geqslant}$. Since $J_{i}$ is non-decreasing, we then infer
\begin{align*}
\Erw_{i}|S_{\tau_{\eta-1}(i)}-D_{\eta}^{i}|^{\alpha}\,\1_{\{\eta<\infty\}}
&\asymp\ \int_{0}^{\infty}x^{\alpha-1}\,F_{i}(x)\ dx\\
&\lesssim\ \int_0^\infty \bigg(x^{\alpha-1} \int_{[x,\infty)} J_{i}(y-x) \ \Prob_{i}(D^{i}\in dy)\bigg) \ dx\\
&\le\ \int \bigg( \int_0^y x^{\alpha-1}\,J_{i}(y)\ dx\bigg) \ \Prob_{i}(D^{i}\in dy)\\
&\asymp\ \int y^\alpha\, J_{i}(y)\ \Prob_{i}(D^{i}\in dy)\\
&=\ \Erw_{i}(D^{i})^\alpha\,J_{i}(D^{i}).
\end{align*} 
the last integral being finite iff \eqref{eq:integral criterion min} holds.

\vspace{.1cm}
On the other hand, Lemma \ref{lem:crucial series estimate} below (with $\alpha=0$) provides us with
\begin{align*}
F_{i}(x)\ &=\ \sum_{n\ge 1} \Prob_{i}(-S_{\tau_{n-1}(i)}+D_{n}^{i}\ge x,\, \eta=n)\\
&=\ \int_{[x,\infty)}\sum_{n\ge 1}\Prob_{i}\bigg(S_{\tau_{n-1}}\le y-x,\min_{1\le k\le \tau_{n-1}}S_{k}>0\bigg)\ \Prob_{i}(D^{i}\in dy)\\
&\gtrsim\ \int_{(x,\infty)}J_{i}(y-x)\ \Prob_{i}(D^{i}\in dy),
\end{align*}
and this implies
\begin{align*}
\Erw_{i}\left|S_{\tau_{\eta-1}}-D_{\eta}^{i}\right|^{\alpha}\1_{\{\eta<\infty\}}\ &\gtrsim\ \int_{1}^{\infty}x^{\alpha-1}\int_{(2x,\infty)}J_{i}(y-x)\ \Prob_{i}(D^{i}\in dy)\ dx\\
&=\ \int_{(2,\infty)}\int_{1}^{y/2}x^{\alpha-1}J_{i}(y-x)\ dx\ \Prob_{i}(D^{i}\in dy)\\
&\gtrsim\ \Erw_{i}(D^{i})^{\alpha}J_{i}(D^{i}).
\end{align*}
We thus see that \eqref{eq:auxiliary moment} and \eqref{eq:integral criterion min} are indeed equivalent.\qed
\end{proof}

The next lemma will also be needed in the next section (see proof of Lemma \ref{lem:K/M (d) implies (a)}).

\begin{Lemma}\label{lem:crucial series estimate}
Suppose that $(M_{n},S_{n})_{n\ge 0}$ is positive divergent and let $(i,x)\in\cS\times\R_{\geqslant}$ be such that $\Prob_{i}(\sle(-x)=\infty)>0$. Then, as $y\to\infty$,
\begin{align}\label{eq:crucial series estimate}
J_{i}(y)^{1+\alpha}\,\lesssim\ \Erw_{i}\left(\sum_{n\ge 1}\tau_{n}(i)^{\alpha}\,\1_{\{S_{\tau_{n}(i)}\le y,\min_{1\le k\le\tau_{n}(i)}S_{k}>-x\}}\right)
\end{align}
for any $\alpha>0$.
\end{Lemma}

\begin{proof}
We begin with some preliminary considerations. Pick $j\in\cS$ (possibly $=i$) such that $\Prob_{j}(\sle=\infty)>0$. By Proposition \ref{prop:ladder chain}, $j$ is a recurrent state for the dual ladder chain $({}^{\#}\Mgn)_{n\ge 0}$. Hence, if $\kappa_{m}$ denotes the $m^{th}$ strictly ascending ladder epoch of $({}^{\#}S_{n})_{n\ge 0}$ with ${}^{\#}M_{\kappa_{m}}=j$ for each $m\in\N$, then these epochs are all $\Prob_{i}$-a.s. finite and $({}^{\#}S_{\kappa_{n}})_{n\ge 0}$ forms a subsequence of $({}^{\#}S_{{}^{\#}\tg_{n}(j)})_{n\ge 0}$ and an ordinary RW with positive increments. Moreover, 
$$ \kappa_{1}\ =\ {}^{\#}\tg_{\vth}(j) $$
for a stopping time $\vth$ with respect to the filtration
$$ \sigma\left({}^{\#}\tg_{1}(j),...,{}^{\#}\tg_{n}(j),\,({}^{\#}M_{k},{}^{\#}S_{k})_{1\le k\le{}^{\#}\tg_{n}(j)}\right),\quad n\ge 0, $$
and $\Erw_{j}\vth\le\Erw_{j}\kappa_{1}<\infty$. As a consequence, by using Wald's identity,
\begin{align*}
\Erw_{j}\left({}^{\#}S_{{}^{\#}\tau(j)}^{+}\wedge y\right)\ &\le\ \Erw_{j}\left({}^{\#}S_{\kappa_{1}}\wedge y\right)\\
&\le\ \Erw_{j}\left(\sum_{k=1}^{\vth}\left(({}^{\#}S_{{}^{\#}\tg_{k}(j)}-{}^{\#}S_{{}^{\#}\tg_{k-1}(j)})\wedge y\right)\right)\\
&=\ \Erw_{j}({}^{\#}S_{{}^{\#}\tg(j)}\wedge y)\,\Erw_{j}\vth,
\end{align*}
giving
\begin{equation}\label{eq:Erw S_kappa}
\Erw_{j}\left({}^{\#}S_{\kappa_{1}}\wedge y\right)\ \asymp\ \Erw_{j}\big({}^{\#}S_{{}^{\#}\tau(j)}^{+}\wedge y\big)\ =\ \Erw_{j}\big(S_{\tau(j)}^{+}\wedge y\big)
\end{equation}
when recalling Lemma \ref{lem:trunc mean}(e) and observing that $S_{\tau(j)},{}^{\#}S_{{}^{\#}\tau(j)}$ have the same distribution under $\Prob_{j}$. Finally, we infer with the help of \eqref{eq:harmonic series asymptotics, alpha>0} for the ordinary RW $({}^{\#}S_{\kappa_{m}})_{m\ge 0}$ that
\begin{align}\label{eq:renewal measure S_kappa}
\sum_{m\ge 1}m^{\alpha}\,\Prob_{j}\left({}^{\#}S_{\kappa_{m}}\le y\right)\ &\asymp\ \left(\frac{y}{\Erw_{j}\left({}^{\#}S_{\kappa_{1}}\wedge y\right)}\right)^{1+\alpha}\ \asymp\ J_{j}(y)^{1+\alpha}
\end{align}
as $y\to\infty$. 

\vspace{.2cm}
Now we are ready to prove \eqref{eq:crucial series estimate}. Since $\Prob_{i}(\sle(-x)=\infty)>0$, there exist $x_{1}\in\R_{\geqslant}$ and $n_{1},n_{2}\in\N$ such that
$$ E\ :=\ \left\{\min_{1\le k\le n_{1}}S_{k}>-x,\,S_{n_{1}}\le x_{1},\,M_{n_{1}}=j,\,\tau_{n_{2}}\le n_{1}<\tau_{n_{2}+1}\right\} $$
has positive probability under $\Prob_{i}$.
\begin{align*}
\Erw_{i}&\left(\sum_{n\ge 1}\tau_{n}^{\alpha}\,\1_{\{S_{\tau_{n}}\le y,\min_{1\le k\le\tau_{n}}S_{k}>-x\}}\right)\\
&\ge\ \Erw_{i}\left(\1_{E}\sum_{n>n_{2}}\tau_{n}^{\alpha}\,\1_{\{S_{\tau_{n}}-S_{n_{1}}\le y-x_{1},\min_{n_{1}<k\le\tau_{n}}(S_{k}-S_{n_{1}})>0\}}\right)\\
&\ge\ \Prob_{i}(E)\,\Erw_{j}\left(\sum_{n\ge 1}\tau_{n}^{\alpha}\,\1_{\{S_{\tau_{n}}\le y-x_{1},\min_{1\le k\le\tau_{n}}S_{k}>0\}}\right)\\
&\gtrsim\ \sum_{m,n\ge 1}m^{\alpha}a_{m,n}
\shortintertext{where}
&a_{m,n}\ :=\ \Prob_{j}\big(S_{\tau_{n}}\le y-x_{1},\min_{1\le k\le\tau_{n}}S_{k}>0,\,\tau_{m}(j)\le\tau_{n}<\tau_{m+1}(j)\big).
\end{align*}
Now use the duality relation (see \eqref{eq:duality relation})
\begin{align*}
a_{m,n}\ &=\ \frac{\pi_{i}}{\pi_{j}}\,\Prob_{i}\big(y-x_{1}\ge{}^{\#}S_{^{\#}\tau_{m+1}(j)}>{}^{\#}S_{k}\text{ for }0\le k<{}^{\#}\tau_{m+1}(j),\\
&\hspace{6.3cm}{}^{\#}\tau_{n-1}\le{}^{\#}\tau_{m+1}(j)<{}^{\#}\tau_{n}\big)\\
&=\ \frac{\pi_{i}}{\pi_{j}}\sum_{l=1}^{m+1}\Prob_{i}\big({}^{\#}S_{\kappa_{l}}\le y-x_{1},{}^{\#}\tau_{n-1}\le\kappa_{l}<{}^{\#}\tau_{n},{}^{\#}\tau_{m+1}(j)=\kappa_{l}\big)
\end{align*}
and choose $x_{2}$ so large that $\Prob_{i}({}^{\#}S_{\kappa_{1}}\le x_{2})>0$. Then
\begin{align*}
\sum_{m,n\ge 1}m^{\alpha}a_{m,n}\ &\ge\ \frac{\pi_{i}}{\pi_{j}}\sum_{m,n\ge 1}\sum_{l=1}^{m+1}l^{\alpha}\,\Prob_{i}\big({}^{\#}S_{\kappa_{l}}\le y-x_{1},{}^{\#}\tau_{n-1}\le\kappa_{l}<{}^{\#}\tau_{n},{}^{\#}\tau_{m+1}(j)=\kappa_{l}\big)\\
&=\ \frac{\pi_{i}}{\pi_{j}}\sum_{l\ge 1}l^{\alpha}\,\Prob_{i}\big({}^{\#}S_{\kappa_{l}}\le y-x_{1}\big)\\
&\gtrsim\ \sum_{l\ge 1}l^{\alpha}\,\Prob_{j}\big({}^{\#}S_{\kappa_{l}}\le y-x_{1}-x_{2}\big)\ \asymp\ J_{j}(y-x_{1}-x_{2}),
\end{align*}
where \eqref{eq:renewal measure S_kappa} has been used for the last relation. The proof is herewith complete because $J_{j}(y-x_{1}-x_{2})\asymp J_{j}(y)\asymp J_{i}(y)$ as $y\to\infty$ by Lemma \ref{lem:trunc mean}.\qed
\end{proof}

\subsection{Proof of Theorem \ref{thm:Kesten-Maller MRW}}\label{subsec:proof Kesten-Maller}

We have organized the proof of the theorem as follows: After the auxiliary Lemma \ref{lem:T>tau_n formula}, Lemmata \ref{lem:K/M (a) implies (b)} and \ref{lem:series involving D^{i,>}} will establish ``(a)$\RA$(b)'' and its converse, respectively, the latter even without the assumption that $\Erw_{i}\tau(i)^{1+\alpha}<\infty$. Then ``(c)$\RA$(d)'' will be shown by Lemma \ref{lem:K/M (c) implies (d)} and ``(d)$\RA$(a)'' by Lemma \ref{lem:K/M (d) implies (a)}. Since (b)$\RA$(c)'' is clear upon noting that $\sigma_{\min}\le\rho(0)$, this completes the proof of the equivalence of (a)--(d).

\begin{Lemma}\label{lem:T>tau_n formula}
Let $(M_{n},S_{n})_{n\ge 0}$ be a nontrivial MRW with $\Erw_{i}\tau(i)^{1+\alpha}<\infty$ and $T$ an arbitrary nonnegative random variable.
Then
$$ \Erw_{i}T^{\alpha}<\infty\quad\text{iff}\quad\sum_{n\ge 1}n^{\alpha-1}\,\Prob_{i}(T>\tau_{n}(i))<\infty. $$
\end{Lemma}

\begin{proof}
Use \eqref{eq:Spitzer condition tau_n} in the Appendix (cf. proof of Lemma \ref{lem:series tau_n(s)}) to infer
\begin{align*}
\Erw_{i}T^{\alpha}\ &\asymp\ \sum_{n\ge 1}n^{\alpha-1}\,\Prob_{i}\left(T>\frac{n}{2\,\Erw_{i}\tau},\,\tau_{n}\le 2n\,\Erw_{i}\tau\right)\\
&\asymp\ \sum_{n\ge 1}n^{\alpha-1}\,\Prob_{i}\left(T>\tau_{n},\,\tau_{n}\le 2n\,\Erw_{i}\tau\right)\\
&\asymp\ \sum_{n\ge 1}n^{\alpha-1}\,\Prob_{i}(T>\tau_{n}),
\end{align*}
and thus the asserted equivalence.\qed
\end{proof}

\begin{Lemma}\label{lem:K/M (a) implies (b)}
Let $(M_{n},S_{n})_{n\ge 0}$ be a positive divergent MRW with $\Erw_{i}\tau(i)^{1+\alpha}<\infty$. Then $\Erw_{i}J_{i}(D^{i})^{1+\alpha}<\infty$ implies $\Erw_{i}\rho(x)^{\alpha}<\infty$ for all $x\in\R_{\geqslant}$ and then also $\Erw_{j}\rho(x)^{\alpha}<\infty$ for all $(j,x)\in\cS\times\R_{\geqslant}$.
\end{Lemma}

\begin{proof}
Under the stated assumptions, $(S_{n})_{n\ge 0}$ is positive divergent and either
\begin{align}
&\Erw_{i}S_{\tau}^{+}\ <\ \infty\quad\text{and}\quad\Erw_{i} J_{i}(D^{i})^{1+\alpha}\ \asymp\ \Erw_{i}(D^{i})^{1+\alpha}\ <\ \infty,\label{eq:K/M (b), case 1}
\shortintertext{or}
&\hspace{2.5cm}\Erw_{i}J_{i}(D^{i})^{1+\alpha}\ <\ \infty\ =\ \Erw_{i}S_{\tau}^{+}.\label{eq:K/M (b), case 2}
\end{align}
These two cases will be treated separately hereafter.

\vspace{.2cm}
Suppose \eqref{eq:K/M (b), case 1} be true, in particular $0<\lim_{y\to\infty}A_{i}(y)=\Erw_{i}S_{\tau}<\infty$, for $(S_{\tau_{n}})_{n\ge 0}$ is positive divergent. Put $\mu:=\Erw_{i}S_{\tau}/(2\,\Erw_{i}\tau)$. Then
$$ \Erw_{i}(S_{\tau}-\mu\tau)\ <\ \infty, $$
and since $\Erw_{i}\tau^{1+\alpha}<\infty$, we have
$$ \Erw_{i}(D^{i}+\mu\tau)^{1+\alpha}\ <\ \infty. $$
Consequently, we may use Theorem \ref{thm:Kesten-Maller MRW min} for the MRW $(M_{n},S_{n}-\mu n)_{n\ge 0}$ to infer
$$ \Erw_{i}\left|\min_{n\ge 0}\left(S_{n}-\mu n\right)\right|^{\alpha}\ <\ \infty. $$
Now $\Erw_{i}\rho(x)<\infty$ for all $x\in\R_{\geqslant}$ follows from
\begin{align*}
\mu\rho(x)\ \le\ x-\left(S_{\rho(x)}-\mu\rho(x)\right)\ \le\ x-\min_{n\ge 0}\left(S_{n}-\mu n\right)
\end{align*}
(see \cite[p.~871,(i)$\RA$(ii)]{Janson:86}).

\vspace{.2cm}
Assuming \eqref{eq:K/M (b), case 2}, it suffices to prove $\Erw_{i}\wh{\rho}_{i}(x)^{\alpha}<\infty$ for all $x\in\R_{\geqslant}$ and
$$ \wh{\rho}_{i}(x)\,:=\,\sup\{n\ge 0:S_{\tau_{n}}-D_{n+1}^{i}\le x\} $$
because $\{\rho(x)>\tau_{n}\}=\{\wh{\rho}_{i}(x)\ge n\}$ for all $n\ge 1$ and an appeal to Lemma \ref{lem:T>tau_n formula}. Put $\wh{X}_{n}:=S_{\tau_{n}}-S_{\tau_{n-1}}$ for $n\ge 1$ and observe that
\begin{align}\label{eq: inequ w eps}
S_{\tau_{n}}-D_{n+1}^{i}\ &=\ \sum_{k=1}^{n}\wh{X}_{k}^{+}-\sum_{k=1}^{n}\wh{X}_{k}^{-}-D_{n+1}^{i}\ \ge\ \sum_{k=1}^{n}\wh{X}_{k}^{+}-\sum_{k=1}^{n+1}D_{k}^{i}\nonumber\\
&\ge\ \sum_{\eps\in\{0,1\}}\left(\sum_{k=1}^{n}R_{k}^{\eps}+(R_{n+1}^{\eps})^{-}\right),
\shortintertext{where}
&\hspace{-1.8cm}R_{n}^{0}\ :=\ \wh{X}_{k}^{+}\1_{\{\theta_{k}=1\}}-D_{k}^{i}\1_{\{\theta_{k}=0\}}\quad\text{and}\quad R_{n}^{1}\ :=\ \wh{X}_{k}^{+}\1_{\{\theta_{k}=0\}}-D_{k}^{i}\1_{\{\theta_{k}=1\}}\nonumber
\end{align}
for a sequence $(\theta_{n})_{n\ge 1}$ of iid symmetric Bernoulli variables which are independent of all other occurring random variables. Notice that $(\sum_{k=1}^{n}R_{k}^{\eps})_{n\ge 0}$ forms an ordinary random walk for each $\eps\in\{0,1\}$. We claim that its last level $x$ exit time $\rho_{\eps}(x)$ is the same as 
$$ \rho_{\eps}'(x)\ :=\ \sup\left\{n\ge 0:\sum_{k=1}^{n}R_{k}^{\eps}-(R_{n+1}^{\eps})^{-}\le x\right\}. $$
Indeed, $\rho_{\eps}'(x)\ge\rho_{\eps}(x)$ is obvious. For the reverse inequality, suppose $\rho_{\eps}'(x)=n$ and thus
\begin{equation}\label{eq: w eps}
\sum_{k=1}^{n}R_{k}^{\eps}-(R_{n+1}^{\eps})^{-}\ \le\ x.
\end{equation}
Then $R_{n+1}^{\eps}\ge 0$ entails $\rho_{\eps}(x)\ge n$, while $R_{n+1}^{\eps}<0$ entails
$x\ge\sum_{k=1}^{n}R_{k}^{\eps}-(R_{n+1}^{\eps})^{-}=\sum_{k=1}^{n+1}R_{k}^{\eps}$ and therefore $\rho_{\eps}(x)\ge n+1$.

\vspace{.1cm}
Now \eqref{eq: inequ w eps} implies $\rho(x)\le\rho_{0}(x)\vee\rho_{1}(x)$. Since
$$ (R_{1}^{\eps})^{+}\ =\ S_{\tau}^{+}\1_{\{\theta_{1}=\eps\}}\quad\text{and}\quad (R_{1}^{\eps})^{-}\ =\ D_{1}^{i}\1_{\{\theta=1-\eps\}}, $$
we have, by \eqref{eq:K/M (b), case 2}, that
\begin{align*}
C(\beta)\ :=\ \int\left(\frac{y}{\Erw_{i}((R_{1}^{\eps})^{+}\wedge y)}\right)^{1+\beta}\ \Prob_{i}((R_{1}^{\eps})^{-}\in ds)\ <\ \infty
\end{align*}
for any $\beta\in [0,\alpha]$ and $\Erw_{i}|R_{1}^{\eps}|\ge\Erw_{i}(R_{1}^{\eps})^{+}=\Erw_{i}S_{\tau}^{+}/2=\infty$. The latter in combination with $C(0)<\infty$ ensures the positive divergence of $(\sum_{k=1}^{n}R_{k}^{\eps})_{n\ge 0}$ for each $\eps\in\{0,1\}$ as pointed out in Remark \ref{rem:Erickson's extension}. Consequently, $\Erw_{i}\rho(x)^{\alpha}\le\Erw_{i}(\rho_{0}(x)\vee\rho_{1}(x))^{\alpha}<\infty$, and the extension of the last conclusion to all $i\in\cS$ follows because $\Erw_{i}\tau(i)^{1+\alpha}<\infty$ and $\Erw_{i}J_{i}(D^{i})^{1+\alpha}<\infty$ are all solidarity properties (for the last two assertion use Lemma  \ref{lem:solidarity D^s} with $\gamma=1$).\qed
\end{proof}

\begin{Lemma}\label{lem:series involving D^{i,>}}
If $\Erw_{i}\rho(x)^{\alpha}<\infty$ for some $(i,x)\in\cS\times\R_{\geqslant}$ and $\alpha>0$, then $\Erw_{i}J_{i}(D^{i})^{1+\alpha}<\infty$.
\end{Lemma}

\begin{proof}
Since $(S_{n})_{n\ge 0}$ is positive divergent under the proviso, there exists $x_{0}>0$ such that $\Prob_{i}(\sle(-x_{0})=\infty)=:p>0$. Plainly, $\Erw_{i}\rho(-x_{0})^{\alpha}\le\Erw_{i}\rho(x)^{\alpha}<\infty$. Use
\begin{align*}
&\{\rho(-x_{0})>n/2\}\ \supset\ \bigcup_{n/2<k\le n}\left\{S_{\tg_{k}}\le D_{k+1}^{i,>}-x_{0}\right\}\\
&\quad\supset\ \bigcup_{n/2<k\le n}\left\{S_{\tg_{k}}\le D_{k+1}^{i,>}-x_{0},\,\inf_{l\ge k+1}(S_{\tg_{l}}-D_{l+1}^{i,>})>-x_{0}\right\}\\
&\quad\supset\ \bigcup_{n/2<k\le n}\left\{S_{\tg_{k}}\le D_{k+1}^{i,>}-x_{0},\,\inf_{l\ge\tg_{k+1}}(S_{l}-S_{\tg_{k+1}})>-x_{0}\right\}
\end{align*}
to infer that
$$ \Prob_{i}(\rho(-x_{0})>n/2)\ \ge\ \frac{pn}{3}\,\Prob_{i}\big(S_{\tg_{n}}\le D_{n+1}^{i,>}-x_{0}\big) $$
and thereupon
\begin{align*}
\sum_{n\ge 1}n^{\alpha}\,\Prob_{i}\big(S_{\tg_{n}}\le D_{n+1}^{i,>}-x_{0}\big)\ \lesssim\ \sum_{n\ge 1}n^{\alpha-1}\,\Prob_{i}(\rho(-x_{0})>n/2)\ <\ \infty,
\end{align*}
as $\Erw_{i}\rho(-x_{0})^{\alpha}<\infty$. Now use \eqref{eq:harmonic series asymptotics, alpha>0} for the ordinary RW $(S_{\tg_{n}})_{n\ge 0}$ to infer
$$ \sum_{n\ge 1}n^{\alpha}\,\Prob_{i}\left(S_{\tg_{n}}\le y\right)\ \asymp\ J_{i}^{>}(y)^{1+\alpha} $$
and thereupon
\begin{align*}
\infty\ >\ \sum_{n\ge 1}n^{\alpha}\,\Prob_{i}\big(S_{\tg_{n}}\le D_{n+1}^{i,>}-x_{0}\big)\ \asymp\ \int_{[x_{0},\infty)}J_{i}^{>}(y-x_0)^{1+\alpha}\ \Prob_{i}(D^{i,>}\in dy).
\end{align*}
Finally, Lemma \ref{lem:trunc mean}(e) provides us with
$$ J_{i}^{>}(y-x_{0})\ \asymp\ J_{i}(y) $$
for all $y\in\R_{\geqslant}$, and since these functions are nondecreasing in $x$ and $D^{i}\le D^{i,>}$, we see that $\Erw_{i}J_{i}(D^{i})^{1+\alpha}<\infty$ as claimed.\qed
\end{proof}

\begin{Lemma}\label{lem:moments sigma_min}
For any $i\in\cS$, $\Erw_{i}\rho(0)^{\alpha}<\infty$ implies $\Erw_{i}\sigma_{\min}^{\alpha}<\infty$.
\end{Lemma}

\begin{proof}
The assertion follows directly from $\rho(S_{\tau})-\tau\eqdist\rho(0)$ under $\Prob_{i}$ and $\rho(S_{\tau})\ge\sigma_{\min}$ $\Prob_{i}$-a.s.\qed
\end{proof}

\begin{Lemma}\label{lem:K/M (c) implies (d)}
Let $\alpha>0$ and $i\in\cS$. Then $\Erw_{i}\sigma_{\min}^{\alpha}<\infty$ implies 
$$ \Erw_{i}\sle(-x)^{\alpha}\1_{\{\sle(-x)<\infty\}}\ <\ \infty $$
for all $x\in\R_{\geqslant}$ as well as $\Prob_{i}(\sle(-x)=\infty)>0$ for all sufficiently large $x$.
\end{Lemma}

\begin{proof}
The first assertion follows from the obvious inequality
$$ \sle(-x)^{\alpha}\1_{\{\sle(-x)<\infty\}}\ \le\ \sigma_{\min}\1_{\{\sle(-x)<\infty\}} $$
\end{proof}
for any $x\in\R_{\geqslant}$, while the second one must hold because $\sigma_{\min}<\infty$ $\Prob_{i}$-a.s. entails the positive divergence of $(S_{n})_{n\ge 0}$.\qed

\begin{Lemma}\label{lem:K/M (d) implies (a)}
If, for some $(i,x)\in\cS\times\R_{\geqslant}$, $\Erw_{i}\sle(-x)^{\alpha}\1_{\{\sle(-x)<\infty\}}<\infty$ and $\Prob_{i}(\sle(-x)=\infty)>0$, then $A_{i}(y)>0$ for all sufficiently large $y$ and $\Erw_{i}J_{i}(D^{i})^{1+\alpha}<\infty$.
\end{Lemma}

\begin{proof}
First, $\Prob_{i}(\sle(-x)=\infty)>0$ implies $S_{n}\to\infty$ $\Prob_{i}$-a.s. and thus $A_{i}(y)>0$ for all sufficiently large $y$ by Theorem \ref{thm:Spitzer-Erickson MRW}. Next, we infer
\begin{align*}
\infty\ &>\ \Erw_{i}\sle(-x)^{\alpha}\1_{\{\sle(-x)<\infty\}}\ =\ \Erw_{i}\left(\sum_{n\ge 1}n^{\alpha}\,\1_{\{\sle(-x)=n\}}\right)\\
&\ge\ \Erw_{i}\left(\sum_{n\ge 1}\sum_{k=\tau_{n}+1}^{\tau_{n+1}}k^{\alpha}\,\1_{\{\sle(-x)=k\}}\right)\\
&\ge\ \Erw_{i}\left(\sum_{n\ge 1}\tau_{n}^{\alpha}\,\1_{\{\tau_{n}<\sle(-x)\le\tau_{n+1}\}}\right)\\
&=\ \int\Erw_{i}\left(\sum_{n\ge 1}\tau_{n}^{\alpha}\,\1_{\{S_{\tau_{n}}\le y-x,\min_{1\le k\le\tau_{n}}S_{k}>-x\}}\right)\,\Prob_{i}(D^{i}\in dy),
\end{align*}
and then $\Erw_{i}J_{i}(D^{i})^{1+\alpha}<\infty$ by an appeal to Lemma \ref{lem:crucial series estimate}.\qed
\end{proof}

\subsection{Proof of Theorem \ref{thm:K/M MRW Spitzer series}}\label{subsec:proof K/M MRW Spitzer series}

We first point out that, if $\Sigma_{\alpha}(i,0)<\infty$ for some $i\in\cS$, then $\Sigma_{\alpha}(j,0)<\infty$ for all $j\in\cS$. To see this, we note that
\begin{align}\label{eq:Spitzer series embedded}
\sum_{n\ge 1}n^{\alpha-1}\Prob_{i}(S_{\tau_{n}(i)}\le x)\ \asymp\ \Erw_{i}\left(\sum_{n\ge 1}\tau_{n}(i)^{\alpha-1}\1_{\{S_{\tau_{n}(i)}\le x\}}\right)\ \le\ \Sigma_{\alpha}(i,x)\ <\ \infty
\end{align}
(see Lemma \ref{lem:solidarity increments}(b) in the Appendix) in combination with Theorem \ref{thm:Kesten-Maller} implies $A_{i}(x)>0$ for all sufficiently large $x$ and $\Erw_{i}J_{i}(S_{\tau(i)}^{-})<\infty$. In particular, we obtain $\Erw_{i}\tg(i)^{1+\alpha}<\infty$ and then also $\Erw_{j}\tg(i)^{1+\alpha}<\infty$ by a straightforward argument. Now,
\begin{align*}
\Sigma_{\alpha}(i,0)\ &\le\ \Erw_{j}\left(\sum_{n=1}^{2 \tg(i)}n^{\alpha-1} + \sum_{n\ge 1} (\tg(i)+n)^{\alpha-1}\, \1_{\{(S_{\tg(i)+n}-S_{\tg(i)})+S_{\tg(i)}\leq 0\}}\right)\\
&\lesssim\ \Erw_{j}\tg(i)^{1+\alpha}\ +\ \Erw_{j}\left( \sum_{n\ge 1} (2n)^{\alpha-1}\,\1_{\{S_{\tg(i)+n}-S_{\tg(i)}\leq 0\}}\right)\\
&\le\ \Erw_{j}\tg(i)^{1+\alpha}\ +\ 2^{\alpha-1}\Sigma_{\alpha}(i,0)
	 \end{align*}

In order to prove Theorem \ref{thm:K/M MRW Spitzer series}, it is therefore enough to show that, for any fixed $(i,x)\in\cS\times\R_{\geqslant}$ (and with $\tau_{n}=\tau_{n}(i)$ and $\tg=\tg(i)$),
\begin{align}\label{eq:Spitzer series approx V_{1}-integral}
\Sigma_{\alpha}(i,x)\ \asymp\ \int\sum_{n\ge 1}n^{\alpha-1}\Prob_{i}(S_{\tau_{n}}\le x+y)\ \V_{i}(dy)
\end{align}
where 
\begin{align*}
\sum_{n\ge 1}n^{\alpha-1}\Prob_{i}(S_{\tau_{n}}\le x+y)\ \asymp\ 
\begin{cases}
\log J_{i}(x+y)\,\asymp\,\log J_{i}(y),&\text{if }\alpha=0,\\
\hfill J_{i}(x+y)^{\alpha}\,\asymp\,J_{i}(y)^{\alpha},&\text{if }\alpha>0,
\end{cases}
\end{align*}
as $y\to\infty$ should be recalled (see Lemma \ref{lem:solidarity increments}, \eqref{eq:harmonic series asymptotics, alpha>0} and the subsequent remark). Note that we may replace $\V_{i}$ with $\W_{i}:=\Erw_{i}\big(\sum_{n=1}^{\tau(i)}\1_{\{-S_{n}\in\cdot\}}\big)$ in \eqref{eq:Spitzer series approx V_{1}-integral}, for
\begin{align*}
\int &\sum_{n\ge 1}n^{\alpha-1}\Prob_{i}(S_{\tau_{n}}\le x+y)\ \V_{i}(dy)\\
&\le\ \int\sum_{n\ge 1}n^{\alpha-1}\Prob_{i}(S_{\tau_{n}}\le x+y)\ \W_{i}(dy)\ +\ \V_{i}(\{0\})\sum_{n\ge 1}n^{\alpha-1}\Prob_{i}(S_{\tau_{n}}\le x)\\
&\le\ \int\sum_{n\ge 1}n^{\alpha-1}\Prob_{i}(S_{\tau_{n}}\le x+y)\ \V_{i}(dy)\ +\ \V_{i}(\{0\})\sum_{n\ge 1}n^{\alpha-1}\Prob_{i}(S_{\tau_{n}}\le x)\\
&\le\ 2\int\sum_{n\ge 1}n^{\alpha-1}\Prob_{i}(S_{\tau_{n}}\le x+y)\ \V_{i}(dy).
\end{align*}
We now distinguish between the cases $0\le\alpha\le 1$ and $\alpha>1$ and put $\chi_{n}:=\tau_{n}-\tau_{n-1}$ for $n\ge 1$ which are iid copies of $\tau=\tau(i)$ under $\Prob_{i}$.

\vspace{.2cm}
\textsc{Case 1}. $0\le\alpha\le 1$.\\[1mm]
Using $(\tau_{n-1}+k)^{\alpha-1}\le n^{\alpha-1}$ for $k=1,...,\chi_{n}$ and $n\ge 1$, we find
\begin{align*}
\Sigma_{\alpha}(i,x)\ &=\ \Erw_{i}\left(\sum_{n\ge 1}\sum_{k=1}^{\chi_{n}}(\tau_{n-1}+k)^{\alpha-1}\1_{\{S_{\tau_{n-1}+k}\le x\}}\right)\\
&\le\ \sum_{n\ge 1}n^{\alpha-1}\,\W_{i}([y-x,\infty))\ \Prob_{i}(S_{\tau_{n-1}}\in dy)\\
&=\ \int\sum_{n\ge 1}n^{\alpha-1}\,\Prob_{i}(S_{\tau_{n-1}}\le x+y)\ \W_{i}(dy)\\
&\lesssim\ \int\sum_{n\ge 1}n^{\alpha-1}\,\Prob_{i}(S_{\tau_{n}}\le x+y)\ \V_{i}(dy).
\end{align*}
For the reverse inequality note first that
\begin{align*}
\Erw_{i}&\left(\sum_{n\ge 1}\sum_{k=1}^{\chi_{n}}(\tau_{n-1}+k)^{\alpha-1}\1_{\{\tau_{n-1}>2n\,\Erw_{i}\tau\}}\right)\\
&\lesssim\ \Erw_{i}\left(\sum_{n\ge 1}n^{\alpha-1}\chi_{n}\,\1_{\{\tau_{n-1}>2n\,\Erw_{i}\tau\}}\right)\\
&=\ \Erw_{i}\tau\sum_{n\ge 1}n^{\alpha-1}\,\Prob_{i}\left(\tau_{n-1}>2n\,\Erw_{i}\tau\right)\ <\ \infty,
\end{align*}
the finiteness of the last series following from \eqref{eq:Spitzer condition tau_n} in the Appendix. We also need that
\begin{align}
\begin{split}\label{eq:moment of tau condition}
\Erw_{i}&\left(\sum_{n\ge 1}n^{\alpha-1}\chi_{n}\1_{\{\chi_{n}>n\}}\right)\ =\ \Erw_{i}\left(\sum_{n\ge 1}\sum_{k=1}^{\chi_{n}}n^{\alpha-1}\1_{\{\chi_{n}>n\}}\right)\\
&\le\ \Erw_{i}\left(\sum_{n\ge 1}\sum_{k=1}^{n}n^{\alpha-1}\1_{\{\chi_{n}>n\}}\right)\ +\ \Erw_{i}\left(\sum_{n\ge 1}\sum_{k\ge n}n^{\alpha-1}\1_{\{\chi_{n}>k\}}\right)\\
&\lesssim\ \sum_{n\ge 1}n^{\alpha}\,\Prob_{i}(\tau>n)\ +\ \sum_{k\ge 1}\sum_{n=1}^{k}n^{\alpha-1}\,\Prob_{i}(\tau>k)\\
&\lesssim\ \Erw_{i}\tau^{1+\alpha}\vee\Erw_{i}(\tau\log\tau)<\ \infty,
\end{split}
\end{align}
where $\sum_{k=1}^{n}k^{-1}\asymp\log n$ and $\sum_{k=1}^{n}n^{\alpha-1}\asymp n^{\alpha}$ for any $\alpha>0$ has been utilized for the final estimate. With this at hand, we infer
\begin{align*}
\Sigma_{\alpha}(i,x)\ &=\ \Erw_{i}\left(\sum_{n\ge 1}\sum_{k=1}^{\chi_{n}}(\tau_{n-1}+k)^{\alpha-1}\1_{\{S_{\tau_{n-1}+k}\le x\}}\right)\\
&\ge\ \Erw_{i}\left(\sum_{n\ge 1}\sum_{k=1}^{\chi_{n}}(\tau_{n-1}+k)^{\alpha-1}\1_{\{S_{\tau_{n-1}+k}\le x,\,\tau_{n-1}\le 2n\,\Erw_{i}\tau\}}\right)\\
&\gtrsim\ \Erw_{i}\left(\sum_{n\ge 1}\sum_{k=1}^{\chi_{n}}(n+k)^{\alpha-1}\1_{\{S_{\tau_{n-1}+k}\le x,\,\chi_{n}\le n\}}\right)\\
&\gtrsim\ \sum_{n\ge 1}n^{\alpha-1}\,\Erw_{i}\left(\sum_{k=1}^{\chi_{n}}\1_{\{S_{\tau_{n-1}+k}\le x\}}\right)\qquad\text{[here \eqref{eq:moment of tau condition} enters]}\\
&\asymp\ \int\sum_{n\ge 1}n^{\alpha-1}\,\Prob_{i}(S_{\tau_{n}}\le x+y)\ \V_{i}(dy)
\end{align*}
as required.

\vspace{.2cm}
\textsc{Case 2}. $\alpha>1$.\\[1mm]
Here $(\tau_{n-1}+k)^{\alpha-1}\ge n^{\alpha-1}$ for $k=1,...,\chi_{n}$ and $n\ge 1$ implies
\begin{align*}
\Sigma_{\alpha}(i,x)\ &=\ \Erw_{i}\left(\sum_{n\ge 1}\sum_{k=1}^{\chi_{n}}(\tau_{n-1}+k)^{\alpha-1}\1_{\{S_{\tau_{n-1}+k}\le x\}}\right)\\
&\ge\ \int\sum_{n\ge 1}n^{\alpha-1}\,\Prob_{i}(S_{\tau_{n-1}}\le x+y)\ \W_{i}(dy)\\
&\gtrsim\ \int\sum_{n\ge 1}n^{\alpha-1}\,\Prob_{i}(S_{\tau_{n}}\le x+y)\ \V_{i}(dy).
\end{align*}
For the reverse estimation, we embark on the inequality
\begin{align*}
&\Erw_{i}\left(\sum_{n\ge 1}\sum_{k=1}^{\chi_{n}}(\tau_{n-1}+k)^{\alpha-1}\1_{\{S_{\tau_{n-1}+k}\le x\}}\right)\ \lesssim\ I_{1}+I_{2},
\shortintertext{where}
&I_{1}\ :=\ \Erw_{i}\left(\sum_{n\ge 1}\sum_{k=1}^{\chi_{n}}(n+k)^{\alpha-1}\1_{\{S_{\tau_{n-1}+k}\le x\}}\right)
\shortintertext{and}
&I_{2}\ :=\ \Erw_{i}\left(\sum_{n\ge 1}\sum_{k=1}^{\chi_{n}}(\tau_{n-1}+k)^{\alpha-1}\1_{\{\tau_{n-1}>2n\,\Erw_{i}\tau\}}\right).
\end{align*}
As for $I_{1}$, we then obtain
\begin{align*}
I_{1}\ &\lesssim\ \Erw_{i}\left(\sum_{n\ge 1}n^{\alpha-1}\sum_{k=1}^{\chi_{n}}\1_{\{S_{\tau_{n-1}+k}\le x\}}\right)\ +\ \Erw_{i}\left(\sum_{n\ge 1}\sum_{k=1}^{\chi_{n}}(n+k)^{\alpha-1}\1_{\{S_{\tau_{n-1}+k}\le x,\,\chi_{n}>n\}}\right)\\
&\lesssim\ \int\sum_{n\ge 1}n^{\alpha-1}\,\Prob_{i}(S_{\tau_{n}}\le x+y)\ \V_{i}(dy)\ +\ \Erw_{i}\left(\sum_{n\ge 1}\chi_{n}^{\alpha}\1_{\{\chi_{n}>n\}}\right)\\
&\lesssim\ \int\sum_{n\ge 1}n^{\alpha-1}\,\Prob_{i}(S_{\tau_{n}}\le x+y)\ \V_{i}(dy)\ +\ \Erw_{i}\tau^{\alpha+1}.
\end{align*}
Finally, since $(\tau_{n-1}+k)^{\alpha-1}\le 2^{\alpha-1}(\tau_{n-1}^{\alpha-1}+\chi_{n}^{\alpha-1})$ for $k=1,...,\chi_{n}$ and $n\ge 1$,
\begin{align*}
I_{2}\ &\lesssim\ \Erw_{i}\left(\sum_{n\ge 1}\sum_{k=1}^{\chi_{n}}(\tau_{n-1}^{\alpha-1}+\chi_{n}^{\alpha-1})\1_{\{\tau_{n-1}>2n\,\Erw_{i}\tau\}}\right)\\
&\lesssim\ \Erw_{i}\left(\sum_{n\ge 1}\tau_{n}^{\alpha-1}\1_{\{\tau_{n}>2n\,\Erw_{i}\tau\}}\right)\ +\ \Erw_{i}\left(\sum_{n\ge 1}\chi_{n}^{\alpha}\,\1_{\{\tau_{n-1}>2n\,\Erw_{i}\tau\}}\right),
\end{align*}
and the last two expectations are finite because $\Erw_{i}\tau^{\alpha+1}<\infty$ (for the second expectation this is obvious, for the first one see Lemma \ref{lem:series tau_n(s)} in the Appendix).

\subsection{Proof of Theorem \ref{thm:K/M MRW N(x)}}\label{subsec:proof K/M MRW N(x)}

Here we start with two auxiliary lemmata.

\begin{Lemma}\label{lem:solidarity N(x)}
Let $\alpha>0$ and $\Erw_{i}\tau(i)^{\alpha}<\infty$ for some/all $i\in\cS$. Then the set
$$ \{(i,x)\in\cS\times\R_{\geqslant}:\Erw_{i}N(x)^{\alpha}<\infty\} $$ 
is either empty or equal to $\cS\times\R_{\geqslant}$.
\end{Lemma}

\begin{proof}
Suppose that $\Erw_{i}N(0)^{\alpha}<\infty$ for some $i\in\cS$, so that in particular
$$ \Erw_{i}\left(\sum_{n\ge 1}\1_{\{S_{\tau_{n}}\le 0\}}\right)^{\alpha}\ <\ \infty, $$
which in turn is equivalent to $\Erw_{i}\nu(x)^{1+\alpha}<\infty$ for all $x\in\R_{\geqslant}$ by Theorem \ref{thm:Kesten-Maller}. Using \cite[Theorem 1.5.1]{Gut:09} for $\alpha\in(0,1)$ and \cite[Theorem 1.5.2]{Gut:09} for $\alpha\ge 1$, we obtain 
$$ \Erw_{i}\tau_{\nu(x)}^{\alpha}\ \lesssim \ \Erw_{i}\tau^{\alpha} \cdot \Erw_{i}\nu(x)^{1\vee \alpha}\ <\ \infty $$
for all $x\in\R_{\geqslant}$. For arbitrary $j\in\cS$, pick $x_{1}\in\R_{\geqslant}$ such that 
$$ 0\ <\ p\ :=\ \Prob_{i}(S_{\tau(j)}\le x_{1},\,\tau\ge\tau(j)). $$
It follows that
$$ \infty\ >\ \Erw_{i} \tau_{\nu(x+x_{1})}^{\alpha}\ \ge \ \Erw_{i} \tau_{\nu(x+x_{1})}^{\alpha}\,\1_{\{S_{\tau(j)}\le x_1,\, \tau\ge\tau(j)\}}\ \ge\ p\,\Erw_{j}\tau_{\nu(x)}^{\alpha}$$
for all $x\in\R_{\geqslant}$. Now put
$$ \widetilde{N}(0)\ :=\ \sum_{n\ge\tau_{\nu(x)}+1} \1_{\{S_{n}-S_{\tau_{\nu(x)}}\le 0\}}, $$
having the law $\Prob_{i}(N(0)\in\cdot)$ under any $\Prob_{j}$, and observe that $N(x)\le\tau_{\nu(x)}+\widetilde{N}(0)$ $\Prob_{j}$-a.s. Therefore, 
$$ \Erw_{j}N(x)^{\alpha}\ \lesssim\ \Erw_{j}\tau_{\nu(x)}^{\alpha}\,+\,\Erw_{i}\, N(0)^{\alpha}\ <\ \infty $$
for all $x\in\R_{\geqslant}$ which completes the proof because $j\in\cS$ was arbitrarily chosen.\qed
\end{proof}

For the next result, recall that $\chi_{n}(i)=\tau_{n}(i)-\tau_{n-1}(i)$ for $n\in\N$.

\begin{Lemma}\label{lem:rho[s]}
For $\alpha>0$, let $\Erw_{i}\tau(i)^{1+\alpha}<\infty$ and $\Erw_{i}J_{i}(S_{\tau(i)}^{-})^{1+\alpha}<\infty$. Then 
\begin{equation*}
\Erw_{i}\left(\sum_{n\ge 1}\chi_{n}(i)\,\1_{\{S_{\tau_{n-1}(i)}\le 0\}}\right)^{\alpha}\ <\ \infty.
\end{equation*}
\end{Lemma}

\begin{proof}
Consider an auxiliary MRW $(M_{n}',S_{n}')_{n\ge 0}$ such that $M'=(M_{n}')_{n\ge 0}$ has state space $\cS'\subset\{0\}\cup\N^{2}$, transition probabilities (with $\tau=\tau(i)$ as usual)
\begin{align*}
p_{0,(n,1)}'=\Prob_{i}(\tau=n),\quad p_{(n,n-1),0}'=1\quad\text{and}\quad p_{(n,k-1),(n,k)}'=1
\end{align*}
for $n\in\N$ and $k=2,...,n-1$, and the stationary distribution $\pi_{0}'=(\Erw_{i}\tau)^{-1}$, $\pi_{(n,k)}'=(\Erw_{i}\tau)^{-1}\Prob_{i}(\tau=n)$ for $n\in\N$ and $k=1,...,n-1$. The conditional increment distributions of $(S_{n}')_{n\ge 0}$ are given by
\begin{align*}
\Prob(S_{1}'\in\cdot|&M_{0}'=j,M_{1}'=k)\\
&=\ 
\begin{cases}
\Prob_{i}(S_{\tau}\in\cdot|\tau=n),&\text{if }j=(n,n-1)\text{ for }n\in\N\text{ and }k=0,\\
\hfill\delta_{0},&\text{otherwise.}
\end{cases}
\end{align*}
Then $\tau':=\inf\{n\ge 1:M_{n}'=0\}$ and $D':=\max_{0\le k\le\tau'}{S_{k}'}^{-}$ satisfy
$$ \bfP_{0}(S_{\tau'}'\in\cdot)\ =\ \Prob_{i}(S_{\tau}\in\cdot)\quad\text{and}\quad\bfP_{0}(D'\in\cdot)\ =\ \bfP_{0}({S_{\tau'}'}\hspace{-3pt}^{-}\in\cdot), $$
where $\bfP_{0}:=\Prob(\cdot|M_{0}'=0)$. Hence, Theorem \ref{thm:Kesten-Maller MRW} for the given $\alpha$ applies to the MRW $(M_{n}',S_{n}')_{n\ge 0}$ and provides us with $\bfE_{0}\rho'(0)^{\alpha}<\infty$ for $\rho'(0):=\sup\{n\ge 0:S_{n}'\le 0\}$. Furthermore, one can easily infer from the definition of the $S_{n}'$ that
\begin{align*}
\Erw_{i}\left(\sum_{n\ge 1}\chi_{n}\,\1_{\{S_{\tau_{n-1}}\le 0\}}\right)^{\alpha}\ \le\ \bfE_{0}\left(\sum_{n\ge 1}\1_{\{S_{n}'\le 0\}}\right)^{\alpha}\ \le\ \bfE_{0}\rho'(0)^{\alpha}\ <\ \infty
\end{align*}
which completes the proof.\qed
\end{proof}

\begin{proof}[of Theorem \ref{thm:K/M MRW N(x)}]
By Lemma \ref{lem:solidarity N(x)}, it suffices to verify the equivalence of $\Erw_{i}N(0)^{\alpha}<\infty$ with $A_{i}(x)>0$ for all large $x$, $\Erw_{i}J_{i}(S_{\tau(i)}^{-})^{1+\alpha}<\infty$ and \eqref{eq:MRW renewal counting} for any fixed $i\in\cS$. Let $\U_{i}$ denote the renewal measure of $(S_{\tau_{n}})_{n\ge 0}$ under $\Prob_{i}$ and note that $\U_{i}((0,x])\asymp J_{i}(x)$ by Lemma \ref{Appendix: asymp} in the Appendix. In the following, the cases $\alpha\le 1$ and $\alpha>1$ are treated separately.

\vspace{.2cm}
\textsc{Case 1}. $0<\alpha\le 1$.\\[2mm]
``(a)$\RA$(b)'' Since
\begin{align*}
N(0)\ \le\ \sum_{n\ge 1}\chi_{n}\1_{\{S_{\tau_{n-1}}\le 0\}}\ +\ \sum_{n\ge 1}\sum_{k=1}^{\chi_{n}}\1_{\{S_{\tau_{n-1}+k}\le 0<S_{\tau_{n-1}}\}}
\end{align*}
it suffices to show by the previous lemma that
\begin{equation}\label{eq:def of term I}
I\ :=\ \Erw_{i}\left(\sum_{n\ge 1}\sum_{k=1}^{\chi_{n}}\1_{\{S_{\tau_{n-1}+k}\le 0<S_{\tau_{n-1}}\}}\right)^{\alpha}\ <\ \infty.
\end{equation}
Using the subadditivity of $x\mapsto x^{\alpha}$, this follows from
\begin{align}
\begin{split}\label{eq:estimate for I}
I\ &\le\ \sum_{n\ge 1}\Erw_{i}\left(\sum_{k=1}^{\chi_{n}}\1_{\{S_{\tau_{n-1}+k}\le 0<S_{\tau_{n-1}}\}}\right)^{\alpha}\ =\ \int_{\R_{>}}\V_{i}^{\alpha}([y,\infty))\ \U_{i}(dy)\\
&=\ \int\U_{i}((0,y])\ \V_{i}^{\alpha}(dy)\ \asymp\ \int J_{i}(y)\ \V_{i}^{\alpha}(dy).
\end{split}
\end{align}

\vspace{.2cm}
\textsc{Case 2}. $\alpha\ge 1$.\\[2mm]
``(b)$\RA$(a)'' If $\Erw_{i}N(0)^{\alpha}<\infty$ and thus \eqref{eq:def of term I} holds, then the superadditivity of $x\mapsto x^{\alpha}$ implies that \eqref{eq:estimate for I} holds with reverse inequality sign, giving \eqref{eq:MRW renewal counting}. Moreover, $N'(0):=\sum_{n\ge 1}\1_{\{S_{\tau_{n}}\le 0\}}\le N(0)$ a.s. entails $\Erw_{i}N'(0)^{\alpha}<\infty$ and thus $A_{i}(x)>0$ for all large $x$ and $\Erw_{i}J_{i}(S_{\tau}^{-})^{1+\alpha}<\infty$ by an appeal to Theorem \ref{thm:Kesten-Maller}. Note that this completes already the proof in the case $\alpha=1$.

\vspace{.2cm}\noindent
``(a)$\RA$(b)'' Regarding the proviso, when stated as $\int_{\R_{>}}\hspace{-2pt}\V_{i}^{\alpha}([x,\infty))\,\U_{i}(dx)<\infty$ (see \eqref{eq:estimate for I}), we point out that it may be extended to
\begin{equation*}
c\ :=\ \int_{\R}\V_{i}^{\alpha}([x,\infty))\ \U_{i}(dx)\ <\ \infty,
\end{equation*}
for the integral over $\R_{\leqslant}$ is bounded by $\Erw_{i}\tau^{\alpha}\,\U_{i}(\R_{\leqslant})<\infty$. As a consequence, we also have
$$ \sup_{\beta\in (0,\alpha]}\int_{\R}\V_{i}^{\beta}([x,\infty))\ \U_{i}(dx)\ =\ c. $$
Let $q\in\N$ and $\delta\in (0,1]$ be such that $\alpha=q+\delta$. The subsequent inductive argument (in $m$) will show that (a) for some $\alpha\le m\in\N$ implies
\begin{equation}\label{eq:E_sN(0)^beta finite}
\Erw_{i}N(0)^{\beta}\ <\ \infty\quad\text{for all }0\le\beta\le\alpha.
\end{equation}
Since this has already been verified for $m=1$, we proceed to the inductive step $m\to m+1$, assume $q=m$ and \eqref{eq:E_sN(0)^beta finite} be true for $\alpha=m$ (inductive hypothesis). The following argument is taken from \cite{AlsIksMei:15} (see their Thm.~3.7). Defining
$$ N_{n}(x)\ :=\ \sum_{k\ge\tau_{n}+1}\1_{\{S_{k}-S_{\tau_{n}}\le x\}}\quad\text{and}\quad L_{n}\ :=\ \sum_{k=1}^{\chi_{n}}\1_{\{S_{\tau_{n-1}+k}\le 0\}} $$
for $n\in\N$ and $x\in\R$, we obviously have
$$ \Erw_{i}\left(\sum_{n\ge 1}L_{n}^{\beta}\right)\ \le\ \Erw_{i}\left(\sum_{n\ge 1}\V_{i}^{\beta}([S_{\tau_{n-1}},\infty))\right)\ \le\ c\ <\ \infty. $$
Observe that $N_{n}(0)$ and $L_{n}$ are independent and
$$ \Prob_{i}(N_{n}(0)\in\cdot)\ =\ \Prob_{i}(N(0)\in\cdot) $$
for all $n\ge 1$. Moreover,
$$ N_{n}(-S_{\tau_{n}})\ =\ L_{n+1}+N_{n+1}(-S_{\tau_{n+1}}) $$
for all $n\in\N_{0}$. By making use of the inequality \cite[Lemma 5.6]{AlsIksMei:15}
$$ (x+y)^{\alpha}\ \le\ x^{\alpha}+y^{\alpha}+\alpha\,2^{\alpha-1}\big(xy^{\alpha-1}+x^{q}y^{\delta}\big), $$
valid for all $x,y\in\R_{\geqslant}$ and $\alpha=q+\delta\ge 1$, we infer
\begin{align*}
N(0)^{\alpha}\ &=\ (L_{1}+N_{1}(-S_{\tau}))^{\alpha}\\
&\le\ L_{1}^{\alpha}\ +\ N_{1}(-S_{\tau})^{\alpha}\ +\ \alpha\,2^{\alpha-1}\left(L_{1}N_{1}(-S_{\tau})^{\alpha-1}+L_{1}^{q}N_{1}(-S_{\tau})^{\delta}\right).
\end{align*}
Since $\Erw_{i}N(0)<\infty$ and $S_{\tau_{n}}\to\infty$ a.s., we have $N_{n}(-S_{\tau_{n}})\le N(-S_{\tau_{n}})\to 0$ a.s. Therefore, by an iteration of the previous inequality, we find
\begin{align*}
N(0)^{\alpha}\ \le\ \sum_{n\ge 1}L_{n}^{\alpha}\ +\ \alpha\,2^{\alpha-1}\sum_{n\ge 1}\left(L_{n}N_{n}(-S_{\tau_{n}})^{\alpha-1}+L_{n}^{q}N_{n}(-S_{\tau_{n}})^{\delta}\right).
\end{align*}
Using further
\begin{align*}
N_{n}(-S_{\tau_{n}})\ &\le\ \sum_{l\ge n}\sum_{k=1}^{\chi_{l+1}}\left(\1_{\{S_{\tau_{l}+k}-S_{\tau_{l}}\le-S_{\tau_{l}}<0\}}+\1_{\{S_{\tau_{l}+k}-S_{\tau_{l}}\le-S_{\tau_{l}},\,S_{\tau_{l}}\le 0\}}\right)\\
&\le\ N_{n}(0)\ +\ \sum_{n\ge 1}\chi_{n}\1_{\{S_{\tau_{n-1}}\le 0\}},
\end{align*}
Lemma \ref{lem:rho[s]} and the inductive hypothesis \eqref{eq:E_sN(0)^beta finite} yield upon taking means
$$ \Erw_{i}N(0)^{\alpha}\ \lesssim\ c\left(1+\alpha\,2^{\alpha-1}\left[\Erw_{i}N(0)^{\alpha-1}+\Erw_{i}N(0)^{\delta}\right]\right)\ <\ \infty.\eqno\qed $$
\end{proof}

\subsection{Proof of Theorem \ref{thm:connections}}\label{subsec:connections}

The following two lemmata will easily establish the asserted implications of Theorem \ref{thm:connections}.

\begin{Lemma}
Let $\alpha>0$ and $\Erw_{i}\tau(i)^{1+\alpha}<\infty$ for some $i\in\cS$. Then $\Erw_{i}J_{i}(D^{i})^{1+\alpha}<\infty$ implies $\Erw_{i}J_{i}(S_{\tau(i)}^{-})^{1+\alpha}<\infty$ as well as \eqref{eq:MRW Spitzer series} and \eqref{eq:MRW renewal counting}.
\end{Lemma}

\begin{proof}
The first implication is obvious because $S_{\tau}^{-}\le D^{i}$. Regarding the other two, use \eqref{eq:def V_alpha^s} and H\"older's inequality to infer
\begin{align*}
\int_{\R_{>}}J_{i}(x)\ &\V_{i}^{\alpha}(dx)\ \asymp\ \int_{\R_{>}}J_{i}'(x)\,\V_{i}^{\alpha}((x,\infty))\ dx\\
&\le\ \Erw_{i}\left(\tau^{\alpha}\int_{0}^{D^{i}}J_{i}'(x)\ dx\right)\ =\ \Erw_{i}\tau^{\alpha}J_{i}(D^{i})\\
&\le\ \left(\Erw_{i}\tau^{1+\alpha}\right)^{\alpha/(1+\alpha)}\left(\Erw_{i}J_{i}(D^{i})^{1+\alpha}\right)^{1/(1+\alpha)}\ <\ \infty.
\end{align*}
Similarly, one finds
\begin{align*}
\int_{\R_{>}}J_{i}(x)^{\alpha}\ &\V_{i}(dx)\ \asymp\ \int_{\R_{>}}J_{i}'(x)J_{i}(x)^{\alpha-1}\,\V_{i}((x,\infty))\ dx\ =\ \Erw_{i}\tau J_{i}(D^{i})^{\alpha}\\
&\le\ \left(\Erw_{i}\tau^{1+\alpha}\right)^{1/(1+\alpha)}\left(\Erw_{i}J_{i}(D^{i})^{1+\alpha}\right)^{\alpha/(1+\alpha)}\ <\ \infty.\qquad\qed
\end{align*}
\end{proof}

\begin{Lemma}
Let $\alpha>0$, $\Erw_{i}\tau(i)^{1+\alpha}<\infty$ for some $i\in\cS$ and $x\in\R_{\geqslant}$. If $\alpha\ge 1$, then $\Sigma_{\alpha}(i,x)<\infty$ implies $\Erw_{i}N(x)^{\alpha}<\infty$, while the reverse implication holds true if $0<\alpha\le 1$.
\end{Lemma}

\begin{proof}
If $\alpha\ge 1$ and $\Sigma_{\alpha}(i,x)<\infty$, then
\begin{align*}
\Erw_{i}(N(x)\wedge m)^{\alpha}\ &\le\ \Erw_{i}\left(\sum_{n\ge 1}(N(x)\wedge m)^{\alpha-1}\1_{\{S_{n}\le x\}}\right)\\
&\le\ \Erw_{i}\left(\sum_{n\ge 1}(N(x)\wedge m)^{\alpha-1}\1_{\{N(x)\wedge m\ge 2n\}}\right)\ +\ 2^{\alpha-1}\Sigma_{\alpha}(i,x)\\
&\le\ \frac{1}{2}\,\Erw_{i}(N(x)\wedge m)^{\alpha}\ +\ 2^{\alpha-1}\Sigma_{\alpha}(i,x)
\end{align*}
for all $m\ge 1$, thus $\Erw_{i}N(x)^{\alpha}\le 2^{\alpha}\Sigma_{\alpha}(i,x)<\infty$.

\vspace{.1cm}
If $0<\alpha\le 1$ and $\Erw_{i}N(x)^{\alpha}<\infty$, then
\begin{align*}
\Erw_{i}N(x)^{\alpha}\ &=\ \Erw_{i}\left(\sum_{n\ge 1}N(x)^{\alpha-1}\1_{\{S_{n}\le x\}}\right)\ \ge\ \Erw_{i}\left(\sum_{n\ge 1}N(x)^{\alpha-1}\1_{\{N(x)\le n,\,S_{n}\le x\}}\right)\\
&\ge\ \Erw_{i}\left(\sum_{n\ge 1}n^{\alpha-1}\1_{\{N(x)\le n,\,S_{n}\le x\}}\right)\ \ge\ \Sigma_{\alpha}(i,x)\ -\ \Erw_{i}\left(\sum_{n=1}^{N(x)}n^{\alpha-1}\right)\\
&\ge\ \Sigma_{\alpha}(i,x)\ -\ \Erw_{i}N(x)^{\alpha},
\end{align*}
thus $\Sigma_{\alpha}(i,x)\le 2\,\Erw_{i}N(x)^{\alpha}<\infty$.\qed
\end{proof}

\begin{proof}[of Theorem \ref{thm:connections}]
In view of the previous two lemmata, the stated implications between Theorems \ref{thm:Kesten-Maller MRW}, \ref{thm:K/M MRW Spitzer series} and \ref{thm:K/M MRW N(x)} are now immediate. As for the finiteness of $\Erw_{i}\sg(x)^{1+\alpha}$ for all $(i,x)\in\cS\times\R_{\geqslant}$ and any given $\alpha>0$, fix $i\in\cS$ and recall that Theorem \ref{thm:K/M MRW Spitzer series}(b) entails $\sum_{n\ge 1}n^{\alpha-1}\Prob_{i}(S_{\tau_{n}}\le x)<\infty$ (see \eqref{eq:Spitzer series embedded}) and that Theorem \ref{thm:K/M MRW N(x)}(b) trivially entails $\Erw_{i}(\sum_{n\ge 1}\1_{\{S_{\tau_{n}}\le x\}})^{\alpha}<\infty$ for all $x\in\R_{\geqslant}$. Hence, under any of these conditions, we infer from Theorem \ref{thm:Kesten-Maller}, applied to the ordinary random walk $(S_{\tau_{n}})_{n\ge 0}$, that $\Erw_{i}\tau_{\nu(x)}^{1+\alpha}<\infty$ for all $x\in\R_{\geqslant}$. Since $\sg(x)\le\tau_{\nu(x)}$, this completes the proof.\qed
\end{proof}

\subsection{Proofs of Theorems \ref{thm:Spitzer-Erickson MRW finite S} and \ref{thm:Kesten-Maller MRW finite S}}\label{subsec:finite S}

Throughout this subsection the state space $\cS$ of $(M_{n})_{n\ge 0}$ is assumed to be finite. It is a well-known fact that the return times $\tau(i)$ then have exponential moments, in particular $\Erw_{i}\tau(i)^{1+\alpha}<\infty$ for all $\alpha\in\R_{\geqslant}$. The proofs of the theorems will be furnished by two subsequent lemmata.

\begin{Lemma}\label{lem:asymp finite space}
Given a MRW $(M_{n},S_{n})_{n\ge 0}$ with finite state space $\cS$, the following assertions hold true for all $i\in\cS$:
\begin{description}[(b)]\itemsep2pt
\item[(a)] There exists $x\in\R_{\geqslant}$ such that $\Prob_{i}(S_{\tau(i)}^{\pm}>y)\gtrsim\Prob_{\pi}(X_{1}^{\pm}>y+x)$ as $y\to\infty$.
\item[(b)] $\Erw_{\pi}|X_{1}|<\infty$ if and only if $\Erw_{i}|S_{\tau(i)}|<\infty$.
\item[(c)] $\Erw_{i}(S_{\tau(i)}^{\pm}\wedge y)\asymp \Erw_{\pi}(X_{1}^{\pm}\wedge y)$ as $y\to\infty$.
\end{description}
\end{Lemma}

\begin{proof}
We will prove (a) for $(X_{1}^{-},S_{\tau(i)}^{-})$ and (c) for $(X_{1}^{+},S_{\tau(i)}^{+})$ because this allows us to directly refer to results stated in this paper. The other cases then follow by switching to $(M_{n},-S_{n})_{n\ge 0}$.

\vspace{.2cm}
(a) Fixing an arbitrary $i\in\cS$, we will show that for all $i,j\in\cS$ with $p_{ij}>0$, there exists $x_{ij}\in\R_{\geqslant}$ such that
\begin{equation}\label{eq1:asymp finite space}
\Prob_{i}(S_{\tau(i)}^{-}>y)\ \gtrsim\ \Prob_{i}(X_{1}^{-}>y+x_{ij}|M_{1}=j)
\end{equation}
as $y\to\infty$. Then, by the finiteness of $\cS$, we can choose $x:=\max_{i,j\in\cS}x_{ij}<\infty$ to obtain the desired result
$$ \Prob_{i}(S_{\tau(i)}^->y)\ \gtrsim  \ \sum_{i\in\cS} \pi_{i}\,p_{ij}\,\Prob_{i}(X_{1}^{-}>y+x|M_{1}=j)\ = \ \Prob_{\pi}(X_{1}^{-}>y+x)$$
as $y\to\infty$.

\vspace{.1cm}
For $u\in\cS$, define $\tau^{0}(u):=\inf\{n\ge 0: M_{n}=u\}$. Pick $i,j\in\cS$ with $p_{ij}>0$. There exist $m_{1},m_{2}\in\N_{0}$ and $z\in\R_{\geqslant}$ such that
\begin{align*}
&p_{1}\ :=\ \Prob_{i}(\tau^{0}(i)=m_1<\tau(i),\,|S_{m_{1}}|\le z)\ >\ 0
\shortintertext{and} 
&\hspace{.6cm}p_{2}\ :=\ \Prob_{j}(\tau^{0}(i)=m_{2},\, |S_{m_{2}}|\le z)\ >\ 0.
\end{align*}
With $m:=m_{1}+m_{2}+1$, it follows that
\begin{align*}
\Prob_{i}&(S_{\tau(i)}^{-}>y)\\
&\ge\ \Prob_{i}(\tau^{0}(j)=m_{1},\, |S_{m_{1}}|\le z,\, M_{m_{1}+1}=j,\,\tau(j)=m,\,|S_{m}-S_{m_{1}+1}|\le z,\, S_{\tau(i)}^{-}>y)\\
&\ge\ p_{1}p_{2}\,\Prob_{i}(X_{1}^{-}>y+2z|M_{1}=j)\ \asymp\ \Prob_{i}(X_{1}^{-}>y+2z|M_{1}=j).
\end{align*}
and this proves \eqref{eq1:asymp finite space}.

\vspace{.1cm}
(b) By using (a) for the positive and the negative part, we infer that $\Erw_{i}|S_{\tau(i)}|<\infty$ implies $\Erw_{\pi}|X_{1}|<\infty$. The reverse implication follows from \eqref{eq:stat mean drift}.	

\vspace{.1cm}
(c) In view of (b), we must only consider the case when $\Erw_{\pi}X_{1}^{+}=\infty$. Let $x\in\R_{\geqslant}$ be the constant provided by part (a) for $(S_{\tau(i)}^{+},X_{1}^{+})$. Then
$$ \Erw_{i}(S_{\tau(i)}^{+}\wedge y) \ =\ \int_{0}^{y} \Prob_{i}(S_{\tau(i)}^{+}>z)\ dz\ 	\gtrsim\ \int_{0}^{y} \Prob_{\pi}(X_{1}^{+} >z+x)\ dz $$
as $y\to\infty$, and the last integral is positive because $\Erw_{\pi} X_{1}^{+}=\infty$. Therefore,
$$ \int_{0}^{y}\Prob_{\pi}(X_{1}^{+} >z+x)\ dz\ \asymp\ \int_{0}^{y+x}\Prob_{\pi}(X_{1}^{+}>z)\ dz\ =\ \Erw_{\pi}[X_{1}^{+}\wedge (y+x)] $$
as $y\to\infty$. Using $\Erw_{\pi}[X_{1}^{+}\wedge (y+x)]\asymp \Erw_{\pi}(X_{1}^{+}\wedge y)$ (cf. Lemma \ref{lem:trunc mean}(b)), we arrive at the conclusion $\Erw_{\pi}(X_{1}\wedge y)\lesssim\Erw_{i}(S_{\tau(i)}^{+}\wedge y)$. For a reverse estimate, use the occupation measure formula \eqref{eq:occ measure formula}
to infer
$$ \Erw_{\pi}(X_{1}^{+}\wedge y)\ =\ \pi_{i}\, \Erw_{i}\left[\sum_{k=1}^{\tau(i)}(X_{k}^{+}\wedge y)\right]\ \geq\ \pi_{i}\,\Erw_{i}\left[\left(\sum_{k=1}^{\tau(i)}X_{k}^{+}\right)\wedge y\right]\ \ge\ \pi_{i}\,\Erw_{i}(S_{\tau(i)}^{+}\wedge y) $$
for all $y\in\R_{\geqslant}$.\qed
\end{proof}

For $\gamma\in [0,1]$, let
\begin{align*}
J_{\pi,\gamma}(x)\ &:=\ 
\begin{cases}
\hfill\displaystyle{\frac{x}{[\Erw_{\pi}(X_{1}^{+}\wedge x)]^{\gamma}}},&\text{if }\Prob_{\pi}(X_{1}>0)>0\\
\hfill x, &\text{otherwise}
\end{cases},
\end{align*}
where $0/[\Erw_{\pi}(X_{1}^{+}\wedge 0)]^{\gamma}:=1$ if $\gamma>0$. Note that $J_{\pi}=J_{\pi,1}$. The second lemma relates the moments of $J_{i,\gamma}(S_{\tau(i)}^{-})$ and $J_{i,\gamma}(D^{i})$ to the respective moments of $J_{\pi,\gamma}(X_{1}^{-})$.

\begin{Lemma}\label{lem:pi int}
Given a nontrivial MRW $(M_{n},S_{n})_{n\ge 0}$ with finite state space $\cS$, the following assertions
are equivalent for any $\alpha\in\R_{\geqslant}$ and $\gamma\in [0,1]$.
\begin{description}[(b)]\itemsep2pt
\item[(a)] $\Erw_{i} J_{i,\gamma}(S_{\tau(i)}^{-})^{1+\alpha} < \infty$ for some/all $i\in\cS$.
\item[(b)] $\Erw_{i}J_{i,\gamma}(D^{i})^{1+\alpha} < \infty$ for some/all $i\in\cS$.
\item[(c)] $\Erw_{\pi}J_{\pi,\gamma}(X_{1}^{-})^{1+\alpha}<\infty$.
\end{description}
\end{Lemma}
	
\begin{proof}
If $\Erw_{\pi}X_{1}^{+}<\infty$, then, by another application of the occupation measure formula \eqref{eq:occ measure formula},
$$ \Erw_{\pi}X_{1}^{+}\ =\ \pi_{i}\,\Erw_{i}\left(\sum_{k=1}^{\tau(i)}X_{k}^{+}\right)\ \ge\ \pi_{i}\,\Erw_{i}S_{\tau(i)}^{+} $$ 
and therefore $J_{\pi,\gamma}(y)\asymp y\asymp J_{i,\gamma}(y)$ as $y\to\infty$. 
If $\Erw_{\pi} X_{1}^{+}=\infty$, then Lemma \ref{lem:asymp finite space}(c) 
entails $J_{\pi,\gamma}(y) \asymp J_{i,\gamma}(y)$ as $y\to\infty$. In any case, it therefore suffices prove the lemma when replacing $J_{i,\gamma}$ with $J_{\pi,\gamma}$ in (a) and (b). Let $i\in\cS$ be an arbitrarily chosen state hereafter.

\vspace{.2cm}	
``(b)$\Rightarrow$(a)'' follows directly from $D^{i}\ge S_{\tau(i)}^{-}$.

\vspace{.2cm}	
``(a)$\Rightarrow$(c)'' With $x\in\R_{\geqslant}$ as provided by Lemma \ref{lem:asymp finite space}(a), we have $J_{\pi,\gamma}(y-x)\asymp J_{\pi,\gamma}(y)$ and thus
\begin{align*}
\infty\ &>\ \Erw_{i} J_{\pi,\gamma}(S_{\tau(i)}^- )^{1+\alpha}\ \gtrsim \ \Erw_{\pi}J_{\pi,\gamma}((X_{1}^{-} -x )^+)^{1+\alpha} \ 
\asymp\ \Erw_{\pi} J_{\pi,\gamma}(X_{1}^{-})^{1+\alpha}.
\end{align*}

\vspace{.2cm}	
``(c)$\Rightarrow$(b)'' Let $F_{jk}$ be the distribution function of $X_{1}^{-}$ given $M_{0}=j$ and $M_{1}=k$ and $F_{jk}^{-1}$ its pseudo-inverse. Given a sequence $(U_{n})_{n\ge 1}$ of iid uniformly distributed random variables on $(0,1)$ which are independent of all other occurring random variables, the sequence $(M_{n},\wh{X}_{n})_{n\ge 1}$ with $\wh{X}_{n}:=F_{M_{n-1},M_{n}}^{-1}(U_{n})$ forms a distributional copy of $(M_{n},X_{n}^{-})_{n\ge 1}$. Put $\wh{S}_n=\sum_{k=1}^{n}\wh{X}_{k}$ for $n\ge 1$ and
$$ G(y)\ :=\ \max_{j,k\in\cS}\Prob_{j}(X_{1}^{-}\le y|M_{1}=k), $$
which is a proper distribution function as $\cS$ is finite. Now $(W_{n})_{n\ge 1}:=(G^{-1}(U_{n}))_{n\ge 1}$ forms an iid sequence independent of $(M_{n})_{n\ge 0}$ and with $\wh{X}_{n}\le W_{n}$ for all $n\ge 1$. Use that $J_{\pi,\gamma}$ is a subadditive and nondecreasing to infer
\begin{align*}
\Erw_{i}J_{\pi,\gamma}(D^{i})^{{1+\alpha}}\ &\le\ \Erw_{i}J_{\pi,\gamma}\left(\sum_{k=1}^{\tau(i)}X_{k}^{-}\right)^{{1+\alpha}}\ =\ \Erw_{i}J_{\pi,\gamma}(\wh{S}_{\tau(i)})^{{1+\alpha}}\\
&\le\ \Erw_{i}\left(\sum_{k=1}^{\tau(i)}J_{\pi,\gamma}(\wh{X}_{k})\right)^{{1+\alpha}}\ \le\ \Erw \left(\sum_{k=1}^{\tau(i)}J_{\pi,\gamma}(W_{k})\right)^{{1+\alpha}}.	
\end{align*}
By \cite[Theorem 1.5.4]{Gut:09}, the upper bound is finite iff $\Erw_{i}\tau(i)^{{1+\alpha}}$ and $\Erw J_{\pi,\gamma}(W_{1})^{1+\alpha}$ are both finite. We must only verify the finiteness of the second expectation which follows from
\begin{align*}	
\infty\ >\ \Erw_{\pi} J_{\pi,\gamma}(X_{1}^{-})^{{1+\alpha}}\ &=\ \sum_{j,k\in\cS} \pi_{j}\,p_{jk}\, \Erw_{j}(J_{\pi,\gamma}(X_{1}^{-})^{1+\alpha}|M_{1}=k)\\
&\ge\ c \sum_{j,k\in\cS:\, p_{jk}>0 }\,\int J_{\pi,\gamma}(y)^{1+\alpha}\ \Prob_{j}(X_{1}^{-}\in dy|M_{1}=k)\\
&\ge\ c\,\Erw J_{\pi,\gamma}(W)^{1+\alpha},
\end{align*}
where $c:=\min\{\pi_{j}\,p_{jk}: j,k\in\cS\text{ and }p_{jk}>0\}$.\qed
\end{proof}

\begin{proof}[of Theorem \ref{thm:Spitzer-Erickson MRW finite S}]
By Lemma \ref{lem:pi int}, we may replace $D^{i}$ with $S_{\tau(i)}^{-}$ in (b) which is assumed hereafter. Let us establish the equivalence of this modified condition and \eqref{eq:J_pi condition S finite}. If $\Erw_{\pi}|X_{1}|<\infty$, then $\Erw_{i}|S_{\tau(i)}|<\infty$ by Lemma \ref{lem:asymp finite space}(b), and since, by another use of \eqref{eq:occ measure formula}, we then also have
$$ \pi_{i}\,\lim_{y\to\infty}A_{i}(y)\ =\ \pi_{i}\,\Erw_{i}S_{\tau(i)}\ =\ \Erw_{\pi}X_{1}\ =\ \lim_{y\to\infty}A_{\pi}(y), $$
the asserted equivalence follows. If $\Erw_{\pi}|X_{1}|=\infty$ and thus $\Erw_{\pi}|S_{\tau(i)}|=\infty$, then only equivalence of $\Erw_{i}J_{i}(S_{\tau(i)}^{-})<\infty$ and $\Erw_{\pi}J_{\pi}(X_{1}^{-})<\infty$ must be verified, as pointed out in Remark \ref{rem:Erickson's extension}. But this is ensured by the previous lemma for $\gamma=1$.

\vspace{.1cm}
It remains to verify that \ref{thm:Spitzer-Erickson MRW}(d), i.e. $\Erw_{i}\sg(x)<\infty$ for all $(i,x)\in\cS\times\R_{\geqslant}$, implies the modified condition \ref{thm:Spitzer-Erickson MRW}(b). By Prop. \ref{prop:ladder chain} (note that the dual MRW is trivially also positive divergent if $|\cS|<\infty$), the ladder chain $(M_{n}^{>})_{n\ge 0}$ is positive recurrent on some $\cS^{>}\subset\cS$ with unique stationary law $\pi^{>}$ vanishing outside $\cS^{>}$. In particular, $\kappa(i):=\inf\{n:\Mgn=s\}$ has finite mean under $\Prob_{i}$ for each $i\in\cS^{>}$. Moreover, the sequence $(M_{n}^{>},\sgn)_{n\ge 0}$ forms a MRW with $\Erw_{\pi^>}\sg<\infty$, for $\cS$ is finite. Now fix any $i\in\cS^{>}$ and recall that $\tg(i)$ denotes the first ascending ladder epoch of $(S_{\tau_{n}(i)})_{n\ge 0}$. Then we obviously have $\tg(i)\le\sigma_{\kappa(i)}=\inf\{\sgn:M_{\sgn}=s\}$ and thus, by making use of the occupation measure formula \eqref{eq:occ measure formula}
$$ \Erw_{i}\tg(i)\ \le\ \Erw_{i}\sigma_{\kappa(i)}\ =\ \Erw_{\pi^>}\sg\,\Erw_{i}\kappa(i)\ <\ \infty $$
which in turn implies the modified \ref{thm:Spitzer-Erickson MRW}(b) by invoking Theorem \ref{thm:Spitzer-Erickson} for the ordinary random walk $(S_{\tau_{n}(i)})_{n\ge 0}$.\qed
\end{proof}

\begin{proof}[of Theorem \ref{thm:Kesten-Maller MRW finite S}]
By Lemmata \ref{lem:pi int} and \ref{lem:solidarity LI}, only ``\eqref{eq: sigma alpha cond}$\,\RA\Erw_{i}J_{i}(S_{\tau(i)}^-)^{1+\alpha}<\infty$'' remains to be proved. On the other hand, this requires further results on the moments of $\sg(x)$, notably Proposition \ref{prop: MRW sigma Prop}, and will therefore be postponed to the end of Section \ref{sec:moment sg(x)}.
\end{proof}

\section{A closer look at the moments of $\sg(x)$}\label{sec:moment sg(x)}

This section is devoted to a more detailed discussion of some aspects regarding the moments of $\sg(x)$ for which none of the conditions provided by Theorem \ref{thm:Kesten-Maller MRW}(a), \ref{thm:K/M MRW Spitzer series}(a) and \ref{thm:K/M MRW N(x)}(a) appears to be necessary. 

\subsection{A counterexample}\label{subsec:sg(x) moments counterexample}

We begin with an example that will actually show that, for any given $\alpha\ge 0$, we can define a nontrivial \emph{oscillating} MRW $(M_{n},S_{n})_{n\ge 0}$ having \emph{negative divergent} embedded RW $(S_{\tau_{n}(i)})_{n\ge 0}$ and yet $\Erw_{i}\sg(x)^{1+\alpha}<\infty$ for all $(i,x)\in\cS\times\R_{\geqslant}$. In other words,  the last property does not entail the positive divergence of $(S_{n})_{n\ge 0}$, nor any of the other assertions stated in the theorems of Section \ref{sec:main results}, which is in sharp contrast to the case of ordinary RW.

\begin{Exa}\label{exa:sg(x) all moments}\rm
Let $(W_{n})_{n\ge 0}$ be an ordinary zero-delayed integer-valued random walk with generic increment $Y$ satisfying $\Prob(Y=-n)>0$ for all $n\in\N_{0}$,
\begin{equation*}
\Prob(Y^{-}>n)\,=\,\frac{1}{n^{1+\alpha}}\quad\text{for all sufficiently large }n\in\N,
\end{equation*}
and $\Prob(Y^{+}\in\cdot)$ such that $\Erw\big(\inf\{n:W_{n}>0\})^{1+\alpha}=\infty$ for any fixed $\alpha\ge 0$, a particular choice being $Y^{+}\equiv 0$. Further defining $f:\R_{>}\to\R_{>}$ by $f(x):=2^{\theta x^{1+\alpha}}$ for some $\theta>1+\alpha$, we have $f(x)\ge x$ for all $x\ge 1$ and also find that
\begin{align}
\begin{split}\label{eq:tail condition on Y^{-}}
&\lim_{n\to\infty}n\,\Prob(f(Y^{-})>c2^{n})\ =\ \lim_{n\to\infty}n\,\Prob\left(Y^{-}>f^{-1}(c2^{n})\right)\\
&\hspace{1.5cm}=\ \lim_{n\to\infty}n\,\Prob\left(Y^{-}>\left(\frac{n+\log_{2}c}{\theta}\right)^{1/(1+\alpha)}\right)\ =\ \theta
\end{split}
\end{align}
for all $c\in\R_{>}$.

\vspace{.1cm}
Now consider a MRW $(M_{n},S_{n})_{n\ge 0}$ such that $(M_{n})_{n\ge 0}$ has state space $\N_{0}$ and transition probabilities $p_{00}=\Prob(Y\ge 0)$, $p_{0i}=\Prob(Y=-i)$ and $p_{i0}=1$ for $i\in\N$. Furthermore,
$$ K_{00}\ =\ \Prob(Y\in\cdot\,|Y\ge 0),\quad K_{0i}\ =\ \delta_{f(i)}\quad\text{and}\quad K_{i0}\ =\ \delta_{-f(i)-i} $$
for all $i\in\N$. Notice that $\tau(0)\le 2$ a.s. and that the law of $(S_{\tau_{n}(0)})_{n\ge 0}$ under $\Prob_{0}$ equals the law of $(W_{n})_{n\ge 0}$.

\vspace{.1cm}
Fixing any $x>0$, the following property of the MRW under $\Prob_{0}$ is essential for our considerations, namely
$$ X_{\tau_{n}(0)+1}\ \le\  x\quad\Longrightarrow\quad S_{\tau_{n+1}(0)}-S_{\tau_{n}(0)}\ \ge\ -x. $$
As a consequence of this property, we infer that
\begin{align*}
\{\sg(x)>\tau(0),M_{0}=0\}\ &\subset\ \{S_{\tau(0)}\ge -x,M_{0}=0\},\\
\{\sg(x)>\tau_{2}(0),M_{0}=0\}\ &\subset\ \{S_{\tau(0)}\ge -x,X_{\tau(0)+1}\le 2x,M_{0}=0\}\\
&\subset\ \{S_{\tau_{2}(0)}\ge -3x,M_{0}=0\}
\shortintertext{and then inductively}
\{\sg(x)>\tau_{n}(0),M_{0}=0\}\ &\subset\ \left\{S_{\tau_{n}(0)}\ge -(2^{n}-1)x,M_{0}=0\right\}
\end{align*}
for all $n\in\N$. Defining $\wh{\sigma}(x):=\inf\{n\ge 1:X_{\tau_{n}(0)+1}>x2^{n}\}$, this implies $\sg(x)\le\wh{\sigma}(x)$ and the subsequent argument will show that $\Erw_{0}\wh{\sigma}(x)^{1+\alpha}<\infty$ for all $x\in\R_{\geqslant}$, giving
$$ \Erw_{i}\sg(x)^{1+\alpha}\ \le\ \Erw_{0}[1+\sg(x+i+f(i))]^{1+\alpha}\ <\ \infty $$
for all $i\in\N$ and $x\in\R_{\geqslant}$ as desired.

\vspace{.1cm}
Put $F(x):=\Prob_{0}(X_{1}\le x)$ and note that
$$ 1-F(x)\ =\ \Prob(f(Y^{-})>x) $$
for $x\in\R_{>}$. We start by pointing out that
\begin{align*}
\Erw_{0}\wh{\sigma}(x)^{1+\alpha}\ \asymp\ \sum_{n\ge 1}n^{\alpha}\,\Prob_{0}(\wh{\sigma}(x)>n)\ =\ \sum_{n\ge 1}n^{\alpha}\prod_{k=1}^{n}F(x2^{k})
\end{align*}
because the $X_{\tau_{n}(0)+1}$ are iid under $\Prob_{0}$ with distribution function $F$.
Put $b_{n}:=n^{\alpha}\prod_{k=1}^{n}F(x2^{k})$ for $n\in\N$. Now $\Erw_{0}\wh{\sigma}^{1+\alpha}<\infty$ follows by Raabe's test (see e.g. Stromberg \cite[(7.16)]{Stromberg:81}) if 
$$ \liminf_{n\to\infty}n\left(\frac{b_{n}}{b_{n+1}}-1\right)\ >\ 1. $$
To this end, use \eqref{eq:tail condition on Y^{-}} to obtain
\begin{align*}
n(1-F(x2^{n+1}))\ =\ n\,\Prob(f(Y^{-})>x2^{n+1})\ =\ \theta+o(1)
\end{align*}
as $n\to\infty$ and then finally conclude
\begin{align*}
\liminf_{n\to\infty}n\left(\frac{b_{n}}{b_{n+1}}-1\right)\ &=\ \liminf_{n\to\infty}n\left(\frac{n^{\alpha}-(n+1)^{\alpha}F(x2^{n+1})}{(n+1)^{\alpha}F(x2^{n+1})}\right)\\
&=\ \liminf_{n\to\infty}\frac{n}{F(x2^{n+1})}\left(\left(1+\frac{1}{n}\right)^{-\alpha}-F(x2^{n+1})\right)\\
&=\ \liminf_{n\to\infty}n\left(1-\frac{\alpha}{n}+o\left(\frac{1}{n}\right)-F(x2^{n+1})\right)\\
&=\ \theta-\alpha\ >\ 1.
\end{align*}
Finally, the finiteness of $\Erw_{0}\sg(x)^{1+\alpha}$ particularly entails $\limsup_{n\to\infty}S_{n}=\infty$ a.s. and thus confirms that $(S_{n})_{n\ge 0}$ is indeed oscillating.\qed
\end{Exa}

\subsection{Solidarity}\label{subsec:sg(x) moments solidarity}

It is clear that $\Erw_{i}\sg(x)^{1+\alpha}<\infty$ for some $(i,x)\in\cS\times\R_{\geqslant}$ does not necessarily entail
\begin{equation}\label{eq:solidarity moments sg(x)}
\Erw_{i}\sg(x)^{1+\alpha}\,<\,\infty\quad\text{for all }(i,x)\in\cS\times\R_{\geqslant},
\end{equation}
as we may have $\Prob_{i}(\sg(x)=1)=1$ for some $(i,x)$ even when $(M_{n},S_{n})_{n\ge 0}$ is negative divergent and thus $\Prob_{j}(\sg(y)=\infty)>0$ for some other $(j,y)\in\cS\times\R_{\geqslant}$. The subsequent lemma provides sufficient conditions for \eqref{eq:solidarity moments sg(x)}. Consider the condition
\begin{equation}\label{eq:cond sg(x)>tau(i)}
q(i,x)\,:=\,\Prob_{i}(\sg(x)>\tau(i),S_{\tau(i)}<0)\,>\,0
\end{equation}
for some $(i,x)\in\cS\times\R_{\geqslant}$ and note that $q(i,x)$ is nondecreasing in $x$.

\begin{Lemma}\label{lem:sg(x) moments solidarity}
Let $\alpha\ge 0$ and $(M_{n},S_{n})_{n\ge 0}$ be a nontrivial MRW. Then any of the following three assumptions implies \eqref{eq:solidarity moments sg(x)}.
\begin{description}[(b)]\itemsep2pt
\item[(a)] $\Erw_{i}\sg(x)^{1+\alpha}<\infty$ for some $i\in\cS$ and all $x\in\R_{\geqslant}$.
\item[(b)] $\Erw_{i}\sg(x)^{1+\alpha}<\infty$ and $q(i,x)>0$ for some $(i,x)\in\R_{\geqslant}$.
\item[(c)] $\Erw_{i}\sg(0)^{1+\alpha}<\infty$ and $\Erw_{i}\tau(i)^{1+\alpha}<\infty$ for all $i\in\cS$.
\end{description}
\end{Lemma}

\begin{proof}
(a) Pick any $j\ne i$. Then there exists $x_{0}\in\R_{\geqslant}$ such that 
$$ \Prob_{i}(\sg(2x)>\tau(j),S_{\tau(j)}\le x)\ \ge\ \Prob_{i}(\sg(2x_{0})>\tau(j),S_{\tau(j)}\le x_{0})\ =:\ p\ >\ 0 $$
for all $x\ge x_{0}$. For any such $x$, we now infer
\begin{align*}
\infty\ >\ \Erw_{i}\sg(2x)^{1+\alpha}\ \ge\ \Erw_{i}\sg(2x)^{1+\alpha}\1_{\{\sg(2x)>\tau(j),\,S_{\tau(j)}\le x\}})\ \ge\ p\,\Erw_{j}\sg(x)^{1+\alpha},
\end{align*}
that is $\Erw_{j}\sg(x)^{1+\alpha}<\infty$. Since $\sg(x)$ is nondecreasing in $x$, \eqref{eq:solidarity moments sg(x)} follows.

\vspace{.1cm}
(b) Since $q(i,x)>0$, we can find $h>0$ small enough such that 
$$ p(x)\ :=\ \Prob_{i}(\sg(x)>\tau(i),S_{\tau(i)}\le-h)\ >\ 0. $$
Consequently,
\begin{align*}
\infty\ >\ \Erw_{i}\sg(x)^{1+\alpha}\1_{\{\sg(x)>\tau(i),S_{\tau(i)}\le-h\}}\ \ge\ p(x)\,\Erw_{i}\sg(x+h)^{1+\alpha}.
\end{align*}
Since $p(x+nh)\ge p(x+(n-1)h)$ for all $n\in\N$, an induction over $n$ provides us with $\Erw_{i}\sg(x+nh)^{1+\alpha}<\infty$, and this implies \eqref{eq:solidarity moments sg(x)} by an appeal to (a).

\vspace{.1cm}
(c) If $\Prob_{i}(S_{\tau(i)}\ge 0)=1$ for some $i\in\cS$, thus $\Prob_{i}(S_{\tau(i)}>0)>0$ by nontriviality, then the assertion follows easily from Lemma \ref{lem:solidarity LI}.

\vspace{.1cm}
Left with the case $\Prob_{i}(S_{\tau(i)}<0)>0$ for all $i\in\cS$, fix $i$. If $q(i,0)>0$, then the assertion follows from (b). Assuming $q(i,0)=0$ and thus $\Prob_{i}(\sg(0)<\tau(i),S_{\tau(i)}<0)>0$, we will show below that
$$ q_{n}(j,0)\,:=\,\Prob_{j}(\sg(0)>\tau_{n}(j),S_{\tau_{n}(j)}<0)\,>\,0 $$
for some $j\in\cS$ and $n\in\N$. Then \eqref{eq:solidarity moments sg(x)} can be concluded in a similar manner as in (b).

\vspace{.1cm}
Define
$$ \kappa\ :=\ \inf\left\{0\le n<\tau(i):S_{n}=\max_{0\le k\le\tau(i)}S_{k}\right\} $$
By assumption, there exists $j\in\cS\backslash\{i\}$ such that
$$ \Prob_{i}(M_{\kappa}=i,\kappa=\tau_{m}(j)\ge\sg(0),S_{\tau(i)}<0,\tau_{m+l}(j)<\tau(i)<\tau_{m+l+1}(j))\ =:\ p'\ >\ 0 $$
for some $m\in\N$ and $l\in\N_{0}$. Put
$E:=\{S_{k}\le 0\text{ for }1\le k\le\tau(i),\tau_{l}(j)<\tau(i)<\tau_{l+1}(j)\}$. Then
\begin{align*}
p'\ &=\ \Prob_{i}\big(S_{k}<S_{\tau_{m}(j)}\text{ for }0\le k<\tau_{m}(j),\,S_{\tau(i)}-S_{\tau_{m}(j)}\le-S_{\tau_{m}(j)},\\
&\hspace{0.8cm}S_{\tau_{m}(j)+k}-S_{\tau_{m}(j)}\le 0\text{ for }1\le k\le\tau(i)-\tau_{m}(j),\tau_{m+l}(j)<\tau(i)<\tau_{m+l+1}(j)\big)\\
&=\ \int_{\R_{>}}\Prob_{j}\big(E\cap\{S_{\tau(i)}<-x\}\big)\ \Prob_{i}(S_{\tau_{m}(j)}\in dx,S_{k}<S_{\tau_{m}(j)}\text{ for }0\le k<\tau_{m}(j)<\tau(i))\\
&=\ \int_{\R_{>}}\Prob_{j}\big(E\cap\{S_{\tau(i)}<-x\}\big)\ \\
&\hspace{1.5cm}\times\Prob_{j}(S_{\tau_{m+l}(j)}-S_{\tau(i)}\in dx,S_{k}<S_{\tau_{m+l}(j)}\text{ for }\tau(i)\le k<\tau_{m}(j)<\tau_{2}(i))\\
&=\ \Prob_{j}\big(E\cap\{S_{\tau_{m+l}(j)}<0,S_{k}<S_{\tau_{m+l}(j)}\text{ for }\tau(i)\le k<\tau_{m+l}(j)<\tau_{2}(i)\}\big)\\
&\le\ \Prob_{j}\big(S_{\tau_{m+l}(j)}<0,\,S_{k}\le 0\text{ for }1\le k<\tau_{m+l}(j)\big)\\
&=\ q_{m+l}(j,0)
\end{align*}
hence $q_{m+l}(j,0)>0$.\qed
\end{proof}

\begin{Rem}\label{rem:Sisyphus chain}\rm
Returning to part (c) of the previous lemma, its proof holds the surprise that $\Erw_{i}\tau(i)^{1+\alpha}<\infty$ is needed if $\Prob_{i}(S_{\tau(i)}\ge 0)=1$ for all $i\in\cS$, but not otherwise. The following simple example illustrates that one cannot dispense with this assumption there. Given any $\alpha\ge 0$, consider a \emph{Sisyphus chain} on $\N_{0}$ with transition probabilities
$$ p_{01}\ =\ 1\quad\text{and}\quad p_{n,n+1}\ =\ 1-p_{n0}\ =\ \left(\frac{n}{n+1}\right)^{1+\alpha}\text{ for }n\ge 1, $$
so that $\Prob_{0}(\tau(0)>n)=p_{01}\cdot...\cdot p_{n-1,n}=n^{-1-\alpha}$ for $n\ge 1$ and thus $\Erw_{0}\tau(0)^{1+\alpha}=\infty$. Further defining $X_{n}=(M_{n}+1)^{-2}>0$, we obviously have $\sg(0)=1$ a.s. and
$$ S_{\tau(0)}\ =\ 1+\sum_{k=0}^{\tau(0)-1}\frac{1}{(M_{k}+1)^{2}}\ <\ 1+\frac{\pi^{2}}{6}\ =:\ x\quad\Prob_{0}\text{-a.s.} $$
Consequently, $\Prob_{0}(\sg(x)>\tau(0))=1$ and therefore $\Erw_{0}\sg(x)^{1+\alpha}=\infty$.
\end{Rem}

\subsection{Moment result for a modification of $\sg(x)$}

In view of Example \ref{exa:sg(x) all moments} and Lemma \ref{lem:sg(x) moments solidarity}, a look at the stopping time
$$ \osg(x)\ :=\ \inf\{n>\tau(M_{0}):S_{n}>x\} $$
appears to be natural. When the driving chain has initial state $i\in\cS$, it is the \emph{post-$\tau(i)$-passage time} of $(S_{n})_{n\ge 0}$ beyond $x$. Evidently, $\sg(x)\le\osg(x)$ a.s. for all $(i,x)\in\cS\times\R_{\geqslant}$. The next result provides an equivalent condition for the finiteness of $\Erw_{i}\osg(x)^{1+\alpha}$.

\begin{Prop}\label{prop: MRW sigma Prop}
Let $\alpha\ge 0$ and $(M_{n},S_{n})_{n\ge 0}$ be a nontrivial MRW with positive divergent embedded RW $(S_{\tau_{n}(i)})_{n\ge 0}$ and $\Erw_{i}\tau(i)^{1+\alpha}<\infty$ for some/all $i\in\cS$. Then $A_{i}(x)$ is ultimately positive for all $i\in\cS$ and the following assertions are equivalent:
\begin{description}[(b)]\itemsep2pt
\item[(a)] $\Erw_{i}J_{i}(S_{\tau(i)}^{-})^{1+\alpha}<\infty$ for some/all $i\in\cS$.
\item[(b)] $\Erw_{i}\osg(x)^{1+\alpha}<\infty$ for some/all $(i,x)\in\cS\times\R_{\geqslant}$.
\item[(c)] $\Erw_{i}\tau_{\nu(i,x)}(i)^{1+\alpha}<\infty$ for some/all $(i,x)\in\cS\times\R_{\geqslant}$.
\end{description}
Moreover, these conditions imply $\Erw_{i}\sg(x)^{1+\alpha}<\infty$ for all $(i,x)\in\cS\times\R_{\geqslant}$.
\end{Prop}

\begin{proof}
Fix any $i\in\cS$ and write again $\tau,\tau_{n},$ etc. as shorthand for $\tau(i),\tau_{n}(i),$ etc. Recalling $\nu(x)=\inf\{n\ge 1:S_{\tau_{n}}>x\}$, observe also that $\osg(x)\le\tau_{\nu(x)}$ $\Prob_{i}$-a.s.

\vspace{.1cm}
``(a)$\RA$(b)'' By Theorem \ref{thm:Spitzer-Erickson} or \ref{thm:Kesten-Maller}, (a) ensures $\Erw_{i}\nu(x)^{1+\alpha}<\infty$ for all $x\in\R_{\geqslant}$ which in combination with $\Erw_{i}\tau^{1+\alpha}<\infty$ provides us with $\Erw_{i}\tau_{\nu(x)}^{1+\alpha}<\infty$ for all $x\in\R_{\geqslant}$ as pointed out earlier (use Wald's equation for $\alpha=0$ and see the proof of Lemma \ref{lem:solidarity N(x)} for $\alpha>0$). Hence (b) follows from $\osg(x)\le\tau_{\nu(x)}$.

\vspace{.1cm}
``(b)$\RA$(a)'' Assuming $\Erw_{i}\osg(0)^{1+\alpha}<\infty$ and putting $p:=\Prob_{i}(\osg(0)=\tau)$, we distinguish the two cases
$p=1$ and $p<1$.

\vspace{.1cm}
If $p=1$ and thus $\osg(0)=\tau_{\nu(0)}$, then (a) follows by an appeal to Theorem \ref{thm:Spitzer-Erickson} or \ref{thm:Kesten-Maller}. If $p<1$, then we infer with the help of Lemma \ref{lem:K/M Lemma 3.5} below
\begin{align*}
\infty\ &>\ \Erw_{i}\osg(0)^{1+\alpha}\ \ge\ \Erw_{i}\osg(0)^{1+\alpha}\1_{\{\osg(0)>\tau\}}\\
&\ge\ \Erw_{i}(\osg(0)-\tau)^{1+\alpha}\1_{\{\osg(0)>\tau\}}\\
&=\ \int\Erw_{i}\sg(y)^{1+\alpha}\ \Prob_{i}(S_{\tau}^{-}\in dy,\,\osg(0)>\tau)\\
&\ge\ \int\Erw_{i}\sg(y)^{1+\alpha}\ \Prob_{i}(S_{\tau}^{-}\in dy,\,S_{\tau}<0)\\
&\gtrsim\ \int J_{i}(y)^{1+\alpha}\ \Prob_{i}(S_{\tau}^{-}\in dy)
\end{align*}
and thus the assertion.

\vspace{.1cm}
``(a)$\LRA$(c)'' This follows directly from Lemma \ref{lem:solidarity LI} when noting that its proviso in (b), namely $A_{i}(y)>0$ for all sufficiently large $y$, is guaranteed here by the positive divergence of $(S_{\tau_{n}})_{n\ge 0}$.\qed
\end{proof}

The following auxiliary lemma forms an extension of Lemma 3.5 in \cite{KesMal:96} to MRW $(M_{n},S_{n})_{n\ge 0}$, but we give the rather technical proof of the result only under the stronger assumption of positive divergence of the $(S_{\tau_n}(i))_{n\ge 0}$ because this is enough for our purposes.

\begin{Lemma}\label{lem:K/M Lemma 3.5}
Let $\alpha\ge 0$ and $(M_{n},S_{n})_{n\ge 0}$ be a nontrivial MRW such that $S_{\tau_{n}(i)}\to\infty$ in probability for some $i\in\cS$. Then, as $x\to\infty$,
\begin{align*}
&\hspace{1cm}\sum_{n\ge 1}\frac{1}{n}\,\Prob_{j}(S_{n}^{*}\le x)\ \gtrsim\ \log J_{j}(x)
\shortintertext{and}
&\Erw_{j}\sg(x)^{\alpha}\ \asymp\ \sum_{n\ge 1}n^{\alpha-1}\,\Prob_{j}(S_{n}^{*}\le x)\ \gtrsim\ J_{j}(x)^{\alpha}
\end{align*}
for all $j\in\cS$ and $\alpha>0$.
\end{Lemma}

\begin{proof}
By Theorem \ref{thm:Spitzer-Erickson}, positive divergence of the $(S_{\tau_{n}(i)})_{n\ge 0}$ ensures $A_{i}(x)>0$ for sufficiently large $x$ and $\lim_{n\to\infty}\Prob_{i}(S_{n}^{*}\le x)=0$ for all $(i,x)\in\cS\times\R_{\geqslant}$, the latter because $\sg(x)\le\tau_{\nu(i,x)}(i)<\infty$ a.s. Fix $i\in\cS$ and define
$$ m_{\delta}(x)\ :=\ \inf\{n\ge 1:\Prob_{i}(S_{n}^{*}\le x)<1-\delta\} $$
for $x\in\R_{>}$ and $0<\delta<1$ which are all finite with $m_{\delta}(x)\uparrow\infty$ as $x\uparrow\infty$. Then
\begin{align*}
\sum_{n\ge 1}n^{\alpha-1}\,\Prob_{i}(S_{n}^{*}\le x)\ &\ge\ \sum_{n=1}^{m_{\delta}(x)}n^{\alpha-1}\,\Prob_{i}(S_{n}^{*}\le x)\ \ge\ 
\begin{cases}
(1-\delta)\log m_{\delta}(x),&\text{if }\alpha=0,\\
\hfill c(1-\delta)m_{\delta}(x)^{\alpha},&\text{if }\alpha>0.
\end{cases}
\end{align*}
for some $c\in\R_{>}$. Using $J_{i}(x)\le x/A_{i}(x)$, it therefore remains to show that
\begin{equation}\label{eq:lower bound m_delta(x)}
\frac{x}{A_{i}(x)}\ \lesssim\ m_{\delta}(x)\quad\text{as }x\to\infty.
\end{equation}

We will now assume that \eqref{eq:lower bound m_delta(x)} fails and produce a contradiction. First note that under this assumption, we find for all $\eps\in (0,1)$ an increasing nonnegative and unbounded sequence $(x_{l})_{l\ge 1}$ (depending on $\eps$) such that
\begin{align}
&\hspace{1.2cm}\sup_{l\ge 1}\,2m_{\delta}(x_{l})\frac{A_{i}(x_{l})}{x_{l}} \le\ \eps\label{eq1:m_delta}
\shortintertext{which may be restated as}
&(1-\eps)x_{l}+2m_{\delta}(x_{l})\ \le\ x_{l}\quad\text{for all }l\ge 1.\label{eq2:m_delta}
\end{align}
Putting
\begin{equation}\label{eq:def H^i}
H_{n}^{i}\ :=\ \max_{\tau_{n-1}<k\le\tau_{n}}(S_{k}-S_{\tau_{n}})^{+}
\end{equation}
for $n\in\N$ with generic copy $H^{i}$ under $\Prob_{i}$ and writing $m_{l}$ as shorthand for $m_{\delta}(x_{l})$, we infer
\begin{align*}
A\,:&=\,\Prob_{i}(S_{\tau_{m_{l}}}^{*}>x_{l})\ =\ \Prob_{i}\left(\max_{1\le k\le m_{l}}(S_{\tau_{k-1}}+H_{k}^{i})>x_{l}\right)\\
&=\ \sum_{n=1}^{m_{l}}\Prob_{i}\left(\max_{1\le k\le n-1}(S_{\tau_{k-1}}+H_{k}^{i})\le x_{l}<S_{\tau_{n-1}}+H_{n}^{i}\right)\ =\ \sum_{n=1}^{m_{l}}\Prob_{i}(E_{n}),
\end{align*}
where $W_{j,k}:=S_{\tau_{j+k-1}}-S_{\tau_{j}}$ and
$$ E_{n}\ :=\ \left\{\max_{1\le k\le n-1}(W_{2m_{l}-n+1,k}+H_{2m_{l}-n+k+1}^{i})\le x_{l}<W_{2m_{l}-n+1,n}+H_{2m_{l}+1}^{i}\right\}. $$
Moreover, we have used that $(S_{\tau_{n}}-S_{\tau_{n-1}},H_{n}^{i})_{n\ge 1}$ forms a sequence of iid random vectors under $\Prob_{i}$. Since $S_{\tau_{n}}\to\infty$ a.s., we can pick $h>0$ and $n_{0}\in\N$ such that
\begin{align*}
\Prob_{i}(H^{i}>h)\ <\ \frac{\delta}{4}\quad\text{and}\quad\theta\ :=\ \inf_{n\ge n_{0}}\Prob_{i}(S_{\tau_{n}}>h)\ >\ \frac{1}{2}.
\end{align*} 
Choosing $l$ so large that $m_{l}>n_{0}$, we now further estimate
\begin{align*}
A\ &\le\ \sum_{n=1}^{m_{l}}\Prob_{i}(E_{n})\,\frac{1}{\theta}\,\Prob_{i}(S_{\tau_{2m_{l}-n+1}}>h)\ \le\ \frac{1}{\theta}\sum_{n=1}^{m_{l}}\Prob_{i}\big(E_{n}\cap\{S_{\tau_{2m_{l}}}+H_{2m_{l}+1}^{i}>x_{l}+h\}\big)\\
&\le\ \frac{1}{\theta}\,\Prob_{i}(S_{\tau_{2m_{l}}}+H_{2m_{l}+1}^{i}>x_{l}+h)\ \le\ \frac{1}{\theta}\left(\Prob_{i}(S_{\tau_{2m_{l}}}>x_{l})\ +\ \Prob_{i}(H^{i}>h)\right)\\
&\le\ \frac{1}{\theta}\,\Prob_{i}(S_{\tau_{2m_{l}}}>x_{l})\ +\ \frac{\delta}{2}.
\end{align*}
Put $\zeta_{k,l}:=[(S_{\tau_{k}}-S_{\tau_{k-1}})\vee(-x_{l})]\wedge x_{l}$ and observe that $A_{i}(x_{l})=\Erw_{i}\zeta_{k,l}$. Use (3.71) in \cite{KesMal:96} to obtain
$$ \Prob_{i}(S_{\tau_{2m_{l}}}> x_{l})\ \le\ \Prob_{i}\left(\sum_{k=1}^{2m_{l}}\zeta_{k,l}>x_{l}\right)\ +\ 2m_{l}\,\Prob_{i}(S_{\tau}>x_{l}). $$
Now use \eqref{eq2:m_delta}, Chebychev's inequality and $x_{l}^{2}\,\Prob_{i}(S_{\tau}>x_{l})\le\Erw_{i}\zeta_{1,l}^{2}$ to obtain
\begin{align*}
\Prob_{i}(S_{\tau_{2m_{l}}}> x_{l})\ &\le\ \Prob_{i}\left(\sum_{k=1}^{2m_{l}}\big(\zeta_{k,l}-A_{i}(x_{l})\big)>(1-\eps)x_{l}\right)\,+\,2m_{l}\,\Prob_{i}(S_{\tau}>x_{l})\\
&\le\ \frac{2m_{l}\,\Erw_{i}\zeta_{1,l}^{2}}{(1-\eps)^{2}x_{l}^{2}}\ +\ 2m_{l}\,\Prob_{i}(S_{\tau}>x_{l})\ \le\ \frac{4m_{l}\,\Erw_{i}\zeta_{1,l}^{2}}{(1-\eps)^{2}x_{l}^{2}}.
\end{align*}
for all $l\ge 1$. By \cite[Lemma 3.2]{KesMal:96}, $\Erw_{i}\zeta_{1,l}^{2}\le 3x_{l}A_{i}(x_{l})$ for all sufficiently large $l$ which in combination with \eqref{eq1:m_delta} provides us with
$$ \Prob_{i}(S_{\tau_{2m_{l}}}^{*}>x_{l})\ \le\ \frac{12m_{l}A_{i}(x_{l})}{(1-\eps)^{2}x_{l}}\ \le\ \frac{6\eps}{(1-\eps)^{2}} $$
for all such $l$. We have thus shown that
$$ \Prob_{i}(S_{\tau_{m_{l}}}^{*}>x_{l})\ \le\ \frac{6\eps}{\theta(1-\eps)^{2}}\,+\,\frac{\delta}{2} $$
for any $\eps\in (0,1)$ and sufficiently large $l$ (not depending on $\eps$), say $l\ge l_{0}$. Finally, fix any $l\ge l_{0}$ and choose $\eps$ so small that
$$ \frac{6\eps}{\theta(1-\eps)^{2}}\ <\ \frac{\delta}{2}. $$
Then we arrive at 
$$ \delta\ >\ \Prob_{i}(S_{\tau_{2m_{l}}}^{*}>x_{l})\ \le\ \Prob_{i}(S_{m_{l}}^{*}>x_{l}) $$
which contradicts our definition of $m_{l}=m_{\delta}(x_{l})$.\qed
\end{proof}

\begin{proof}[of Theorem \ref{thm:Kesten-Maller MRW finite S},``\eqref{eq: sigma alpha cond}$\,\RA\Erw_{i}J_{i}(S_{\tau(i)}^-)^{1+\alpha}<\infty$'']
Suppose \eqref{eq: sigma alpha cond} is true and recall that $\cS^{>}$ denotes the set of recurrent states of the ladder chain $(M_{n}^{>})_{n\ge 0}$. Positive divergence ensures $\cS^{>}\ne\emptyset$. Pick any $i\in\cS^{>}$ and put $\kappa:=\inf\{n\ge 0: S_{n}=H_{1}^{i}\}$, $H_{1}^{i}$ as defined in \eqref{eq:def H^i}, and $\wh{\sigma}^{>}:=\inf\{n: S_{\kappa+n}-S_{\kappa}>0\}$. Using $|\cS^{>}|<\infty$, we then obtain
\begin{align*}
\Erw_{i}\osg(0)^{1+\alpha}\ &\le\ \Erw_{i}\left( \sigma^>\,\1_{\{H_{1}^{i}=0\}}+ \kappa\,\1_{\{H^{i}>0\}}+\sum_{j\in\cS^{>}}\1_{\{M_{\kappa}=j,\,H_{1}^{i}>0\}}\,\wh{\sigma}^{>}\right)^{1+\alpha}\\	
&\le\ (|\cS^{>}|+2)^\alpha\,\left(\Erw_{i}(\sg)^{1+\alpha}+\Erw_{i}\tau(i)^{1+\alpha}+ \sum_{j\in\cS^{>}}\Erw_j(\sigma^>)^{1+\alpha}\right)\ <\ \infty,
\end{align*}
which is equivalent to \eqref{eq: sigma alpha cond} by Proposition \ref{prop: MRW sigma Prop}.\qed
\end{proof}

\section{Asymptotic behavior of $S_{n}/n$}\label{sec:SLLN}

\subsection{Strong law of large numbers}\label{subsec:SLLN}

It is well-known that an ordinary random walk $(S_{n})_{n\ge 0}$ satisfies the strong law of large numbers (SLLN), viz. $n^{-1}S_{n}\to\mu$ a.s.\ for some $\mu\in\R$, iff $X_{1}$ is integrable and $\mu=\Erw X_{1}$. The natural substitute for the latter condition in the case of a MRW $(M_{n},S_{n})_{n\ge 0}$ is that $X_{1}$ is $\Prob_{\pi}$-integrable and $\Erw_{\pi}X_{1}=\mu$. However, this is only sufficient but not necessary for the SLLN to hold as shown by the next theorem and a subsequent example. For $n\ge 1$, put
$$ S_{n}^{\oplus}\,:=\,\sum_{k=1}^{n}X_{k}^{+}\quad\text{and}\quad S_{n}^{\ominus}\,:=\,\sum_{k=1}^{n}X_{k}^{-} $$
which are clearly again MRW with driving chain $(M_{n})_{n\ge 0}$.

\begin{Theorem}\label{thm:SLLN}
Given a MRW $(M_{n},S_{n})_{n\ge 0}$, the following assertions are equivalent for any $\mu\in\R$:
\begin{description}[xxx]\itemsep1pt
\item[(a)] $X_{1}$ is $\Prob_{\pi}$-integrable and $\Erw_{\pi}X_{1}=\mu$.\vspace{.05cm}
\item[(b)] $n^{-1}S_{n}\to\mu$ a.s. and $\Erw_{\pi}X_{1}$ exists, i.e. $\Erw_{\pi}X_{1}^{-}<\infty$ or $\Erw_{\pi}X_{1}^{+}<\infty$.
\item[(c)] $n^{-1}S_{n}^{\,\ominus}\to\mu^{-}$, $n^{-1}S_{n}^{\,\oplus}\to\mu^{+}$ a.s. and $\mu^{+}-\mu^{-}=\mu$.\vspace{.05cm}
\item[(d)] $\tau_{n}(i)^{-1}S_{\tau_{n}(i)}^{\,\ominus}\to\mu^{-}$, $\tau_{n}(i)^{-1}S_{\tau_{n}(i)}^{\,\oplus}\to\mu^{+}$ $\Prob_{i}$-a.s. for some/all $i\in\cS$ and $\mu^{+}-\mu^{-}=\mu$.\vspace{.05cm}
\item[(e)] $S_{\tau(i)}^{\,\ominus},S_{\tau(i)}^{\,\oplus}$ are $\Prob_{i}$-integrable and $\Erw_{i}S_{\tau(i)}=\pi_{i}^{-1}\mu$ for some/all $i\in\cS$.
\item[(f)] $\Erw_{\pi}X_{1}$ exists and $\sum_{n\ge 1}n^{-1}\Prob_{i}(|n^{-1}S_{n}-\mu|>\eps)<\infty$ for all $\eps>0$ and some/all $i\in\cS$.
\end{description}
\end{Theorem}

\begin{proof}
``(a)$\RA$(b)'' Since $(X_{n}^{-})_{n\ge 1}$ and $(X_{n}^{+})_{n\ge 1}$ are ergodic stationary sequences under $\Prob_{\pi}$ with finite means $\mu^{-}=\Erw_{\pi}X_{1}^{-}$ and $\mu^{+}=\Erw_{\pi}X_{1}^{+}$, respectively, $n^{-1}S_{n}\to\mu$ a.s. with $\mu=\mu^{+}-\mu^{-}$ follows from Birkhoff's ergodic theorem.

\vspace{.1cm}
``(b)$\RA$(c)'' Suppose that $\mu^{+}:=\Erw_{\pi}X_{1}<\infty$. By using ``(a)$\RA$(b)'' for the MRW $(M_{n},S_{n}^{\oplus})_{n\ge 0}$, we infer $n^{-1}S_{n}^{\oplus}\to\mu^{+}$ a.s. and then further
$$ n^{-1}S_{n}^{\ominus}\ =\ n^{-1}(S_{n}-S_{n}^{\oplus})\ \to\ \mu-\mu^{+}\ =:\ \mu^{-}\quad\Prob_{\pi}\text{-a.s.} $$
Assuming $\Erw_{\pi}X_{1}^{-}<\infty$, we may argue in a similar manner.

\vspace{.1cm}
``(c)$\RA$(d)'' is trivial.

\vspace{.1cm}
``(d)$\RA$(e)'' Since $n^{-1}\tau_{n}(i)\to\Erw_{i}\tau(i)=\pi_{i}^{-1}$ $\Prob_{i}$-a.s., we have that 
$$ n^{-1}S_{\tau_{n}(i)}^{\,\ominus}\to\pi_{i}^{-1}\mu^{-}\quad\text{and}\quad n^{-1}S_{\tau_{n}(i)}^{\,\oplus}\to\pi_{i}^{-1}\mu^{+}\quad\Prob_{i}\text{-a.s.} $$
which implies the $\Prob_{i}$-integrability of $S_{\tau(i)}^{\,\ominus},S_{\tau(i)}^{\,\oplus}$ and
$$ \Erw_{i}S_{\tau(i)}\ =\ \Erw_{i}S_{\tau(i)}^{\,\oplus}-\Erw_{i}S_{\tau(i)}^{\,\ominus}\ =\ \pi_{i}^{-1}(\mu^{+}-\mu^{-})\ =\ \pi_{i}^{-1}\mu. $$

\vspace{.1cm}
``(e)$\RA$(a)'' If $S_{\tau(i)}^{\,\ominus}$ is $\Prob_{i}$-integrable, then $\Erw_{i}S_{\tau(i)}^{\,\ominus}=\pi_{i}^{-1}\Erw_{\pi}X_{1}^{-}$ implies that $X_{1}^{-}$ is $\Prob_{\pi}$-integrable. Similarly, the $\Prob_{i}$-integrability of $S_{\tau(i)}^{\,\oplus}$ implies the $\Prob_{\pi}$-integrability of $X_{1}^{+}$. Hence, we may use \eqref{eq:stat mean drift} to obtain
$$ \Erw_{\pi}X_{1}\,\Erw_{i}\tau(i)\ =\ \Erw_{i}S_{\tau(i)}\ =\ \mu\,\pi_{i}^{-1} $$
and thus $\Erw_{\pi}X_{1}=\mu$.

\vspace{.1cm}
``(b)$\RA$(f)'' If $n^{-1}S_{n}\to\mu$ a.s., then, for any $\eps>0$, $(S_{n}-(\mu-\eps)n)_{n\ge 0}$ is positive divergent and $(S_{n}-(\mu+\eps)n)_{n\ge 0}$ negative divergent, hence
$$ \sum_{n\ge 1}\frac{1}{n}\Prob_{i}(S_{n}\le (\mu-\eps)n)\ <\ \infty\quad\text{and}\quad\sum_{n\ge 1}\frac{1}{n}\Prob_{i}(S_{n}\ge (\mu+\eps)n)\ <\ \infty $$
for all $i\in\cS$ by Theorem \ref{thm:Spitzer-Erickson MRW}(c), which in turn yields the asserted finiteness of the series $\sum_{n\ge 1}n^{-1}\Prob_{i}(|n^{-1}S_{n}-\mu|\ge\eps)$.

\vspace{.1cm}
``(f)$\RA$(b)'' If $\Erw_{\pi}X_{1}$ exists, w.l.o.g. positive, then $n^{-1}S_{n}\to\Erw_{\pi}X_{1}>0$ a.s. Moreover, $\sum_{n\ge 1}n^{-1}\Prob(|n^{-1}S_{n}-\mu|\ge\eps )<\infty$ for all $\eps>0$ in combination with a simple Borel-Cantelli argument entails the positive divergence of $(S_{n}-n(\mu+\eps))_{n\ge 0}$ and the negative divergence of $(S_{n}-n(\mu-\eps))_{n\ge 0}$ for any such $\eps$. Consequently, by another appeal to Birkhoff's ergodic theorem, we infer
$$ \Erw_{\pi}X_{1}-(\mu+\eps)\ =\ \lim_{n\to\infty}\frac{S_{n}-n(\mu-\eps)}{n}\ \ge\ 0\quad\text{a.s.} $$
for all $\eps>0$, thus $\Erw_{\pi}X_{1}\ge\mu$, and a similar argument for $(S_{n}-n(\mu-\eps))_{n\ge 0}$ shows $\Erw_{\pi}X_{1}\le\mu$. Hence, $n^{-1}S_{n}\to\mu$ a.s. as claimed.\qed
\end{proof}

It is readily seen that the above equivalences (a)-(e) remain valid if $\mu=\Erw_{\pi}X_{1}\in\{\pm\infty\}$. On the other hand, in contrast to ordinary random walks with iid increments, one cannot dispense with the finiteness of $\Erw_{\pi}X_{1}^{-}$ or $\Erw_{\pi}X_{1}^{+}$ in condition (b) as shown by the following example, where $n^{-1}S_{n}$ converges a.s. to 0 although $\Erw_{\pi}X_{1}^{-}=\Erw_{\pi}X_{1}^{+}=\infty$. 

\begin{Exa}\label{exa:SLLN}\rm
Let $(M_{n})_{n\ge 0}$ be a birth-death chain on $\N_{0}$ with transition probabilities 
$$ p_{i,i-1}\ =\ 1-p_{i,i+1}\ =\ \frac{i+2}{2(i+1)}, $$
thus $p_{01}=1$. The stationary distribution is given by
$$ \pi_{i}\ =\ \frac{p_{01}\cdot...\cdot p_{i-1,i}}{p_{10}\cdot...\cdot p_{i,i-1}}\pi_{0}\ \asymp\ \frac{1}{i^{2}} $$
for all $i\in\N_{0}$ and $\pi_{0}$ such that $\sum_{i\ge 0}\pi_{i}=1$. Define $\gamma:\N_{0}\to\R$ by $\gamma(0)=0$ and
\begin{equation*}
\gamma(2i)\ =\ -\gamma(2i-1)\,:=\,i
\end{equation*}
for $i\ge 1$. Then, since $p_{2i-1\,2i}\asymp \frac{1}{2}$ as $i\to\infty$, we obtain
\begin{align*}
\Erw_{\pi}X_{1}^{+}\ &\ge \ \sum_{i\ge 1}\pi_{2i-1}\cdot p_{2i-1\, 2i}\cdot \Erw(X_1|M_0=2i-1, M_1=2i)\ \asymp \ \sum_{i\ge 1} \frac{1} {(2i)^2}\cdot \frac{1}{2}\cdot  2i\ =\ \infty
\end{align*}
and $\Erw_{\pi} X_{1}^{-}=\infty$ follows analogously, whereas each $S_{\tau(i)}$ is $\Prob_{i}$-a.s. vanishing and thus particularly integrable. Put $N_n:=\inf\{k: \tau_k(0)\ge n\}$. Since
$$ |S_{\tau_{n}(0)+k}|\ \le\ |X_{\tau_{n}(0)+k}|\ \le\ k\quad\Prob_{0}\text{-a.s.} $$
for all $n\in\N_{0}$ and $\tau_{n}(0)\le k<\tau_{n+1}(0)$ and $\Erw_{0}\tau(0)<\infty$, we infer that
$$ \left|\frac{S_{n}}{n}\right|\ \le\ \left|\frac{X_{n}}{n}\right|\ \le\ \frac{\nu_{N_{n}}}{n}\ \stackrel{n\to\infty}{\longrightarrow}\ 0\quad\Prob_{0}\text{-a.s.} $$
and then the same convergence a.s.

Now choosing any integrable sequence $(Y_{n})_{n\ge 1}$ of nondegenerate iid\ random variables independent of $(M_{n})_{n\ge 0}$ and putting
$$ X_{n}\,:=\,Y_{n}+\gamma(M_{n})-\gamma(M_{n-1}) $$
for $n\ge 1$, the associated MRW $(M_{n},S_{n})_{n\ge 0}$ is easily seen to be regular with $n^{-1}S_{n}\to\mu=:\Erw Y_{1}$ a.s., and
$$ \Erw_{i}S_{\tau(i)}\ =\ \Erw_{i}\left(\sum_{k=1}^{\tau(i)}Y_{k}\right)\ =\ \mu\,\Erw_{i}\tau(i)\ <\ \infty $$
for all $i\in\N_{0}$, although $\Erw_{\pi}X_{1}$ does not exist.
\end{Exa}

\subsection{Proof of Theorem \ref{thm:Kesten trichotomy MRW} (Kesten trichotomy)}\label{subsec:Kesten trichotomy}

Since $(X_{n})_{n\ge 0}$ forms an ergodic stationary sequence under $\Prob_{\pi}$, 
$$ \liminf_{n\to\infty}\,n^{-1}\,S_{n}\quad\text{and}\quad\limsup_{n\to\infty}\,n^{-1}\,S_{n} $$ are both a.s. constant. Therefore, it suffices to prove the trichotomy under $\Prob_{i}$ for any fixed $i\in\cS$. Note that $\Erw_{i}|S_{\tau(i)}|=\infty$ rules out that the given MRW is null-homologous.

\vspace{.1cm}
If $(S_{n})_{n\ge 0}$ is positive divergent, then \eqref{eq:Pruitt1} and \eqref{eq:Pruitt2} provide us with
\begin{align*}
\liminf_{n\to\infty}\,\frac{S_{n}}{n}\ \ge\ \liminf_{n\to\infty}\,\frac{S_{\tau_{n}}-D_{n+1}^{i}}{\tau_{n+1}}\ =\ \liminf_{n\to\infty}\, \frac{S_{\tau_{n}}^{+}\,(1-o(1))}{\tau_{n+1}}\ =\ \infty\quad\text{a.s.}
\end{align*}
and thus the desired result \textsf{(PD+)}. In the negative divergent case, the same conclusion holds for $(-S_{n})_{n\ge 0}$, thus giving \textsf{(ND+)}.

\vspace{.1cm}
If $(S_{n})_{n\ge 0}$ is oscillating, then $\Erw_{i}|S_{\tau(i)}|=\infty$ entails at least one of
$$ \liminf_{n\to\infty}n^{-1}S_{\tau_{n}}\ =\ -\infty\quad\text{or}\quad\limsup_{n\to\infty}n^{-1}S_{\tau_{n}}\ =\ \infty\quad\text{a.s.} $$
by Kesten's trichotomy for ordinary random walks, and thus also
$$ \liminf_{n\to\infty}n^{-1}S_{n}\ =\ -\infty\quad\text{or}\quad\limsup_{n\to\infty}n^{-1}S_{n}\ =\ \infty\quad\text{a.s.} $$
W.l.o.g. assuming the second alternative, let $c:=\liminf_{n\to\infty}n^{-1}S_{n}\le 0$ be finite. Then $\liminf_{n\to\infty}\,n^{-1}(S_{n}+n\,(c+1))=1>0$ which in turn entails the positive divergence of $(S_{n}+n(c+1))_{n\ge 0}$ and so, by the first part of this proof,
$$ \infty\ =\ \liminf_{n\to\infty}\,n^{-1}(S_{n}+ n\,(c+1))\ =\ \liminf_{n\to\infty}\,n^{-1}S_{n}+c+1\quad\text{a.s.} $$
which contradicts the finiteness of $c$.\qed

\section{Counterexamples}\label{sec:counterexamples}

\subsection{Theorem \ref{thm:Spitzer-Erickson MRW}: $\Erw_{i}\tau(i)\log\tau(i)<\infty$ is necessary for the equivalence of part (c) and \eqref{eq:Spitzer series criterion}}

Our first counterexample will show that one cannot dispense with $\Erw_{i}\tau(i)\log\tau(i)<\infty$ for the equivalence of part (c) and the integral criterion \eqref{eq:Spitzer series criterion} in Theorem \ref{thm:Spitzer-Erickson MRW}, not even when $(M_{n},S_{n})_{n\ge 0}$ is positive divergent with finite stationary drift.

\begin{Exa}\label{exa:xlogx for tau(s)}\rm
The infinite petal flower chain from \ref{exa:infinite petal flower chain} may be generalized by having excursions of variable length away from the central state 0. To be more precise, suppose that $(M_{n})_{n\ge 0}$ has state space $\cS\subset\{0\}\cup\N^{2}$ and satisfies
\begin{align*}
p_{0,(n,1)}\,=\,\Prob(\Gamma=n)\quad\text{and}\quad p_{(n,k),(n,k+1)}\,=\,1\,=\,p_{(n,n-1),0}
\end{align*}
for all $n\ge 2$ and $k=1,\ldots,n-2$, where $\Gamma$ denotes an $\N$-valued random variable with finite mean. In other words, whenever in state 0, the chain picks a state $(n,n-1)$ with probability $\Prob(\Gamma=n)$ and then moves deterministically through $(n,n-2),\ldots,(n,1)$ before returning to 0, hence $\Prob_{0}(\tau(0)\in\cdot)=\Prob(\Gamma\in\cdot)$.

\vspace{.1cm}
Let $\Gamma$ further be such that $\Erw\Gamma\log\Gamma=\infty$. Define the increments of $(M_{n},S_{n})_{n\ge 0}$ by
\begin{align*}
X_{n}\ :=\ 
\begin{cases}
-(k-1),&\text{if }M_{n}=(k,1)\text{ for }k\in\N,\\
\hfill k,&\text{if }M_{n-1}=(k,k-1),M_{n}=0\text{ for }k\in\N,\\
\hfill 0,&\text{otherwise}.
\end{cases}
\end{align*}
Again, we then have $S_{\tau(0)}=1$, in particular positive divergence of $(S_{\tau_{n}(0)})_{n\ge 0}$ and also $\Erw_{\pi}|X_{1}|<\infty$, for
$$ \Erw_{\pi}|X_{1}|\ =\ \frac{1}{\Erw_{0}\tau(0)}\sum_{n\ge 1}(2n-1)\,\Prob_{0}(\tau(0)=n)\ =\ \frac{2\,\Erw_{0}\tau(0)-1}{\Erw_{0}\tau(0)}. $$
Hence, by Proposition \ref{prop:regular <-> mean exists}, $(M_{n},S_{n})_{n\ge 0}$ is regular and therefore also positive divergent, in particular Theorem \ref{thm:Spitzer-Erickson MRW}(c) holds true. On the other hand, using $J_{0}(x)\asymp x$ as $x\to\infty$, we find
\begin{align*}
\int\log J_{0}(x)\ \V_{0}(dx)\ &\asymp\ \int \log x\ \V_{0}(dx)\ =\ \Erw_{0}\left(\sum_{n=1}^{\tau(0)}\log S_{n}^{-}\right)\\
&=\ \Erw_{0}(\tau(0)-1)\log(\tau(0)-1)\ =\ \Erw(\Gamma-1)\log(\Gamma-1)\ =\ \infty.
\end{align*}
Here we have used that \eqref{eq:occ measure formula} and $S_{1}^{-}=\ldots=S_{\tau(0)-1}^{-}=\tau(0)-1$ $\Prob_{0}$-a.s.
\end{Exa}

\subsection{The integral criteria \ref{thm:Kesten-Maller MRW}(a), \eqref{eq:MRW Spitzer series} and \eqref{eq:MRW renewal counting}}

Let $\alpha>0$ and $(M_{n},S_{n})_{n\ge 0}$ be such that $S_{\tau_{n}(i)}\to\infty$ a.s. and $\Erw_{i}J_{i}(S_{\tau(i)}^{-})^{1+\alpha}<\infty$. The following example shows that the integral criteria
\begin{description}\itemsep2pt
\item[(1)]   $\Erw_{i}J_{i}(D^{i})^{1+\alpha}<\infty$ (equivalent to $\Erw_{i}\rho(0)^{\alpha}<\infty$),
\item[(2)]  $\int  J_{i}(y)^\alpha\ \V_{i}^1(dy)<\infty$ (equivalent to  $\sum_{n\ge 1} n^{\alpha-1}\,\Prob_{i}(S_{n}\le 0) <\infty$),
\item[(3)]  $ \int  J_{i}(y)\ \V_{i}^{\alpha}(dy)<\infty$ (equivalent to $\Erw_{i}N(0)^{\alpha}<\infty$ for $\alpha\ge 1$)
\end{description}
are generally not equivalent.

\begin{Exa}\label{exa:integral criteria}\rm
Let $(M_{n},S_{n})_{n\ge 0}$ be a MRW with infinite petal flower driving chain from Example \ref{exa:infinite petal flower chain} and
\begin{align*}
X_{n}\ :=\ 
\begin{cases}
\hfill -x_{i},&\text{if }M_{n-1}=0,\,M_{n}=i,\\
x_{i}+2,&\text{if }M_{n-1}=i,\,M_{n}=0,
\end{cases}
\quad n\ge 1,
\end{align*}
for a sequence of positive numbers $(x_{i})_{i\ge 1}$. Then $\tau(0)=2$ and thus, by \eqref{eq:tail D^s and V_alpha},
$$ \V_{0}^{\alpha}((y,\infty))\ \asymp\ \Prob_{0}(D^{0}>y) $$
for any $\alpha>0$. Moreover, $S_{\tau(0)}=2$, thus $J_{0}(x)\asymp x$, and $D_{1}^{0}=x_{M_{1}}$ $\Prob_{0}$-a.s. Therefore, (1)--(3) for $i=0$ may here be restated as
\begin{description}\itemsep2pt
\item[(1)]   $\Erw_{0}(D^{0})^{1+\alpha}=\sum_{i\ge 1}p_{0i}\,x_{i}^{1+\alpha}<\infty$,
\item[(2)]  $\int y^{\alpha}\ \V_{0}^{1}(dy)=\Erw_{0}(D^{0})^{\alpha}=\sum_{i\ge 1}p_{0i}\,x_{i}^{\alpha}<\infty$,
\item[(3)]  $\int  y\ \V_{0}^{1}(dy)=\Erw_{0}D^{0}=\sum_{i\ge 1}p_{0i}\,x_{i}<\infty$,
\end{description}
respectively, and it is now obvious that, by appropriate choice of the $x_{i}$, any of the three cases ``only (3) holds'', ``(2) and (3) hold, but (1) fails'', ``(1)--(3) all fail to hold'' may occur while at the same time $\Erw_{0}J_{0}(S_{\tau}^{-})^{\beta}<\infty$ is obviously satisfied for all $\beta\in\R_{>}$.
\end{Exa}

\subsection{Tail comparison: $\Prob_{\pi}(X_{1}\in\cdot)$ versus $\Prob_{i}(S_{\tau(i)}\in\cdot)$}

We show next that the tails of $\Prob_{\pi}(X_{1}\in\cdot)$ versus $\Prob_{i}(S_{\tau(i)}\in\cdot)$ can be very different.

\begin{Exa}\label{exa:Sisyphus}\rm
Let $\alpha\in (1,2)$ and $(M_{n})_{n\ge 0}$ be the Sisyphus chain on $\N_{0}$ as introduced in Remark \ref{rem:Sisyphus chain}. This chain has stationary probabilities
\begin{align*}
\pi_{n}\ =\ c\,\Erw_{0}\left(\sum_{k=1}^{\tau(0)}\1_{\{M_{k}=n\}}\right)\ =\ c\,\Prob_{0}(\tau(0)>n)\ =\ \frac{c}{n^{\alpha}}
\end{align*}
with $c=1/\Erw_{0}\tau(0)$.
Define $X_{n}=M_{n}$, thus $S_{n}=M_{1}+...+M_{n}$ for $n\ge 1$. Moreover, $S_{\tau(0)}=(\tau(0)-1)\tau(0)/2$ $\Prob_{0}$-a.s. and therefore
$$ \Prob_{0}(S_{\tau(0)}>n)\ =\ \Prob((\tau(0)-1)\tau(0)>2n)\ \asymp\ \Prob(\tau(0)>n^{1/2})\ \asymp\ \frac{1}{n^{\alpha/2}} $$
as $n\to\infty$. On the other hand,
$$ \Prob_{\pi}(X_{1}>n)\ =\ \Prob_{\pi}(M_{1}>n)\ =\ \sum_{k>n}\pi_{k}\ \asymp\ \frac{1}{n^{\alpha-1}}. $$ 
As a consequence, $\Erw_{0}(S_{\tau(0)}\wedge x)=\int_{0}^{x}\Prob(S_{\tau(0)}>y)\ dy $ grows like $x^{(2-\alpha)/2}$, while $\Erw_{\pi}(X_{1}\wedge x)$ grows like $x^{2-\alpha}$, as $x\to\infty$.
\end{Exa}

\subsection{Theorem \ref{thm:Kesten-Maller MRW min}: Part (c) does not imply part (b)}\label{subsec:min theorem}

Our last example will provide an instance where $\Erw_{i}|S_{\sle(-x)}|^{\alpha}\1_{\{\sle(-x)<\infty\}}<\infty$ for all $(i,x)\in\cS\times\R_{\geqslant}$, while $\Erw_{i}\left|\min_{n\ge 0}S_{n}\right|^{\alpha}=\infty$.

\begin{Exa}\label{exa:general IPFC example}\rm
Given any $\alpha>1$, let $(M_{n})_{n\ge 0}$ be the generalized infinite petal flower chain from Example \ref{exa:xlogx for tau(s)}, but with $\Gamma$ satisfying the moment conditions
$$ \Erw\Gamma^{2(1+1/\alpha)}\ <\ \infty\ =\ \Erw\Gamma^{(1+\alpha)(1+1/\alpha)}\quad\text{and}\quad\Erw\Gamma^{1+\alpha}\ <\ \infty. $$
Define the increments of $(M_{n},S_{n})_{n\ge 0}$ by
\begin{align*}
X_{n}\ :=\ 
\begin{cases}
\hfill -l^{1/\alpha},&\text{if }M_{n}=(k,l)\text{ for }k,l\in\N,\\
1+\sum_{l=1}^{k-1}l^{1/\alpha},&\text{if }M_{n-1}=(k,k-1),M_{n}=0\text{ for }k\in\N,
\end{cases}
\end{align*}
hence $S_{\tau(0)}=1$ $\Prob_{0}$-a.s., in particular $J_{0}(x)\asymp x$ as $x\to\infty$. Observe that 
$$ D^{0}\ =\ \sum_{k=1}^{\tau(0)-1}k^{1/\alpha}\ \asymp\ \tau(0)^{1+1/\alpha}\quad\Prob_{0}\text{-a.s.} $$
which, by construction, entails
$$ \Erw_{0}D^{0}\ \le\ \Erw_{0}(D^{0})^{2}\ \asymp\ \Erw_{0}\tau(0)^{2(1+1/\alpha)}\ <\ \infty\ =\ \Erw_{0}\tau(0)^{(1+\alpha)(1+1/\alpha)}\ \asymp\ \Erw_{0}(D^{0})^{1+\alpha}. $$
and thereupon the positive divergence of $(M_{n},S_{n})_{n\ge 0}$ and $\Erw_{i}\left|\min_{n\ge 0}S_{n}\right|^{\alpha}=\infty$ for all $i\in\N_{0}$ by invoking Theorem \ref{thm:Kesten-Maller MRW min}.

\vspace{.1cm}
Finally, we prove $\Erw_{0}|S_{\sle(-x)}|^{\alpha}\1_{\{\sle(-x)<\infty\}}<\infty$ for all $x\in\R_{\geqslant}$ which easily implies that $\Erw_{i}|S_{\sle(-x)}|^{\alpha}\1_{\{\sle(-x)<\infty\}}<\infty$ for all $(i,x)\in\N_{0}\times\R_{\geqslant}$. Define
$$ \kappa(x)\ :=\ \inf\{n\ge 1:\tau_{n}(0)\ge\sle(-x)\} $$
and notice that $\kappa(x)\le\sle(-x)$ as well as
$$ |S_{\sle(-x)}|\ \le\ \tau_{\kappa(x)}(0)^{1/\alpha}\quad\Prob_{0}\text{-a.s. on }\{\sle(-x)<\infty\}. $$
Consequently, by making use of Wald's identity,
\begin{align*}
\Erw_{0}|S_{\sle(-x)}|^{\alpha}\1_{\{\sle(-x)<\infty\}}\ &\le\ \Erw_{0}\tau_{\kappa(x)}(0)\1_{\{\kappa(x)<\infty\}}\\
&=\ \Erw_{0}\tau(0)\,\Erw_{0}\kappa(x)\1_{\{\kappa(x)<\infty\}}\ \lesssim\ \Erw_{0}\sle(-x)\1_{\{\sle(-x)<\infty\}}.
\end{align*}
The last expectation is finite by an appeal to Theorem \ref{thm:Kesten-Maller MRW}.
\end{Exa}

\section{Comparison with perturbed random walks}\label{sec:PRW}

There is partial overlap of the present work with a recent article by the first author with Iksanov and Meiners \cite{AlsIksMei:15} on fluctuation theory for \emph{perturbed random walks (PRW)}, defined by $(\sum_{k=1}^{n-1}Z_{k}+\eta_{n})_{n\ge 1}$ for an iid $\R^{2}$-valued sequence $(Z_{n},\eta_{n})_{n\ge 1}$.
Although MRW and PRW are quite different stochastic sequences, in some regards and under additional assumptions, their study reduces to similar objects. For example, positive divergence of a MRW $(M_{n},S_{n})_{n\ge 0}$ discussed in this paper is equivalent to the positive divergence of the PRW $(S_{\tau_{n-1}(i)}-D_{n}^{i})_{n\ge 1}$ for some/all $i\in\cS$. Moreover, if $\Erw_{i}\tau(i)^{1+\alpha}<\infty$, then Lemma \ref{lem:T>tau_n formula} implies that $\Erw_{i}\rho(i)^{\alpha}<\infty$ holds iff the $\alpha$-moment of the last exit time of $(S_{\tau_{n-1}(i)}-D_{n}^{i})_{n\ge 1}$ is finite. This indicates that one can translate
results for a MRW in terms of a suitable PRW and then draw on the fluctuation theory for the latter class of sequences developed in \cite{AlsIksMei:15}. The translation actually also goes the other way when assuming the $\eta_{n}$ to be integer-valued which is possible without loss of generality for only the tails of their distribution matters here. On the other hand, this correspondence has its limitations. There are in fact equivalences with no counterpart for PRW where we have used the particular structure of a MRW, notably its dual $({}^{\#}M_{n},{}^{\#}S_{n})_{n\ge 0}$ and the ladder chain $(M_{n}^{>})_{n\ge 0}$. Theorem \ref{thm:Kesten-Maller MRW min} on the power moments of
$$ \left|\min_{n\ge 0}S_{n}\right|\ =\ \left|\min_{n\ge 1}\,(S_{\tau_{n-1}(i)}-D_{n}^{i})\right|\quad\Prob_{i}\text{-a.s.} $$
may serve as another example where we have benefitted from the use of MRW and which has no counterpart in \cite{AlsIksMei:15}. Last but not least, the study of power moments of $N(x)$ and of the weighted renewal measures $\sum_{n\ge 1}n^{\alpha-1}\,\Prob_{i}(S_{n}\le x)$ provide instances where PRW cannot be used at all because the behavior of these quantities depends on the entire excursions of the $S_{n}$ between visits of the driving chain to a state $i$ rather than just their minima $D_{n}^{i}$.

\section{Appendix}

\begin{Lemma}\label{Appendix: asymp}
Let $(S_n)_{n\ge 0}$ be a positive divergent random walk. Then, as $y\to\infty$,
$$ \sum_{n\ge 1}\Prob(0\le S_{n}\le y)\ \asymp \ J(y). $$
\end{Lemma}

\begin{proof}
If $\Ug$ denotes the renewal measure of the strictly ascending ladder height process $(S_{\sgn})_{n\ge 0}$, then it is easily verified that
$$ \Ug([0,y])-1\ \le\ \sum_{n\ge 1}\Prob(0\le S_{n}\le y)\ \le\ \Erw\sg\Ug([0,y]) $$
for all $y\in\R_{\geqslant}$. Now use \eqref{eq:harmonic series asymptotics, alpha>0} with $\alpha=1$ and $\Erw(S_{\sg}\wedge y)\asymp\Erw(S_{1}^{+}\wedge y)$ (cf. \cite[Eq. (4.5)]{KesMal:96}) to infer
$$ \Ug([0,y])\ =\ \sum_{n\ge 0}\Prob(S_{\sgn}\le y)\ \asymp\ \frac{y}{\Erw(S_{\sg}\wedge y)}\ \asymp\ \frac{y}{\Erw(S_{1}^{+}\wedge y)}\ =\ J(y) $$
as $y\to\infty$.\qed
\end{proof}

\begin{Lemma}\label{lem:series tau_n(s)}
Let $\alpha\in\R_{\geqslant}$ and $(M_{n},S_{n})_{n\ge 0}$ be a nontrivial MRW such that
$\Erw_{i}\tau(i)^{1+\alpha}<\infty$ for some/all $i\in\cS$. Then
\begin{equation}\label{eq2:Spitzer condition tau_n}
\sum_{n\ge 1}\Erw_{i}\tau_{n}(i)^{\alpha-1}\,\1_{\{\tau_{n}(i)>bn\}}\ <\ \infty,
\end{equation}
for all $b>\Erw_{i}\tau(i)$ and $i\in\cS$.
\end{Lemma}

\begin{proof}
We start by pointing out that, for any $b>\Erw_{i}\tau(i)$, $(\tau_{n}(i)-bn)_{n\ge 0}$ is negative divergent which in combination with $\Erw_{i}\tau(i)^{1+\alpha}<\infty$ implies
\begin{equation}\label{eq:Spitzer condition tau_n}
\sum_{n\ge 1}n^{\alpha-1}\,\Prob_{i}(\tau_{n}(i)>bn)\ <\ \infty
\end{equation}
by Theorem \ref{thm:Spitzer-Erickson}(c) $(\alpha=0)$ or Theorem \ref{thm:Kesten-Maller}(e) $(\alpha>0)$. Since $\tau_{n}(i)\ge n$ for all $n\in\N$, this further implies \eqref{eq2:Spitzer condition tau_n} if $0\le\alpha\le 1$.

\vspace{.1cm}
Left with the case $\alpha>1$, we find
\begin{align*}
\sum_{n\ge 1}&\Erw_{i}\tau_{n}(i)^{\alpha-1}\,\1_{\{\tau_{n}(i)>bn\}}\\
&\asymp\ \sum_{n\ge 1}\int_{0}^{\infty}x^{\alpha-2}\,\Prob_{i}(\tau_{n}(i)>bn\vee x)\ dx\\
&\asymp\ \sum_{n\ge 1}n^{\alpha-1}\Prob_{i}(\tau_{n}(i)>bn)\ +\ \sum_{n\ge 1}\sum_{k\ge n}\int_{bk}^{b(k+1)}x^{\alpha-2}\,\Prob_{i}(\tau_{n}(i)>x)\ dx
\end{align*}
so that, in view of \eqref{eq:Spitzer condition tau_n}, we must still verify finiteness of the second term. We infer
\begin{align*}
\sum_{n\ge 1}&\sum_{k\ge n}\int_{bk}^{b(k+1)}x^{\alpha-2}\,\Prob_{i}(\tau_{n}(i)>x)\ dx\ \asymp\ \sum_{n\ge 1}\sum_{k\ge n}k^{\alpha-2}\,\Prob_{i}(\tau_{n}(i)>bk)\\
&=\ \sum_{k\ge 1}\sum_{n=1}^{k}k^{\alpha-2}\,\Prob_{i}(\tau_{n}(i)>bk)\ \le\ \sum_{k\ge 1}k^{\alpha-1}\,\Prob_{i}(\tau_{k}(i)>bk),
\end{align*}
and the last sum is again finite by \eqref{eq:Spitzer condition tau_n}.\qed
\end{proof}

\begin{Lemma}\label{lem:solidarity increments}
Let $\alpha\in\R_{\geqslant}$ and $(M_{n},S_{n})_{n\ge 0}$ be a nontrivial MRW such that
$\Erw_{i}\tau(i)^{1+\alpha}<\infty$ for some/all $i\in\cS$. Then
$$ \sum_{n\ge 1}n^{\alpha-1}\Prob_{i}(S_{\tau_{n}(i)}\le y)\ \asymp\ \Erw_{i}\left(\sum_{n\ge 1}\tau_{n}(i)^{\alpha-1}\1_{\{S_{\tau_{n}(i)}\le y\}}\right) $$
for all $i\in\cS$.
\end{Lemma}

\begin{proof}
Using the previous lemma and \eqref{eq:Spitzer condition tau_n}, we obtain
\begin{align*}
\Erw_{i}&\left(\sum_{n\ge 1}\tau_{n}(i)^{\alpha-1}\1_{\{S_{\tau_{n}(i)}\le y\}}\right)\\
&\asymp\ \Erw_{i}\left(\sum_{n\ge 1}\tau_{n}(i)^{\alpha-1}\1_{\{S_{\tau_{n}(i)}\le y,\,n\le\tau_{n}(i)\le 2n\,\Erw_{i}\tau(i)\}}\right)\\
&\asymp\ \Erw_{i}\left(\sum_{n\ge 1}n^{\alpha-1}\1_{\{S_{\tau_{n}(i)}\le y,\,n\le\tau_{n}(i)\le 2n\,\Erw_{i}\tau(i)\}}\right)\\
&\asymp\ \sum_{n\ge 1}n^{\alpha-1}\,\Prob_{i}(S_{\tau_{n}(i)}\le y)
\end{align*}
as claimed.\qed
\end{proof}

\section*{Glossary}

\begin{tabular}[c]{p{0.15\textwidth} p{0.85\textwidth}}
 	$\N_{0}$		& set of nonnegative integers $\{0,1,2,\ldots\}$\\[.3mm]
	$\R_{>}$		& positive halfline $(0,\infty)$\\[.3mm]
	$\R_{\geqslant}$	& nonnegative halfline $[0,\infty)$\\[.3mm]
	$\cS$	&countable state space of the driving chain\\[.3mm]
	$f(x)\lesssim g(x)$	&shorthand for $\limsup_{x\to\infty}\frac{f(x)}{g(x)}<\infty$\\[1.5mm]
	$f(x)\gtrsim g(x)$	&shorthand for $\liminf_{x\to\infty}\frac{f(x)}{g(x)}>0$\\[1.5mm]
	$f(x)\asymp g(x)$	&shorthand for $f(x)\lesssim g(x)$ and $f(x)\gtrsim g(x)$\\[1.5mm]
	$A\lesssim(\gtrsim) B$	&shorthand for $A\le(\ge) cB$ for some $c\in\R_{>}$\\[1mm]
	$A\asymp B$	&shorthand for $A\lesssim B$ and $A\gtrsim B$, thus $c^{-1}B\le A\le cB$ for some $c\in\R_{>}$\\[1mm]
	$(M_{n})_{n\ge 0}$ &positive recurrent driving chain with transition matrix $(p_{ij})_{i,j\in\cS}$ and stationary distribution $\pi=(\pi_{i})_{i\in\cS}$\\[.3mm]
	$({}^{\#}M_{n})_{n\ge 0}$ &dual of $(M_{n})_{n\ge 0}$ with transition matrix $(\pi_{j}p_{ji}/\pi_{i})_{i,j\in\cS}$\\[.3mm]
	$\tau(i),\tau_{n}(i)$	& first, $n^{th}$ return epoch of the driving chain to state $i$, thus $\tau_{1}(i)=\tau(i)$\\[.3mm]
	${}^{\#}\tau(i),{}^{\#}\tau_{n}(i)$	&same as $\tau(i),\tau_{n}(i)$ for the dual chain $({}^{\#}M_{n})_{n\ge 0}$\\[.4mm]
	$\upsilon=\upsilon(i,j)$	&$=\inf\{n\ge 1: \tau_n(i)>\tau(j)\}=$ first return to state $i$ after a visit to $j$.\\[.3mm]
	$(S_{n})_{n\ge 0}$ &(additive part of the) MRW with driving chain $(M_{n})_{n\ge 0}$\\[.3mm]
	$X_{n}$	&$=S_{n}-S_{n-1}=$ $n^{th}$ increment of the MRW\\[.3mm]
	$({}^{\#}S_{n})_{n\ge 0}$ &(additive part of the) dual MRW with driving chain $({}^{\#}M_{n})_{n\ge 0}$\\[.3mm]
	$\sg$	& $=\inf\{n\ge 1:S_{n}>0\}=$ first strictly ascending ladder epoch of the MRW\\[.3mm]
	$\sle$	&$=\inf\{n\ge 1:S_{n}\le 0\}=$ first weakly descending ladder epoch of the MRW\\[.3mm]
	$(\sgn)_{n\ge 1}$ &sequence of strictly ascending ladder epochs, thus $\sg_{1}=\sg$\\[.3mm]
		$(\slen)_{n\ge 1}$ &sequence of weakly descending ladder epochs, thus $\sle_{1}=\sle$
\end{tabular}

\begin{tabular}[c]{p{0.15\textwidth} p{0.85\textwidth}}
	$\sg(x)$	&$=\inf\{n\ge 1:S_{n}>x\}=$ level $x$ first passage time, thus $\sg(0)=\sg$\\[.3mm]
	$\osg(x)$	&$=\inf\{n>\tau(M_{0}):S_{n}>x\}=$ level $x$ first passage time after first return to initial state of the driving chain\\[.3mm]
	$\rho(x)$	&$=\sup\{n\ge 0:S_{n}\le x\}=$ level $x$ last exit time\\[.3mm]
	$\sigma_{\rm min}$ &$=\inf\{n\ge 1:S_{n}=\min_{k\ge 1}S_{k}\}=$ hitting of minimum epoch\\[.3mm]
	$N(x)$	&$=\sum_{n\ge 1}\1_{\{S_{n}\le x\}}=$ renewal counting process of the MRW\\[1.2mm]
	$\nu(i,x)$	&$=\inf\{n\ge 1:S_{\tau_{n}(i)}>x\}=$ level $x$ first passage time of $(S_{\tau_{n}(i)})_{n\ge 0}$\\[.5mm]
	$\zeta_{n}(i)$	&$=\inf\{k>\zeta_{n-1}(i):S_{\tau_{k}(i)}>S_{\tau_{\zeta_{n-1}(i)}}\}$ for $n\ge 1$, where $\zeta_{0}(i):=0$\\[.4mm]
	$\tg_{n}(i)$	&$=\tau_{\zeta_{n}(i)}$ for $n\ge 1$, and $\tg(i)$ also used for $\tg_{1}(i)=\nu(i,0)$\\[.4mm]
	${}^{\#}\tg_{n}(i)$	&same as $\tg_{n}(i)$ for the dual MRW $({}^{\#}M_{n},{}^{\#}S_{n})_{n\ge 0}$\\[.4mm]
	$\nu^{>}(i,x)$	&$=\inf\{n\ge 1:S_{\tg_{n}(i)}>x\}=$ level x first passage time of $(S_{\tg_{n}(i)})_{n\ge 0}$\\[.4mm]
	$A_{i}(x)$	&$=\Erw_{i}(S_{\tau(i)}^{+}\wedge x)-\Erw_{i}(S_{\tau(i)}^{-}\wedge x)$ for $x>0$\\[.4mm]
	$J_{i}(x)$	&$=\begin{cases}\frac{x}{\Erw(S_{\tau(i)}^{+}\wedge x)},&\text{if }\Prob_{i}(S_{\tau(i)}>0)>0\\[-1.5mm]
\hfill x&\text{otherwise}\\
\end{cases}$\quad for $x>0$, and $J_{i}(0)=1$\\[.4mm]
	$J_{i,\gamma}(x)$ &same as $J_{i}(x)$, but with $[\Erw_{i}(S_{\tau(i)}^{+}\wedge x)]^{\gamma}$ in the denominator, $\gamma\in [0,1]$\\[.4mm]
	$J_{i}^{>}(x)$ &same as $J_{i}(x)$, but with $\Erw_{i}(S_{\tg(i)}\wedge x)$ in the denominator, $\gamma\in [0,1]$\\[.4mm]
	$D_{n}^{i}$	&$=\max_{\tau_{n-1}(i)<k\le\tau_{n}(i)}(S_{k}-S_{\tau_{n-1}(i)})^{-}$ for $n\in\N$,\\[.3mm]
	&$=$ maximal downward excursion of the MRW in the discrete stochastic time interval $\{\tau_{n-1}(i),\ldots,\tau_{n}(i)\}$, iid under $\Prob_{i}$ with generic copy $D^{i}$\\[.5mm]
	$D_{n}^{i,>}$	&$=\max_{\tg_{n-1}(i)<k\le\tg_{n}(i)}(S_{k}-S_{\tg_{n-1}(i)})^{-}$ for $n\in\N$,\\[1.5mm]
	$\V_{i}^{\alpha}((x,\infty))$ &$=\,\Erw_{i}\big(\sum_{n=1}^{\tau(i)}\1_{\{S_{n}^{-}>x\}}\big)^{\alpha}$, $x\in\R_{\geqslant}$, $\alpha>0$\\[1.5mm]
	$\V_{i}$	&$=\V_{1}^{1}$\\[1mm]
	$\Sigma_{\alpha}(i,x)$	&$=\sum_{n\ge 1}n^{\alpha-1}\,\Prob_{i}(S_{n}\le x)$, $x\in\R_{\geqslant}$, $\alpha\ge 0$
\end{tabular}

\section*{Acknowledgment}

We are indebted to an anonymous referee for many constructive comments that helped to improve the presentation of this article.

\bibliographystyle{abbrv}
\bibliography{StoPro}

\end{document}